\documentclass[a4paper,11pt]{article}
\usepackage{MyStyle}
\usepackage{eatrees}
\usepackage{mathabx}

\usepackage{hyperref}

\let\OLDthebibliography\thebibliography
\renewcommand\thebibliography[1]{
  \OLDthebibliography{#1}
  \setlength{\parskip}{4.5pt}
  \setlength{\itemsep}{0pt plus 0.3ex}
}

\usepackage{algorithm}
\floatname{algorithm}{Method}
\usepackage{algorithmic}

\allowdisplaybreaks

\newtheorem{theorem}{Theorem}[section]
\newtheorem{definition}[theorem]{Definition}
\newtheorem*{definition*}{Definition}

\newtheorem{proposition}[theorem]{Proposition}
\newtheorem{lemma}[theorem]{Lemma}
\newtheorem*{lemma*}{Lemma}
\newtheorem{remark}[theorem]{Remark}
\newtheorem*{remark*}{Remark}

\newtheorem*{remarks*}{Remarks}

\newtheorem{ass}[theorem]{Assumption}
\newtheorem*{notation*}{Notation}

\newtheorem*{ex*}{Example}

\newtheorem*{exs*}{Examples}

\newtheorem*{app*}{Application}
\newtheorem{conjecture*}{Conjecture}

\def\ts{\thinspace}
\newcommand{\diam}{\blackdiamond}

\title{
Order conditions for sampling the invariant measure of ergodic stochastic differential equations on manifolds
}

\author{
Adrien Laurent\textsuperscript{1} and Gilles Vilmart\textsuperscript{1}
}

\begin{document}
\footnotetext[1]{
Universit\'e de Gen\`eve, Section de math\'ematiques, UNI DUFOUR, 24 rue du Général Dufour, CP 64, 1211 Geneva 4, Switzerland. Adrien.Laurent@unige.ch, Gilles.Vilmart@unige.ch.}

\maketitle

\begin{abstract}

We derive a new methodology for the construction of high order integrators for sampling the invariant measure of ergodic stochastic differential equations with dynamics constrained on a manifold.
We obtain the order conditions for sampling the invariant measure for a class of Runge-Kutta methods applied to the constrained overdamped Langevin equation.
The analysis is valid for arbitrarily high order and relies on an extension of the exotic aromatic Butcher-series formalism.
To illustrate the methodology, a method of order two is introduced, and numerical experiments on the sphere, the torus and the special linear group confirm the theoretical findings.

\smallskip

 \end{abstract}


\section{Introduction}
\setcounter{footnote}{1}

We consider systems of stochastic differential equations (SDEs) in~$\R^d$ subject to a smooth scalar constraint and a Stratonovich noise of the form
\begin{equation}
\label{equation:projected_SDE_general}
dX(t)=\Pi_\MM(X(t)) f(X(t))dt +\Pi_\MM(X(t)) \Sigma(X(t)) \circ dW(t),\quad X(0)=X_0\in \MM,
\end{equation}
where~$\Pi_\MM\colon\R^d\to\R^{d\times d}$ is the orthogonal projection on the tangent bundle of the manifold~$\MM=\{x\in\R^d,\zeta(x)=0\}$ of codimension~$q$,~$\zeta\colon\R^d\to\R^q$ is a given constraint,~$f\colon\R^d\to\R^d$ is a smooth drift,~$\Sigma\colon\R^d\to\R^{d\times d}$ is a smooth diffusion coefficient and~$W$ is a standard~$d$-dimensional Brownian motion in~$\R^d$ on a probability space equipped with a filtration and fulfilling the usual assumptions.
For simplicity of the analysis, we assume that~$\MM$ is a compact smooth manifold of codimension~$q=1$.
The smoothness and compactness of~$\MM$ guaranty in particular the existence and uniqueness of a solution to~\eqref{equation:projected_SDE_general} with bounded moments for all times~$t>0$.\footnote{The extension to manifolds of any codimension~$q\geq 1$ will be further discussed in Remarks~\ref{remark:weakened_assumptions} and~\ref{remark:extension_RK_methods}.}
In addition, thanks to the projection operator~$\Pi_\MM$, the solution~$X(t)$ lies on~$\MM$ for all~$t>0$.
In the additive noise case where~$\Sigma(x)=\sigma I_d$ with~$\sigma>0$, equation~\eqref{equation:projected_SDE_general} can also be rewritten equivalently with a Lagrange multiplier (see~\cite[Sect.\ts 3.2.4.1]{Lelievre10fec} or~\cite[Sect.\ts 3.3]{Lelievre12ldw}) as
\begin{equation*}
dX(t)=f(X(t))dt +\sigma dW(t)+g(X(t)) d\lambda_t, \quad \zeta(X(t))=0,\quad X(0)=X_0\in \MM,
\end{equation*}
where~$g=\nabla \zeta$ and~$\lambda$ is an adapted stochastic process determined by the equation~$\zeta(X)=0$.

A major motivation of model~\eqref{equation:projected_SDE_general} appears in computational problems in molecular dynamics with the constrained overdamped Langevin equation (obtained in the particular case where~$\Sigma(x)=\sigma I_d$ is a constant homothety),
\begin{equation}
\label{equation:projected_Langevin}
dX(t)=\Pi_\MM(X(t)) f(X(t))dt +\sigma \Pi_\MM(X(t)) \circ dW(t),\quad X(0)=X_0\in \MM,
\end{equation}
with~$\sigma>0$,~$f=-\nabla V$ and~$V\colon\R^d\to\R$ is a smooth potential.
The overdamped Langevin equation is widely used to model the motion of a set of particles subject to a potential~$V$ in a high friction regime.
The possible constraints can be induced for example by strong covalent bonds between atoms, or fixed angles in molecules.
Sampling from the constrained overdamped Langevin equation allows to compute the so-called free energy, which is a key quantity in thermodynamic (see, for instance,~\cite{Ciccotti08pod,Lelievre10fec,Lelievre12ldw} and references therein).
Equations of the form~\eqref{equation:projected_SDE_general} appear naturally when studying conservative SDEs, that is, SDEs possessing an invariant~$H$ conserved almost surely by all realisations of~\eqref{equation:projected_SDE_general}. The solution of conservative SDEs are subject to the constraint~$\zeta(X)=0$ with~$\zeta(x)=H(x)-H(X_0)$.
Drawing samples on a manifold also has many applications in statistics (see~\cite{Brubaker12afo,Diaconis13sfa} and references therein).

Under regularity conditions on the generator of the SDE and on~$\MM$, it was shown in~\cite{Ciccotti08pod,Faou09csd} that the solution~$X(t)$ of the SDE~\eqref{equation:projected_SDE_general} is ergodic, that is, there exists a unique invariant measure~$d\mu_\infty$ on~$\MM$ that has a density~$\rho_\infty$ with respect to~$d\sigma_\MM$, the canonical measure on~$\MM$ induced by the Euclidean metric of~$\R^d$, such that for all test functions~$\phi\in \CC^\infty(\R^d,\R)$,
\begin{equation}
\label{equation:ergodic_theorem}
\lim_{T\to\infty}\frac{1}{T}\int_0^T \phi(X(t)) dt= \int_{\MM} \phi(x) d\mu_\infty(x)\quad \text{almost surely}.
\end{equation}
In the case of the overdamped Langevin equation~\eqref{equation:projected_Langevin} on~$\MM$, the process is naturally ergodic and the invariant measure is given by~$d\mu_\infty=\rho_\infty d\sigma_\MM=\frac{1}{Z}\exp\left(-\frac{2}{\sigma^2}V\right)d\sigma_\MM$ with~$Z=\int_{\MM}\exp\left(-\frac{2}{\sigma^2}V\right) d\sigma_\MM$.
Approximating the quantity~$\int_{\MM} \phi(x) d\mu_\infty(x)$ is a computational challenge when the dimension~$d$ is high, which is the case in the context of molecular dynamics where the dimension is proportional to the number of particles, because a standard quadrature formula becomes prohibitively expensive.
We emphasize that~$\mu_\infty$ is singular with respect to the Lebesgue measure on~$\R^d$.
In addition, the integrator samples should remain on the manifold~$\MM$.
Hence, the order conditions for sampling the invariant measure in the Euclidean context of~$\R^d$ do not generalize straightforwardly to the manifold case.
The main goal of this article is to build and analyse high order one-step integrators for approximating~$\int_{\MM} \phi(x) d\mu_\infty(x)$ that lie on the manifold~$\MM$ and that have the form
\begin{equation}
\label{equation:integrator}
X_{n+1}=\Phi(X_n,h,\xi_n),
\end{equation}
where the~$\xi_n$ are standard independent random vectors and~$h$ is the numerical step.

There are different ways to approximate the solution of the SDE problem~\eqref{equation:projected_SDE_general}.
A strong approximation focuses on approaching the realisation of a single trajectory of~\eqref{equation:projected_SDE_general} for a given realisation of the Wiener process~$W$.
A weak approximation approaches the average of functionals of the solution.
We focus here on the approximation for the invariant measure, that is approaching averages of functionals of the solution in the stationary state. This convergence is the numerical equivalent of~\eqref{equation:ergodic_theorem}.
The integrator~\eqref{equation:integrator} is said to have order~$p$ for the invariant measure if for all~$\phi\in \CC^\infty(\R^d,\R)$, there exists a positive constant~$C(\phi)$ independent of the initial condition~$X_0$ such that
\begin{equation}
\label{equation:def_error_inv_measure}
e(\phi,h)\leq C(\phi)h^p\quad \text{where} \quad e(\phi,h)=\bigg|\lim_{N\to \infty} \frac{1}{N+1}\sum_{n=0}^N \phi(X_n)-\int_\MM \phi d\mu_\infty\bigg|.
\end{equation}
We recall that a scheme of weak order~$r$ immediately has order~$p\geq r$ for the invariant measure.
For the underdamped and overdamped Langevin dynamics in~$\R^d$, the articles~\cite{BouRabee10lra,Leimkuhler13rco,Abdulle14hon,Abdulle15lta,Leimkuhler16tco} proposed multiple schemes of high order for the invariant measure with low weak order (typically~$r=1$).
We mention in particular the work~\cite{Abdulle14hon} that introduced a methodology for the analysis and design of high order integrators for the invariant measure. This methodology, that relies on Talay-Tubaro expansions~\cite{Talay90eot}, backward error analysis and modified differential equations for SDEs~\cite{Zygalakis11ote,Abdulle12hwo,Debussche12wbe,Kopec15wbea,Kopec15wbeb}, is generalised in the context of manifolds in the present paper.

A widely used and simple numerical scheme for sampling the invariant measure distribution on manifolds is the Euler scheme (see~\cite{Ciccotti05bms,Lelievre08adc,Lelievre10fec,Lelievre12ldw} for instance).
Two variants exist for the overdamped Langevin equation~\eqref{equation:projected_Langevin}, both of order one in the weak sense, or for sampling the invariant measure: the Euler integrator with explicit projection direction
\begin{equation}
\label{equation:EE}
X_{n+1}=X_n+hf(X_n) +\sigma \sqrt{h} \xi_n + \lambda g(X_n),
\quad
\zeta(X_{n+1})=0,
\end{equation}
and alternatively the Euler integrator with implicit projection direction
\begin{equation}
\label{equation:IE}
X_{n+1}=X_n+hf(X_n) +\sigma \sqrt{h} \xi_n + \lambda g(X_{n+1}),
\quad
\zeta(X_{n+1})=0.
\end{equation}
To the best of our knowledge, no high order numerical integrators for sampling the invariant measure of the overdamped Langevin equation with constraints~\eqref{equation:projected_Langevin} have been proposed in the literature.
In~\cite{Lelievre19hmc}, an order two discretization based on the RATTLE integrator (see~\cite{Ryckaert77nio,Andersen83rav,Hairer10sod}) is applied to the underdamped Langevin equation, rather than to the overdamped Langevin dynamic~\eqref{equation:projected_Langevin}.
The previously described discretizations can be combined with Metropolis-Hastings rejection procedures~\cite{Metropolis53eos,Hastings70mcs}. We quote in particular the Markov-Chain Monte-Carlo (MCMC) methods~\cite{Girolami11rml,Brubaker12afo,Lelievre12ldw} and the Hybrid Monte-Carlo methods~\cite{Zappa18mco,Lelievre19hmc}, where the need for a reverse projection check is shown to be a key step.
We also mention the integrators in~\cite{Zhang19eso,Lelievre20mpm} that are based on an Euler discretization and present new approaches for projecting on the manifold.
The alternative approach of using Metropolis-Hastings rejection procedure allows to fully remove the bias on the invariant measure. Analogously to the Euclidean case, this procedure does not make high order discretizations obsolete because, in particular, the rejection rate depends on the quality of the discretization and the dimension of the problem in general, and in the case of stiff problems or problems in high dimension, it suffers from timestep restrictions.
Note also that in the specific case where~$\MM$ is a Lie group, high order integrators can be naturally obtained using splitting methods, that are, however, typically limited to weak order two of accuracy (see~\cite{Blanes05otn} for further details in the context of ODEs).

This article proposes new tools for constructing integrators of any high order for sampling the invariant measure of constrained SDEs of the form~\eqref{equation:projected_SDE_general} and relies on the formalism of trees and Butcher-series.
Originally introduced by Hairer and Wanner in~\cite{Hairer74otb}, and based on the work of Butcher~\cite{Butcher72aat}, B-series have proved to be a powerful standard tool for the numerical analysis of deterministic differential equations, as presented, for instance, in the textbooks~\cite{Hairer06gni,Butcher16nmf}.
In the last decades, several works extended B-series to the stochastic context. We mention in particular Burrage and Burrage~\cite{Burrage96hso,Burrage00oco} and Komori, Mitsui and Sugiura~\cite{Komori97rta} who first introduced stochastic trees and B-series for studying the order conditions of strong convergence of SDEs, Rö{\ss}ler~\cite{Rossler04ste,Rossler06rta,Rossler06rkm,Rossler10ste,Rossler10saw} and Debrabant and Kv{\ae}rn{\o}~\cite{Debrabant08bsa,Debrabant10rkm,Debrabant11cos} for the design and analysis of high order weak and strong integrators on a finite time interval,~\cite{Anmarkrud17ocf} for creating schemes preserving quadratic invariants, and~\cite{Kupper12ark}, where tree series were applied to a class of stochastic differential algebraic equations (SDAEs) for the computation of strong order conditions.
Finally we mention the recent work~\cite{Laurent20eab}, that introduced the exotic aromatic B-series for the computation of order conditions for sampling the invariant measure of ergodic SDEs in~$\R^d$, and that we extend in this paper to the context of SDEs on manifolds.

This article is organized as follows.
Section~\ref{section:high_order_ergodic_approximation} is devoted to the analysis of the accuracy of integrators for sampling the invariant measure on a manifold~$\MM$.
In Section~\ref{section:Runge_Kutta_section}, we apply this methodology on a class of Runge-Kutta methods for solving the constrained overdamped Langevin equation~\eqref{equation:projected_Langevin}, to derive arbitrary high order conditions for the invariant measure, with special emphasis on order two conditions, and to introduce a new order two scheme that uses only a few evaluations of~$f$ per step.
The detailed calculations of the order conditions for the invariant measure are done in Section~\ref{section:B_series} with the help of an extension of the exotic aromatic B-series formalism~\cite{Laurent20eab}.
We compare in Section~\ref{section:numerical_experiments} the new order two scheme with the Euler scheme~\eqref{equation:IE} in numerical experiments on a sphere, a torus and the special linear group~$\SL(m)$ to confirm its order of convergence for sampling the invariant measure.

\section{High order ergodic approximation on a manifold}
\label{section:high_order_ergodic_approximation}

In this section, we present a new criterion for building integrators of any order for the invariant measure by extending the~$\R^d$ results in~\cite{Debussche12wbe,Abdulle14hon} to the context of manifolds.
We first settle down a few notations and assumptions, before we recall the standard weak expansions of the exact and numerical solution using the backward Kolmogorov equation.
For~$\zeta\colon \R^d\rightarrow \R$ a smooth map, we denote~$g=\nabla \zeta$ its gradient, and~$G(x)=g^T(x)g(x)=\abs{g(x)}^2$ the Gram function related to the manifold~$\MM=\{x\in\R^d,\zeta(x)=0\}$, where we denote by~$\abs{x}=(x^T x)^{1/2}$ the Euclidean norm in~$\R^d$.
We assume in the rest of the article that~$\MM$ is a compact and smooth manifold of codimension one embedded in~$\R^d$.
We suppose in addition that the Gram function~$G$ is strictly positive on~$\MM$,~$G(x)\geq \alpha>0$ for all~$x\in\MM$.
With these notations, the projection~$\Pi_\MM$ on the tangent bundle is given by~$\Pi_\MM(x)=I-G(x)^{-1}g(x)g(x)^T$.
We denote~$\LL$ the generator of the SDE~\eqref{equation:projected_SDE_general}. It is given, for~$\phi\in\CC^\infty(\R^d,\R)$, by
\begin{equation}
\label{equation:generator_L_general}
\LL\phi
=\phi'(\Pi_\MM f)
+\frac{1}{2}\sum_{i=1}^d \phi'((\Pi_\MM\Sigma e_i)'(\Pi_\MM\Sigma e_i))
+\frac{1}{2}\sum_{i=1}^d \phi''(\Pi_\MM \Sigma e_i,\Pi_\MM \Sigma e_i),
\end{equation}
where~$(e_i)_{i=1, \dots, d}$ is the canonical basis of~$\R^d$ and, for all vectors~$a^1$, \dots,~$a^m\in\R^d$, we use the following notation for differentials in~$\R^d$,
$$\phi^{(m)}(a^1,\dots,a^m)
=\sum_{i_1,\dots,i_m=1}^d \partial_{i_1,\dots,i_m} \phi \ts a^1_{i_1}\dots a^m_{i_m}
=\sum_{i_1,\dots,i_m=1}^d \frac{\partial^m\phi}{\partial x_{i_1}\dots\partial x_{i_m}} \ts a^1_{i_1}\dots a^m_{i_m}.$$
For the overdamped Langevin equation~\eqref{equation:projected_Langevin}, the generator~\eqref{equation:generator_L_general} reduces to
\begin{align}
\label{equation:generator_L_Langevin}
\LL\phi
&=\phi'f
-G^{-1}(g,f)\phi'g
-\frac{\sigma^2}{2}G^{-1}\Div(g)\phi'g
+\frac{\sigma^2}{2}G^{-2}(g,g'g)\phi'g
+\frac{\sigma^2}{2}\Delta \phi\\
&-\frac{\sigma^2}{2}G^{-1}\phi''(g,g)
=\frac{\sigma^2}{2}\exp\Big(\frac{2}{\sigma^2}V\Big)\Div_\MM\Big(\exp\Big(-\frac{2}{\sigma^2}V\Big)\nabla_\MM\phi\Big),\nonumber
\end{align}
where~$\nabla_\MM\psi:=\Pi_\MM \nabla\psi$ and~$\Div_\MM (H):=\Div(H)-G^{-1}(g,H'(g))$.
The adjoint~$\LL^*$ of the generator~\eqref{equation:generator_L_general} in~$L^2(d\sigma_\MM)$ for the SDE~\eqref{equation:projected_SDE_general}, i.e.\ts, the operator that satisfies for all test functions~$\phi$,~$\psi\in \CC^\infty(\R^d,\R)$,
$$
\int_\MM (\LL\phi) \psi d\sigma_\MM=\int_\MM \phi (\LL^*\psi) d\sigma_\MM,
$$
is given by
$$
\LL^*\phi=-\Div_\MM(\phi f)+\frac{1}{2}\sum_{i=1}^d \Div_\MM(\Div_\MM(\phi \Sigma e_i)\Sigma e_i).
$$

\begin{remark}
As~$\LL$ is a self-adjoint operator in~$L^2(d\mu_\infty)$, but not in~$L^2(d\sigma_\MM)$ in general, it could be more natural to perform the analysis in the space~$L^2(d\mu_\infty)$.
However, as we allow the substages of our numerical integrators to explore the open neighbourhood of~$\MM$ in~$\R^d$, we shall work in this paper with differential operators that cannot be rewritten in general with intrinsic derivatives on the manifold~$\MM$.
In addition, performing directly the integration by parts calculations in~$L^2(d\mu_\infty)$ with such operators is not straightforward, and this motivated the choice of~$L^2(d\sigma_\MM)$ for the analysis. A similar choice was done in~\cite{Laurent20eab} in the context of~$\R^d$.
\end{remark}

We follow the framework of~\cite{Faou09csd}.
In particular, we rely on the construction of the local orthogonal coordinates.
In a neighbourhood~$N_\MM$ of the manifold~$\MM$, there exists an atlas of local orthogonal coordinate systems~$(y,z)\in (V\subset\R^{d-1})\times (-\varepsilon,\varepsilon)$ for~$\varepsilon>0$, with respect to local charts~$\psi\colon U\subset N_\MM\rightarrow (V\subset\R^{d-1})\times (-\varepsilon,\varepsilon)$, such that if~$\psi(x)=(y,z)$, then~$z=\zeta(x)$.
We make the following regularity assumption on the generator~$\LL$.
\begin{ass}
\label{assumption:ellipticity}
On an open neighbourhood~$N_\MM$ of~$\MM$ in~$\R^d$, there exists a constant~$C>0$ such that for all~$x \in N_\MM$ and~$(y,z)=\psi(x)$, for all one-form field~$v\colon T\MM\rightarrow\R$ on~$\MM$ of norm one, we have
$$\sum_{i,j=1}^{d-1} \sum_{k=1}^d \widetilde{\Sigma}_{ik}(y,z)\widetilde{\Sigma}_{jk}(y,z) v_i(\widetilde{x}) v_j(\widetilde{x}) \geq C,$$
where~$\widetilde{x}\in\MM$ is such that~$\psi(\widetilde{x})=(y,0)$ and, for~$k=1,\dots,d$,~$(\widetilde{\Sigma}_{ik}(y,z))_i\in \R^{d-1}$ is defined as the restriction of the vector~$(\Pi_\MM(\widetilde{x}) \Sigma_{ik}(x))_i\in \R^d$ to the tangent space~$T_{\widetilde{x}}\MM$ of~$\MM$, rewritten in the local orthogonal coordinate system.
\end{ass}
This assumption is a variant in the manifold case of the uniform ellipticity property of the generator~$\LL$ in the Euclidean context of~$\R^d$.
In addition, Assumption~\ref{assumption:ellipticity} is automatically satisfied for the constrained overdamped Langevin equation~\eqref{equation:projected_Langevin} and yields that the function~$u(x,t)=\E[\phi(X(t))|X(0)=x]$ satisfies the backward Kolmogorov equation (see~\cite{Faou09csd}):
\begin{equation}
\label{equation:Kolmogorov}
\frac{\partial u}{\partial t}(x,t) = \LL u(x,t), \quad u(x,0)=\phi(x),
\quad x\in N_\MM, \quad t>0.
\end{equation}
We refer to~\cite{Kopec15wbea,Kopec15wbeb} for similar results in the context of~$\R^d$.
The backward Kolmogorov equation~\eqref{equation:Kolmogorov} allows us to write the following expansion of~$u(x,h) =\E[\phi(X(h))|X_0=x]$ for~$h$ small enough,
\begin{equation}
\label{equation:dvp_exa}
u(x,h) = \phi(x)+\sum_{j=1}^N \frac{h^j}{j!} \LL^j\phi(x)+ h^{N+1} R_N^h(\phi,x), \quad x\in N_\MM,
\end{equation}
where~$N_\MM$ is an open neighbourhood of~$\MM$ in~$\R^d$ and the remainder satisfies the estimate~$\abs{R_N^h(\phi,x)}\leq C_N(\phi)$ where the constant~$C_N(\phi)$ is independent of~$h$ and~$x$.

We now assume the existence and uniqueness of an invariant measure, as well as an additional regularity property on~$\LL$, in the spirit of~\cite[Hypotheses H1-H2]{Debussche12wbe} in the context of~$\R^d$.
\begin{ass}
\label{assumption:inv_measure_Poisson}
There exists an open neighbourhood~$N_\MM$ of~$\MM$ in~$\R^d$ and a unique positive function~$\rho_\infty\in\CC^\infty(N_\MM,\R)$ satisfying~$\int_\MM \rho_\infty d\sigma_\MM=1$ and~$\LL^*\rho_\infty=0$ on~$N_\MM$.
Moreover, for all~$\phi\in\CC^\infty(N_\MM,\R)$ such that~$\int_\MM \phi d\sigma_\MM=0$, there exists a unique solution~$\rho\in\CC^\infty(N_\MM,\R)$ to the Poisson problem~$\LL^*\rho=\phi$ that satisfies~$\int_\MM \rho d\sigma_\MM=0$.
\end{ass}
The existence and uniqueness of the invariant measure are in particular satisfied for the constrained overdamped Langevin equation~\eqref{equation:projected_Langevin} (see~\cite[Sect.\ts 2.3]{Faou09csd} for further details).
Assumption~\ref{assumption:inv_measure_Poisson} yields the ergodicity of the process~$X(t)$ solution of~\eqref{equation:projected_SDE_general} with the unique invariant measure~$d\mu_\infty=\rho_\infty d\sigma_\MM$ on~$\MM$.
To proceed further, we shall assume that the integrator~\eqref{equation:integrator} is ergodic, that is, there exists a measure~$d\mu_h$ that has a density with respect to~$d\sigma_\MM$ such that
\begin{equation}
\label{equation:ergodicity_num}
\lim_{N\to \infty} \frac{1}{N+1}\sum_{n=0}^N \phi(X_n)=\int_{\MM} \phi d\mu^h\quad \text{almost surely}.
\end{equation}
We refer to~\cite{Talay90sod,Talay96pnm,Mattingly02efs,Talay02shs} in the Euclidian case, and to~\cite{Faou09csd} in the manifold case, and references therein, for further details on the ergodicity of numerical integrators.
In addition, we suppose that~$\E[\phi(X_1)|X_0=x]$, the numerical analog of~$u(x,h)$, can be developed in powers of~$h$ as was done, for instance, in~\cite{Talay90eot,Abdulle14hon} in the context of~$\R^d$.
\begin{ass}
\label{assumption:dvp_num}
For all~$\phi \in \CC^\infty(\R^d ,\R)$, the numerical integrator~\eqref{equation:integrator} has a weak Taylor expansion of the form
\begin{equation}
\label{equation:dvp_num}
\E[\phi(X_1)|X_0=x] = \phi(x) + \sum_{j=1}^N h^j \AA_{j-1}\phi(x)+ h^{N+1}R_N^h(\phi,x),\quad x\in N_\MM,
\end{equation}
for all~$h$ assumed small enough, and where~$N_\MM$ is an open neighbourhood of~$\MM$ in~$\R^d$ and the remainder satisfies~$\abs{R_N^h(\phi,x)}\leq C_N(\phi)$ where the constant~$C_N(\phi)$ is independent of~$h$ and~$x$.
The~$\AA_j$'s,~$j=0,1,2,\ldots$ are linear differential operators with coefficients depending smoothly on~$f$,~$g$ and their (high order) derivatives (and depending on the choice of the integrator).
\end{ass}
Under Assumptions~\ref{assumption:ellipticity} and~\ref{assumption:dvp_num}, by comparing the expansions~\eqref{equation:dvp_num} and~\eqref{equation:dvp_exa}, the integrator has at least weak order~$p$ if~$\AA_{j-1}=\LL^j/j!$ for~$j=1,\dots,p$.
However, as observed already in~$\R^d$, high order for the invariant measure can be achieved in spite of a low weak order. This is the purpose of Theorem~\ref{theorem:operator_conditions_invariant_measure} where we present a new sufficient condition for a scheme to have order~$r$ for the invariant measure. This result, that relies on the powerful tool of backward error analysis for SDEs, is similar to~\cite[Thm.\ts 3.3]{Abdulle14hon} in the context of smooth compact manifolds.
\begin{theorem}
\label{theorem:operator_conditions_invariant_measure}
Under Assumptions~\ref{assumption:ellipticity},~\ref{assumption:inv_measure_Poisson} and~\ref{assumption:dvp_num}, if the numerical scheme is consistent (that is,~$\AA_0=\LL$) and ergodic, and if it satisfies in~$L^2(d\sigma_\MM)$
\begin{equation*}
\AA_j^*\rho_{\infty}=0, \quad j=1,\dots,r-1,
\end{equation*}
then it has order~$r$ for the invariant measure and the numerical error~\eqref{equation:def_error_inv_measure} satisfies, for~$h\rightarrow 0$,
\begin{align*}
e(\phi,h)&=h^r \int_\MM \phi(x) \rho_r(x) d\sigma_\MM(x)+\OO(h^{r+1})\\
&=h^r\int_0^\infty \int_\MM u(x,t)\AA_r^* \rho_\infty(x) d\sigma_\MM(x) dt+\OO(h^{r+1}),
\end{align*}
where~$\rho_r\in\CC^\infty(N_\MM,\R)$ is the unique solution of the Poisson problem~$\LL^* \rho_r=-\AA_r^*\rho_\infty$ in~$N_\MM$ that satisfies~$\int_\MM \rho_r d\sigma_\MM=0$, with~$N_\MM$ an open neighbourhood of~$\MM$ in~$\R^d$.
\end{theorem}
The proof of Theorem~\ref{theorem:operator_conditions_invariant_measure} is detailed in Appendix~\ref{section:proof_order_invariant_measure} for the sake of completeness. The idea is to write an expansion of the error in the spirit of~\cite{Talay90eot}, and to generalize the analysis in~\cite{Debussche12wbe,Abdulle14hon} on~$\T^d$ and in~\cite{Faou09csd} to the context of smooth compact manifolds.

Theorem~\ref{theorem:operator_conditions_invariant_measure} states the result for times~$t\rightarrow \infty$. A bound of the error at finite time~$t_n=nh$ is typically given by the following exponential estimate (see~\cite{Faou09csd,Debussche12wbe})
\begin{equation*}
\bigg|\E[\phi(X_n)]-\int_{\MM} \phi(x) d\mu_\infty(x)\bigg|\leq Ke^{-\mu t_n}+Ch^p,
\end{equation*}
where the constant~$\mu>0$ is in practice the spectral gap of a certain operator that depends on the numerical integrator.
Reducing the error term~$Ke^{-\mu t_n}$ is out of the scope of this paper, though the recent works~\cite{Lelievre13onr,Duncan16vru,Abdulle19act} proposed numerical methods in~$\R^d$ that improve the rate of convergence at infinity, while sometimes also reducing the variance.

\begin{remark}
\label{remark:weakened_assumptions}
One can consider possible generalisations of Theorem~\ref{theorem:operator_conditions_invariant_measure} in the case where~$\MM$ is not compact, or if~$\MM$ is a manifold of any dimension. We refer to~\cite{Abdulle14hon} for the non-compact extension of Theorem~\ref{theorem:operator_conditions_invariant_measure} in the context of~$\R^d$.
\end{remark}

\section{High order integrators for constrained Langevin dynamics}
\label{section:Runge_Kutta_section}

In this section, we propose a new class of Runge-Kutta methods for sampling the invariant measure of equation~\eqref{equation:projected_Langevin}, and present the methodology for deriving the conditions of any order for the invariant measure using Theorem~\ref{theorem:operator_conditions_invariant_measure}.
In particular, we compute exactly the consistency and order two conditions for the invariant measure as they are the most relevant for the applications.

\subsection{Runge-Kutta methods for constrained overdamped Langevin}
\label{section:class_and_consistency}

When discretizing naively equation~\eqref{equation:projected_Langevin}, one cannot ensure in general that the integrator stays on~$\MM$. It is natural to discretize instead the equivalent formulation with Lagrange multipliers
\begin{equation*}
dX=f(X)dt +\sigma dW+g(X)d\lambda_t,\quad \zeta(X)=0,\quad X(0)=X_0\in \MM.
\end{equation*}
The class of numerical schemes we obtain is in the spirit of deterministic Runge-Kutta methods for differential algebraic problems such as the methods SHAKE and RATTLE (see~\cite{Ryckaert77nio,Andersen83rav,Hairer10sod}), introduced in the context of constrained
Hamiltonian dynamics, or the SPARK class of methods for general DAEs (see~\cite{Jay98spf}).
Since evaluating~$f$ is in practical applications the most expensive part of the algorithm compared to evaluating~$g$, we propose high order integrators that are implicit in~$g$ and explicit in~$f$ in the spirit of implicit-explicit (IMEX) integrators (see, e.g.,~\cite{Hairer10sod}), so that there are only a few evaluations of~$f$ per step.
We thus consider the following class of Runge-Kutta integrators
\begin{align}
\label{equation:defRK}
  Y_i&=X_n+h\sum \limits_{j=1}^s a_{ij}f(Y_j) +\sigma \sqrt{h} d_i \xi_n + \lambda_i \sum \limits_{j=1}^s \widehat{a}_{ij} g(Y_j), \quad i = 1, \dots ,s, \nonumber\\
  \zeta(Y_i)&=0 \quad \text{if} \quad \delta_i=1, \quad i = 1, \dots ,s,\\
  X_{n+1}&=Y_s, \nonumber
\end{align}
where~$A=(a_{ij}), \widehat{A}=(\widehat{a}_{ij})\in \R^{s\times s}$ and~$\delta_i=\sum_{j=1}^s \widehat{a}_{ij}\in \{0,1\}$ are the given Runge-Kutta coefficients, and the~$\xi_n\sim\NN(0,I_d)$ are independent standard Gaussian random vectors in~$\R^d$ (an alternative with discrete bounded random variables is discussed in Remark~\ref{remark:discrete_rv}). We fix~$\delta_s=1$ so that~$X_{n+1}\in\MM$ and we ask that if~$\delta_i=0$, then~$\widehat{a}_{ij}=0$ for~$j=1,\dots,s$ (internal stages without projection,~$Y_i\notin \MM$ a.s.).
Ideally, one aims for IMEX integrators with a low number of evaluations of~$f$, we hence assume in addition that~$\widehat{A}$ is a lower triangular matrix and~$A$ is a strictly lower triangular matrix (in the spirit of DIRK methods).
We represent the numerical integrators with their associated Butcher tableau, where~$b=(a_{s,i})_i$,~$\widehat{b}=(\widehat{a}_{s,i})_i$,~$c= A\ind$ and~$\ind=(1,\dots,1)^T$.
$$\begin{array}
{c|c||c|c||c}
c & A & \delta & \widehat{A} & d\\
\hline
 & b^T &  & \widehat{b}^T & 
\end{array}$$
For instance, the Euler schemes can be written as Runge-Kutta methods of the form~\eqref{equation:defRK} with~$s=2$ and the following Butcher tableaux.
$$
\text{Euler }\eqref{equation:EE}:
\begin{array}
{c|cc||c|cc||c}
0 & 0 & 0 & 0 & 0 & 0 & 0\\
1 & 1 & 0 & 1 & 1 & 0 & 1\\
\hline
 & 1 & 0 &  & 1 & 0 & 
\end{array}
\qquad
\text{Euler }\eqref{equation:IE}:
\begin{array}
{c|cc||c|cc||c}
0 & 0 & 0 & 0 & 0 & 0 & 0\\
1 & 1 & 0 & 1 & 0 & 1 & 1\\
\hline
 & 1 & 0 &  & 0 & 1 & 
\end{array}
$$
Note that the class of methods~\eqref{equation:defRK} satisfies automatically Assumption~\ref{assumption:dvp_num}.

\begin{remark}
\label{remark:extension_RK_methods}
The class of Runge-Kutta methods~\eqref{equation:defRK} can be straightforwardly generalized (as done in~\cite{Laurent20eab} in the Euclidean case~$\R^d$) to study partitioned problems where~$f=f_1+f_2$ and, for example, to create IMEX schemes.
In order to improve the order of the method without increasing its cost, one could also apply a postprocessor (in the spirit of~\cite{Vilmart15pif} in~$\R^d$) or use multiple independent noises in~\eqref{equation:defRK} instead of only one random variable~$\xi_n\sim\NN(0,I_d)$.
This last extension can increase the number of conditions but may also increase the set of solutions. We refer in particular to~\cite{Debrabant10rkm,Laurent20eab} in the context of~$\R^d$, where it is shown for a class of stochastic Runge-Kutta method that the order conditions for weak order 3 cannot be satisfied in general, unless we use at least two independent noises.
In addition, if we rewrite the internal stages of~\eqref{equation:defRK} as
$$Y_i=X_n+h\sum \limits_{j=1}^s a_{ij}f(Y_j) +\sigma \sqrt{h} d_i \xi_n + \bigg(\sum \limits_{j=1}^s \widehat{a}_{ij} g(Y_j)\bigg)\lambda_i,$$
where~$g:\R^d\to \R^{d\times q}$ and~$\lambda_i\in\R^q$, then the same class of methods is also fit for solving~\eqref{equation:projected_Langevin} with a multidimensional constraint~$\zeta:\R^d\to\R^q$.
Note that the coefficients of the method do not depend on the dimension of the space~$d$ or the codimension~$q$ of the manifold.
This will be studied in future work.
\end{remark}

\begin{remark}
\label{remark:discrete_rv}
If~$\xi_n$ is a Gaussian random variable, its realisations can be arbitrarily large, and the existence and uniqueness of the solution of the system~\eqref{equation:defRK} does not hold in general.
A standard remedy to ensure that the projection on~$\MM$ always exists for~$h\leq h_0$ small enough is to replace the standard Gaussian random vectors~$\xi$ in~\eqref{equation:defRK} by bounded discrete random vectors~$\widehat{\xi}$ that have the same first moments in the spirit of~\cite[Chap.\ts 2]{Milstein04snf}.
This way, the order of the method is preserved both in the weak sense and for the invariant measure, and the method is well-posed for all~$h$ small enough.
For weak/ergodic order two, one can consider, for instance, the random vectors~$\widehat{\xi}$ with independent components~$\widehat{\xi}_i$ that satisfy
\begin{equation}
\label{equation:discrete_rv}
\P(\widehat{\xi}_i=0)=\frac{2}{3} \quad \text{and} \quad \P(\widehat{\xi}_i=\pm\sqrt{3})=\frac{1}{6}, \quad i=1,\dots,d.
\end{equation}
\end{remark}

The following lemma guarantees the well-posedness of a method of the form~\eqref{equation:defRK} with bounded random variables~$\widehat{\xi}_n$. The result is still true when~$A$ and~$\widehat{A}$ are general matrices, but we consider only the lower triangular case for the sake of brevity. This result is in the spirit of~\cite[Chap.\ts VII]{Hairer10sod} for deterministic DAEs.

\begin{lemma}
\label{lemma:existence_fixed_point}
For Runge-Kutta methods of the form~\eqref{equation:defRK} where the~$\xi_n$ are replaced by bounded random variables~$\widehat{\xi}_n$, there exists~$h_0>0$ such that for all~$h\leq h_0$, for any initial condition~$X_n\in\MM$, there exists a unique solution~$X_{n+1}$ of~\eqref{equation:defRK} in a neighbourhood of~$X_n$.
Furthermore, the internal stages satisfy~$Y_i=X_n+\OO(\sqrt{h})$ and~$\lambda_i=\OO(\sqrt{h})$ for~$i=1,\dots,s$.
\end{lemma}

\begin{proof}
We proceed by induction on~$i$. We assume that for~$j<i$, the~$Y_j$ are already defined and satisfy~$Y_j=X_n+\OO(\sqrt{h})$. The result is straightforward if~$\delta_i=0$. We thus assume that~$\delta_i=1$ and prove the existence of a unique solution to the equations of the internal stage~$i$:
\begin{align}
\label{equation:proof_existence_RK_1}
  Y_i&=X_n+h\sum_{j=1}^{i-1} a_{ij}f(Y_j) +\sigma \sqrt{h} d_i \widehat{\xi}_n + \lambda_i \sum_{j=1}^i \widehat{a}_{ij} g(Y_j),\\
\label{equation:proof_existence_RK_2}
  \zeta(Y_i)&=0.
\end{align}
Using~$\zeta(X_n)=0$, we rewrite equation~\eqref{equation:proof_existence_RK_2} as
\begin{equation}
\label{equation:proof_existence_RK_3}
\zeta(Y_i)-\zeta(X_n)=\int_0^1 g^T(X_n+\tau(Y_i-X_n))d\tau(Y_i-X_n)=0.
\end{equation}
Inserting~\eqref{equation:proof_existence_RK_1} in~\eqref{equation:proof_existence_RK_3} yields
\begin{equation}
\label{equation:proof_existence_RK_4}
\int_0^1 g^T(X_n+\tau(Y_i-X_n))d\tau\Big[h\sum_{j=1}^{i-1} a_{ij}f(Y_j) +\sigma \sqrt{h} d_i \widehat{\xi}_n + \lambda_i \sum_{j=1}^i \widehat{a}_{ij} g(Y_j)\Big]=0.
\end{equation}
Multiplying both sides of equation~\eqref{equation:proof_existence_RK_1} by~$\int_0^1 g^T(X_n+\tau(Y_i-X_n))d\tau \Big(\sum_{j=1}^i \widehat{a}_{ij} g(Y_j)\Big)$, and substituting~$\lambda_i$ in~\eqref{equation:proof_existence_RK_1} with its value from~\eqref{equation:proof_existence_RK_4}, we deduce that~$F(Y_i,h)=0$, where the function~$F:\R^d\times\R\rightarrow\R^d$ is given by
\begin{align*}
F(y,t)&=\int_0^1 g^T(X_n+\tau(y-X_n))d\tau
\bigg[
\Big(t\sum_{j=1}^{i-1} a_{ij}f(Y_j) +\sigma \sqrt{t} d_i \widehat{\xi}_n\Big)\Big(\sum_{j=1}^{i-1} \widehat{a}_{ij} g(Y_j)+\widehat{a}_{ii} g(y)\Big)\\
&+\Big(\sum_{j=1}^{i-1} \widehat{a}_{ij} g(Y_j)+\widehat{a}_{ii} g(y)\Big) \Big(y-X_n-t\sum_{j=1}^{i-1} a_{ij}f(Y_j) -\sigma \sqrt{t} d_i \widehat{\xi}_n\Big)
\bigg].
\end{align*}
As~$F(X_n,0)=0$ and the partial differential~$\partial_y F(X_n,0)=G(X_n) I_d$ is invertible, the implicit function theorem yields the existence and uniqueness of~$Y_i$ in a ball of center~$X_n$ for~$h\leq h_0$ small enough. As~$\widehat{\xi}_n$ is bounded and~$\MM$ is compact, there exists a deterministic~$h_0$ that works for every initial condition~$X_n\in\MM$.
Now that~$Y_i$ is well-posed, we deduce from the identity~$F(Y_i,h)=0$ that~$Y_i=X_n+\OO(\sqrt{h})$ and we derive from~\eqref{equation:proof_existence_RK_4} that~$\lambda_i$ is well-posed for~$h$ small enough and satisfies~$\lambda_i=\OO(\sqrt{h})$. Finally we observe that~$(Y_i,\lambda_i)$ is indeed a solution to~\eqref{equation:proof_existence_RK_1}-\eqref{equation:proof_existence_RK_2}.
\end{proof}

\begin{remark}
\label{remark:Newton_method}
In practice, one can solve numerically each internal stage of the set of equations~\eqref{equation:defRK} with a fixed point iterations or a Newton method starting from~$Y_i=X_n$ and~$\lambda_i=0$. As~$\MM$ is compact, if the~$\xi_n$ are replaced by bounded random variables, these two methods converge for~$h\leq h_0$ where~$h_0$ is small enough and independent of the initial condition. It is crucial to initialize the~$Y_i$ in a neighbourhood of~$X_n$ as~\eqref{equation:defRK} has multiple solutions in general. For example, the Euler scheme~\eqref{equation:IE} always has two solutions if~$\MM$ is a sphere (the two intersections of~$\MM$ and a straight line going through the center of~$\MM$).
\end{remark}

Before looking at the consistency and the order conditions of the class of methods~\eqref{equation:defRK}, we introduce a concise notation for multiplying vectors component-wise.
\begin{definition}
\label{definition:diamond}
For~$y, y^{(1)}, \dots, y^{(n)}\in\R^d$ and~$m\geq 0$, we define the diamond product and the diamond power as the vectors in~$\R^d$,
$$y^{(1)}\diam \dots \diam y^{(n)}=\bigg(\prod_{k=1}^n y^{(k)}_i\bigg)_i \quad \text{and} \quad y^{\diam m}=(y_i^m)_i.$$
\end{definition}

We present below the detailed calculation of the consistency conditions of the class of methods~\eqref{equation:defRK} for the constrained overdamped Langevin equation~\eqref{equation:projected_Langevin}.
Similar proofs can be found in~\cite[Prop.\ts 3.24]{Lelievre10fec} for the Euler schemes~\eqref{equation:EE}-\eqref{equation:IE}, and in~\cite{Abdulle14hon,Laurent20eab} for Runge-Kutta methods in~$\R^d$.
\begin{proposition}
\label{proposition:consistency}
For a Runge-Kutta method of the form~\eqref{equation:defRK}, the operator~$\AA_0$ in~\eqref{equation:dvp_num} is given for~$\phi\in \CC^\infty(\R^d,\R)$ by
\begin{align*}
\AA_0\phi&=
b^T\ind \phi'f
- b^T\ind G^{-1}(g,f)\phi'g
-\frac{\sigma^2}{2} d_s^2 G^{-1}\Div(g)\phi'g
+\frac{\sigma^2}{2} d_s^2 \Delta\phi
-\frac{\sigma^2}{2} d_s^2 G^{-1} \phi''(g,g)\\
&+\sigma^2 d_s\left(\widehat{b}^T d-\widehat{b}^T(\delta \diam d)+\frac{1}{2}d_s\right) G^{-2} (g,g'g)\phi'g
+\sigma^2 d_s \left(\widehat{b}^T(\delta \diam d)-\widehat{b}^T d \right) G^{-1} \phi'g'g.
\end{align*}
In particular, if 
\begin{equation}
\label{equation:consistency_conditions}
b^T\ind=d_s=1 \quad \text{and} \quad \widehat{b}^T d=\widehat{b}^T(\delta \diam d),
\end{equation}
then the method is consistent, that is,~$\AA_0=\LL$.
\end{proposition}

\begin{proof}
If we apply one step of a method of the form~\eqref{equation:defRK} with the initial condition~$X_0=x$, then the internal stages~$Y_i$ satisfy the following expansion
\begin{align*}
Y_i&=x+\sigma \sqrt{h}d_i \xi +hc_i f(x) +R^h,\quad \text{if} \quad \delta_i=0,\\
Y_i&=x+\sqrt{h}\left[\sigma d_i \xi+\lambda_{1/2,i}(x) g(x)\right]
+h\Big[c_i f(x)+\lambda_{1,i}(x) g(x)+\sigma\lambda_{1/2,i}(x)\sum_{j=1}^s \widehat{a}_{ij}d_j g'(x)\xi \\&+\lambda_{1/2,i}(x)\sum_{j=1}^s \widehat{a}_{ij} \lambda_{1/2,j}(x)\delta_j g'(x)g(x) \Big]
+R^h,\quad \text{if} \quad \delta_i=1,
\end{align*}
where the remainder satisfies~$\abs{R^h}\leq Ch^{3/2}$, and where we used that~$\lambda_{i}$ can be developed in powers of~$\sqrt{h}$ as~$\lambda_{i}=\sqrt{h}\lambda_{1/2,i}+h\lambda_{1,i}+\dots$ in the spirit of~\cite[Lemma\ts 3.25]{Lelievre10fec}.
If~$\delta_i=1$,~$\zeta(Y_i)$ can also be expanded as
\begin{align*}
\zeta(Y_i)&=\zeta(x)
+\sqrt{h}\left[\sigma d_i (g,\xi)+ \lambda_{1/2,i} G \right]
+h\Big[c_i (g,f)+\lambda_{1,i} G +\lambda_{1/2,i}\sum_{j=1}^s \widehat{a}_{ij} \lambda_{1/2,j}\delta_j (g,g'g) \\&+\sigma\lambda_{1/2,i}\sum_{j=1}^s \widehat{a}_{ij}d_j (g,g'\xi)+\frac{1}{2}\sigma^2 d_i^2 (\xi,g'\xi)+\sigma \lambda_{1/2,i} d_i (g,g'\xi) + \frac{1}{2} \lambda_{1/2,i}^2 (g,g'g) \Big]+\dots
\end{align*}
where we omitted the dependency in~$x$ of~$G$,~$g$,~$g'$ and the~$\lambda_{k/2,j}$'s for brevity.
We have~$\zeta(Y_i)=\zeta(x)=0$ (since~$x\in \MM$), thus by identifying each term of the expansion with zero, we get
\begin{align*}
\lambda_{1/2,i}&=-\sigma \delta_i d_i G^{-1}(g,\xi),\\
\lambda_{1,i}&=- \delta_i c_i G^{-1}(g,f)
+\sigma^2 \delta_i\left(\sum_{j=1}^s \widehat{a}_{ij} d_i d_j+d_i^2\right) G^{-2}(g,\xi)(g,g'\xi)\\
&-\sigma^2 \delta_i\left(\sum_{j=1}^s \widehat{a}_{ij} \delta_j d_i d_j+\frac{1}{2} d_i^2\right) G^{-3}(g,\xi)^2 (g,g'g)
-\frac{\sigma^2}{2} \delta_i d_i^2 G^{-1}(\xi,g'\xi).
\end{align*}
For~$\phi$ a test function, the operator~$\AA_0\phi$ satisfies
$$\E[\phi(X_1)]=\E[\phi(Y_s)]=\phi(x)+h\AA_0\phi(x)+h^2\AA_1\phi(x)+\dots$$
By replacing~$Y_s$ with its expansion in powers of~$h^{1/2}$, and by identifying the first terms, we deduce that
\begin{align*}
\AA_0\phi&=\E\Big[
c_s\phi'f
- c_s G^{-1}(g,f)\phi'g
-\frac{\sigma^2}{2} d_s^2 G^{-1}(\xi,g'\xi)\phi'g
+\sigma^2 d_s (\widehat{b}^T d+d_s) G^{-2}(g,\xi)(g,g'\xi)\phi'g\\
&-\sigma^2 d_s \left(\widehat{b}^T(\delta \diam d)+\frac{1}{2} d_s \right) G^{-3}(g,\xi)^2 (g,g'g)\phi'g
+\frac{\sigma^2}{2} d_s^2 \phi''(\xi,\xi)
-\sigma^2 d_s^2 G^{-1} (g,\xi) \phi''(g,\xi)\\
&+\frac{\sigma^2}{2} d_s^2 G^{-2} (g,\xi)^2 \phi''(g,g)
+\sigma^2 d_s \widehat{b}^T(\delta \diam d) G^{-2} (g,\xi)^2 \phi'g'g
-\sigma^2 d_s \widehat{b}^T d G^{-1} (g,\xi)\phi'g'\xi
\Big],
\end{align*}
where we used that~$\delta_s=1$ and that all the terms containing an odd number of~$\xi$ vanish since odd moments of~$\xi$ are zero. Distributing the expectation on each term and using~$c_s=b^T\ind$ yields the desired expression of~$\AA_0\phi$.
We deduce the consistency conditions~$b^T\ind=d_s=1$ and~$\widehat{b}^T d=\widehat{b}^T(\delta \diam d)$ in order to get~$\AA_0=\LL$.
\end{proof}

\begin{remark}
The analysis presented in Section~\ref{section:class_and_consistency} is conducted for the overdamped Langevin dynamics~\eqref{equation:projected_Langevin}.
It would be interesting to consider extensions with multiplicative noise or a non-gradient vector field~$f$. The calculations would likely become more involved and we may get more order conditions (see, for instance,~\cite[Thm.\ts 3.3]{Abdulle14hon} and~\cite[Remark 5.1 and Sect.\ts 5.5]{Laurent20eab} in the context of~$\R^d$, where many additional terms arise, in particular for the integration by parts calculations).
This will be studied in future work.
\end{remark}

\subsection{Order conditions for the invariant measure on manifolds}

We now derive the methodology for getting the conditions of arbitrary high order for sampling the invariant measure of the constrained overdamped Langevin equation~\eqref{equation:projected_Langevin}.
In particular, the following theorem presents the Runge-Kutta conditions for order two for the invariant measure on~$\MM$.
Note that the number of conditions does not depend on the dimension of the space~$d$.

\begin{theorem}[Runge-Kutta conditions for order two for the invariant measure]
\label{theorem:RK_conditions}
We consider a Runge-Kutta method of the form~\eqref{equation:defRK} and assume the consistency condition~\eqref{equation:consistency_conditions}.
If the method is ergodic and if the following conditions are satisfied, then the integrator has order two for the invariant measure:
$$\begin{array}{l}
\widehat{b}^T d=b^T d,\\
b^T c=b^T (\delta \diam c)=b^T d^{\diam 2}=b^T (\delta \diam d^{\diam 2})=2\widehat{b}^T d-\frac{1}{2},\\
\widehat{b}^T c=\widehat{b}^T (\delta \diam c)=\widehat{b}^T d^{\diam 2}=\widehat{b}^T(\delta \diam d^{\diam 2})=\widehat{b}^T d^{\diam 3}=\widehat{b}^T(\delta \diam d^{\diam 3})=2\widehat{b}^T d-\frac{1}{2},\\
\widehat{b}^T (c \diam d)=\widehat{b}^T (\delta \diam c \diam d),\\
b^T (d \diam \widehat{A}((\ind-\delta)\diam d)))=0,\\
\widehat{b}^T A((\delta-\ind)\diam d))=\widehat{b}^T (\delta\diam A((\delta-\ind)\diam d)))=(\widehat{b}^T d)^2-2 \widehat{b}^T d +\frac{1}{2},\\
\widehat{b}^T (d \diam \widehat{A}c)=\widehat{b}^T (d \diam \widehat{A}d^{\diam 2})=\widehat{b}^T (d \diam \widehat{A}(\delta\diam d^{\diam 2}))=2 \widehat{b}^T (d \diam \widehat{A}d)+(\widehat{b}^T d)^2-2 \widehat{b}^T d +\frac{1}{2},\\
\widehat{b}^T (d^{\diam 2} \diam \widehat{A}d)=\widehat{b}^T (d \diam \widehat{A}d)+\frac{1}{2}(\widehat{b}^T d)^2,\\
\widehat{b}^T (c \diam \widehat{A}((\delta-\ind)\diam d)+\widehat{b}^T (d \diam \widehat{A}((\delta-3\cdot\ind)\diam d)+\widehat{b}^T (d \diam \widehat{A}(\delta\diam c)=2(\widehat{b}^T d)^2-4 \widehat{b}^T d +1,\\
\widehat{b}^T (d^{\diam 2} \diam \widehat{A}(\delta \diam d))+\widehat{b}^T (d \diam \widehat{A}(\delta \diam d))=2\widehat{b}^T (d \diam \widehat{A}d)+\frac{3}{2}(\widehat{b}^T d)^2-2 \widehat{b}^T d +\frac{1}{2},\\
\widehat{b}^T (d \diam (\widehat{A}((\ind-\delta) \diam d))^{\diam 2})=0,\\
\widehat{b}^T (d \diam (\widehat{A} d)^{\diam 2})+3 \widehat{b}^T (d \diam \widehat{A} (d \diam \widehat{A} ((\ind-\delta)\diam d)))=(4-2\widehat{b}^T d)\widehat{b}^T (d \diam \widehat{A}d)+3(\widehat{b}^T d)^2-4 \widehat{b}^T d +1.
\end{array}$$
In the particular case where we set~$\delta=\ind$, the order two conditions reduce to the following:
$$\begin{array}{l}
(\widehat{b}^T d)^2-2 \widehat{b}^T d +\frac{1}{2}=0,\\
\widehat{b}^T d=b^T d,\\
b^T c=b^T d^{\diam 2}=\widehat{b}^T c=\widehat{b}^T d^{\diam 2}=\widehat{b}^T d^{\diam 3}=2\widehat{b}^T d-\frac{1}{2},\\
\widehat{b}^T (d \diam \widehat{A}c)=\widehat{b}^T (d \diam \widehat{A}d^{\diam 2})=2 \widehat{b}^T (d \diam \widehat{A}d),\\
\widehat{b}^T (d^{\diam 2} \diam \widehat{A}d)=\widehat{b}^T (d \diam \widehat{A}d)+\widehat{b}^T d-\frac{1}{4},\\
\widehat{b}^T (d \diam (\widehat{A} d)^{\diam 2})=(4-2\widehat{b}^T d)\widehat{b}^T (d \diam \widehat{A}d) +2 \widehat{b}^T d -\frac{1}{2}.
\end{array}$$
\end{theorem}

For simplicity, we used in Theorem~\ref{theorem:RK_conditions} the notation~$\diam$ of Definition~\ref{definition:diamond}. For instance, the condition~$\widehat{b}^T (d^{\diam 2} \diam \widehat{A}d)=\widehat{b}^T (d \diam \widehat{A}d)+\frac{1}{2}(\widehat{b}^T d)^2$ rewrites into
$$\sum_{i,j=1}^s \widehat{b}_i d_i^2 \widehat{a}_{ij} d_j=\sum_{i,j=1}^d \widehat{b}_i d_i \widehat{a}_{ij} d_j+\frac{1}{2}\Big(\sum_{i=1}^d \widehat{b}_i d_i\Big)^2.$$
The order conditions of Theorem~\ref{theorem:RK_conditions} can be obtained from straightforward calculations with the following methodology.
We compute the operator~$\AA_1$ with the same method used for~$\AA_0$ in Proposition~\ref{proposition:consistency}. It is a differential operator of order four with the following first terms
\begin{equation}
\label{equation:expansion_A1_order_4}
\AA_1\phi=\frac{\sigma^4}{8}\Delta^2\phi-\frac{\sigma^4}{4}G^{-1}\Delta\phi''(g,g)+\frac{\sigma^4}{8}G^{-2}\phi^{(4)}(g,g,g,g)+\BB\phi,
\end{equation}
where~$\BB$ is a differential operator of order three.
We present the complete expansion of~$\AA_1$ in Section~\ref{section:B_series} by using a B-series approach.
If we assume that~$\widehat{b}^T d=b^T d$, then we can integrate by parts to transform~$\int_\MM \AA_1\phi d\mu_\infty$ into an integral of the form~$\int_\MM \AA_1^0\phi d\mu_\infty$ where~$\AA_1^0\phi$ is a differential operator of order one in~$\phi$ (in the spirit of~\cite{Abdulle14hon,Laurent20eab}).
On a manifold, the integration by parts is a corollary of the Green theorem (see, for instance,~\cite[Chap.\ts II]{Sakai96tom}).
As we shall see below, it reveals a crucial tool for deriving order conditions for the invariant measure.
To perform the calculations in a systematic manner, a formalization of the integration by parts process with trees and B-series is presented in Section~\ref{section:B_series}.
\begin{lemma}[Integration by parts on~$\MM$]
If~$\psi:\R^d\to\R$ and~$H:\R^d\to\R^d$ are smooth functions, then
\begin{equation*}
\int_{\MM}(\nabla_\MM\psi, H) d\sigma_\MM=-\int_\MM \psi \Div_\MM(\Pi_\MM H)d\sigma_\MM,
\end{equation*}
where~$\nabla_\MM\psi:=\Pi_\MM \nabla\psi$ and~$\Div_\MM (H):=\Div(H)-G^{-1}(g,H'g)$.
In addition, with the invariant measure~$d\mu_\infty=\rho_\infty d\sigma_\MM$ and~$k\geq 0$, we obtain
\begin{align}
\label{equation:IPP_mu_infty}
\int_{\MM} \Big[G^{-k} \psi'H
&-G^{-(k+1)}(g,H)\psi'g\Big] d\mu_\infty
=
\int_\MM \Big[G^{-(k+1)}(g,H'g)\psi\\
&-(2k+1)G^{-(k+2)}(g,g'g)(g,H)\psi
-G^{-k}\Div(H)\psi
+2kG^{-(k+1)}(g,g'H)\psi \nonumber\\
&+G^{-(k+1)}\Div(g)(g,H)\psi
+\frac{2}{\sigma^2}G^{-(k+1)}(g,f)(g,H)\psi
-\frac{2}{\sigma^2}G^{-k}(f,H)\psi\Big]d\mu_\infty. \nonumber
\end{align}
\end{lemma}

For example, let us integrate by part the terms of order four w.r.t.~$\phi$ of the operator~$\AA_1\phi$ in equation~\eqref{equation:expansion_A1_order_4}.
Applying identity~\eqref{equation:IPP_mu_infty} with~$\psi=\frac{\sigma^4}{8}\Delta\phi'(e_i)$,~$H=e_i$ and~$k=0$, and then summing on~$i=1,\dots,d$ yields
\begin{align}
\label{equation:ex_IPP_1}
\int_{\MM} \Big[\frac{\sigma^4}{8}\Delta^2\phi
&-\frac{\sigma^4}{8}G^{-1}\Delta\phi''(g,g)\Big] d\mu_\infty
=
\int_\MM \Big[-\frac{\sigma^4}{8} G^{-2}(g,g'g)\Delta\phi'g\\
&+\frac{\sigma^4}{8} G^{-1}\Div(g)\Delta\phi'g
+\frac{\sigma^2}{4} G^{-1}(g,f)\Delta\phi'g
-\frac{\sigma^2}{4} \Delta\phi'f\Big]d\mu_\infty. \nonumber
\end{align}
We apply again~\eqref{equation:IPP_mu_infty} with~$\psi=\frac{\sigma^4}{8}\phi^{(3)}(g,g,e_i)$,~$H=e_i$ and~$k=1$, and then sum on~$i=1,\dots,d$ to get
\begin{align}
\label{equation:ex_IPP_2}
\int_{\MM} \Big[\frac{\sigma^4}{8}G^{-1}\Delta\phi''(g,g)
&-\frac{\sigma^4}{8}G^{-2}\phi^{(4)}(g,g,g,g)\Big] d\mu_\infty\\
&=
\int_\MM \Big[-\frac{\sigma^4}{4}G^{-1}\sum_i \phi^{(3)}(g,\partial_i g,ei)
+\frac{\sigma^4}{2}G^{-2}\phi^{(3)}(g,g,g'g) \nonumber\\
&-\frac{3\sigma^4}{8}G^{-3}(g,g'g)\phi^{(3)}(g,g,g)
+\frac{\sigma^4}{8}G^{-2}\Div(g)\phi^{(3)}(g,g,g) \nonumber\\
&+\frac{\sigma^2}{4}G^{-2}(g,f)\phi^{(3)}(g,g,g)
-\frac{\sigma^2}{4}G^{-1}\phi^{(3)}(g,g,f)\Big]d\mu_\infty. \nonumber
\end{align}
Substracting~\eqref{equation:ex_IPP_2} from~\eqref{equation:ex_IPP_1} allows to express 
$\int_\MM \AA_1\phi d\mu_\infty$ with derivatives of~$\phi$ of order strictly less than 4.
We iterate this method to obtain~$\int_\MM \AA_1\phi d\mu_\infty=\int_\MM \AA_1^0\phi d\mu_\infty$ where~$\AA_1^0$ is an operator of order one in~$\phi$, and then find sufficient conditions such that~$\AA_1^0=0$. This implies that~$\AA_1^*\rho_\infty=0$, and Theorem~\ref{theorem:operator_conditions_invariant_measure} then gives the order two for the invariant measure.
The computation of~$\AA_1^0$ is further detailed in Section~\ref{section:B_series}.

Although constructing methods of high weak order is not the main focus of this paper, considering the explicit formula for~$\AA_1$ and comparing with~$\LL^2/2$ (see Section~\ref{section:B_series} for their detailed expansion in B-series), one immediately obtains the following theorem for weak order two of accuracy.
\begin{theorem}[Runge-Kutta conditions for weak order two]
\label{theorem:weak_RK_conditions}
We consider a Runge-Kutta method of the form~\eqref{equation:defRK} and assume that it satisfies~\eqref{equation:consistency_conditions}.
If the following conditions are satisfied, then the integrator has weak order two:
$$\begin{array}{l}
b^T d=b^T c=b^T (\delta \diam c)=b^T d^{\diam 2}=b^T (\delta \diam d^{\diam 2})=\frac{1}{2},\\
\widehat{b}^T d=\widehat{b}^T c=\widehat{b}^T (\delta \diam c)=\widehat{b}^T d^{\diam 2}=\widehat{b}^T(\delta \diam d^{\diam 2})=\widehat{b}^T d^{\diam 3}=\widehat{b}^T(\delta \diam d^{\diam 3})=\frac{1}{2},\\
\widehat{b}^T (c \diam d)=\widehat{b}^T (\delta \diam c \diam d),\\
\widehat{b}^T (d \diam \widehat{A}d)=\frac{1}{8},\\
b^T (d \diam \widehat{A}((\ind-\delta)\diam d))=0,\\
\widehat{b}^T A((\ind-\delta)\diam d)=\widehat{b}^T (\delta\diam A((\ind-\delta)\diam d))=\frac{1}{4},\\
\widehat{b}^T (d \diam \widehat{A}c)=\widehat{b}^T (d \diam \widehat{A}d^{\diam 2})=\widehat{b}^T (d \diam \widehat{A}(\delta\diam d^{\diam 2}))=0,\\
\widehat{b}^T (d^{\diam 2} \diam \widehat{A}d)=\frac{1}{4},\\
\widehat{b}^T (c \diam \widehat{A}((\ind-\delta)\diam d))-\widehat{b}^T (d \diam \widehat{A}(\delta\diam d))-\widehat{b}^T (d \diam \widehat{A}(\delta\diam c))=\frac{1}{8},\\
\widehat{b}^T (d^{\diam 2} \diam \widehat{A}(\delta \diam d))+\widehat{b}^T (d \diam \widehat{A}(\delta \diam d))=\frac{1}{8},\\
\widehat{b}^T (d \diam (\widehat{A}((\ind-\delta) \diam d))^{\diam 2})=0,\\
\widehat{b}^T (d \diam (\widehat{A} d)^{\diam 2})+3 \widehat{b}^T (d \diam \widehat{A} (d \diam \widehat{A} ((\ind-\delta)\diam d)))=\frac{1}{8}.
\end{array}$$
\end{theorem}

\begin{remark}
For~$\delta=\ind$, the weak order two conditions of Theorem~\ref{theorem:weak_RK_conditions} have no solution, which is in contrast with the invariant measure case presented in Theorem~\ref{theorem:RK_conditions}.
Indeed, the condition~$\widehat{b}^T A((\ind-\delta)\diam d)=\frac{1}{4}$ cannot be fulfilled if we fix~$\delta=\ind$.
\end{remark}

\subsection{Illustrative examples of high order Runge-Kutta methods on manifolds}
\label{section:example_methods}

In this section, we present several examples of high order Runge-Kutta methods of the form~\eqref{equation:defRK}. The purpose of these examples is to illustrate our analysis, and deriving new integrators with small error constant, favourable stability properties, small variance and fast convergence to equilibrium is a challenging open question which is not addressed in the present paper.
First, we introduce a method that has order two for sampling the invariant measure of the constrained Langevin dynamics~\eqref{equation:projected_Langevin}.
Since there are many solutions to the order conditions, we obtain this integrator by solving numerically an optimization problem: we minimize the absolute values of the coefficients of the method under the constraints given by the order conditions of Theorem~\ref{theorem:RK_conditions}.
This method is explicit in~$f$ and uses only three evaluations of~$f$ per step.
It is defined by the following Butcher tableau
$$
\begin{array}
{c|cccc||c|cccc||c}
0 & 0 & 0 & 0 & 0 &
 1 & 1 & 0 & 0 & 0 & d_1\\
c_2 & c_2 & 0 & 0 & 0 &
1 & \widehat{a}_{21} & \widehat{a}_{22} & 0 & 0 & d_2\\
c_3 & 0 & c_3 & 0 & 0 &
1 & \widehat{a}_{31} & \widehat{a}_{32} & \widehat{a}_{33} & 0 & d_3\\
1 & \widehat{a}_{41} & \widehat{a}_{42} & \widehat{a}_{43} & 0 &
1 & \widehat{a}_{41} & \widehat{a}_{42} & \widehat{a}_{43} & 0 & 1\\
\hline
 & \widehat{a}_{41} & \widehat{a}_{42} & \widehat{a}_{43} & 0 &
 & \widehat{a}_{41} & \widehat{a}_{42} & \widehat{a}_{43} & 0 & 
\end{array}
$$
or by the associated set of equations
\begin{align}
\label{equation:RK_order_2}
  Y_1&=X_n+\sigma \sqrt{h} d_1 \xi_n + \lambda_1 g(Y_1), \nonumber\\
  Y_2&=X_n+hc_2 f(Y_1)+\sigma \sqrt{h} d_2 \xi_n + \lambda_2 \left[\widehat{a}_{21} g(Y_1)+\widehat{a}_{22}g(Y_2)\right], \nonumber\\
  Y_3&=X_n+hc_3 f(Y_2)+\sigma \sqrt{h} d_3 \xi_n + \lambda_3 \left[\widehat{a}_{31} g(Y_1)+\widehat{a}_{32} g(Y_2)+\widehat{a}_{33}g(Y_3)\right], \nonumber\\
  X_{n+1}&=X_n+h\sum_{j=1}^3 \widehat{a}_{4j} f(Y_j)+\sigma \sqrt{h} \xi_n + \lambda_4 \sum_{j=1}^3 \widehat{a}_{4j} g(Y_j), \\
\text{where }&\lambda_1,\lambda_2,\lambda_3,\lambda_4 \text{ are such that }\zeta(Y_1)=\zeta(Y_2)=\zeta(Y_3)=\zeta(X_{n+1})=0, \nonumber
\end{align}
and with the values of~$c_i$,~$d_i$,~$\widehat{a}_{ij}$ given in Appendix~\ref{section:RK_coefficients}.
To implement one step of this scheme, we apply a few iterations of the Newton method to find the projections on~$\MM$.
We emphasize that if the stepsize~$h$ is not small enough, the fixed point problems of finding~$\lambda_i$ such that~$\zeta(Y_i)=0$ may not be well defined, leading to diverging Newton iterations.
Following Remark~\ref{remark:discrete_rv}, we replace the standard Gaussian random vectors~$\xi$ in~\eqref{equation:defRK} by independent bounded discrete random vectors~$\widehat{\xi}$ that satisfy~\eqref{equation:discrete_rv}.
This way, the order two for the invariant measure is preserved and the method is well-posed for~$h$ small enough.

With the same methodology we used to obtain the order conditions of Theorem~\ref{theorem:RK_conditions} and Theorem~\ref{theorem:weak_RK_conditions}, and with the expressions of~$\AA_1\phi$ and~$\AA_1^0\phi$ (see Section~\ref{section:B_series} for further details), we also get classes of Runge-Kutta integrators and their order conditions for the following specific subproblems.
\paragraph{Euclidean case~$\R^d$.\!\!}
Fixing~$g=0$ in the expressions of~$\AA_1\phi$ and~$\AA_1^0\phi$ yields the order two conditions in the weak sense and for the invariant measure in~$\R^d$ as given in~\cite[Tables 1-2]{Laurent20eab}.
\paragraph{Deterministic case.\!\!}
Fixing~$\sigma=0$ in the expression of~$\AA_1\phi$ yields the order conditions for approximating the solution of ODEs of the form~$\dot x=\Pi_\MM(x)f(x)$, where~$f$ is a gradient.
Note that this equation can be rewritten as the following differential algebraic equation (DAE) of index two (see~\cite[Chap.\ts VII]{Hairer10sod}):
\begin{align}
\label{equation:DAE_problem}
\dot{x}&=f(x)+\lambda g(x),\\
0&=\zeta(x). \nonumber
\end{align}
We obtain a class of deterministic Runge-Kutta methods for solving DAEs of the form~\eqref{equation:DAE_problem} by setting~$\sigma=0$ in~\eqref{equation:defRK}. A Runge-Kutta method of this form is consistent if~$b^T \ind=1$, and has order two if~$\widehat{b}^T c=\widehat{b}^T (\delta \diam c)=b^T c=b^T (\delta \diam c)=1/2$.
For instance, an order two method for solving ODEs of the form~\eqref{equation:DAE_problem} is
$$
  X_{n+1}=X_n+h\frac{f(X_n)+f(X_{n+1})}{2}+ \lambda \frac{g(X_n)+g(X_{n+1})}{2},\quad
  \zeta(X_{n+1})=0.
$$
\paragraph{Spherical case.\!\!}
In the simple case where~$\MM$ is the unit sphere in~$\R^d$ (that is, when the constraint is of the form~$\zeta(x)=(\abs{x}^2-1)/2$ and~$g(x)=x$), the consistency conditions~\eqref{equation:consistency_conditions} reduce to~$b^T \ind=d_s=1$.
The weak order two conditions of Theorem~\ref{theorem:weak_RK_conditions} reduce to the following conditions:
$$\begin{array}{l}
b^T d=b^T c=b^T (\delta \diam c)=b^T d^{\diam 2}=b^T (\delta \diam d^{\diam 2})=\widehat{b}^T d=\widehat{b}^T c=\frac{1}{2},\\
\widehat{b}^T (d \diam \widehat{A}d)=\frac{1}{8},\\
\widehat{b}^T A((\ind-\delta)\diam d)=\frac{1}{4},\\
\widehat{b}^T (d \diam \widehat{A}c)=0.
\end{array}$$
On the other hand, the order two conditions for the invariant measure of Theorem~\ref{theorem:RK_conditions} on the sphere are the following:
$$\begin{array}{l}
\widehat{b}^T d=b^T d,\\
b^T c=b^T (\delta \diam c)=b^T d^{\diam 2}=b^T (\delta \diam d^{\diam 2})=\widehat{b}^T c=2\widehat{b}^T d-\frac{1}{2},\\
\widehat{b}^T (d \diam \widehat{A}c)=2 \widehat{b}^T (d \diam \widehat{A}d)+(\widehat{b}^T d)^2-2 \widehat{b}^T d +\frac{1}{2},\\
\widehat{b}^T A((\delta-\ind)\diam d)=(\widehat{b}^T d)^2-2 \widehat{b}^T d +\frac{1}{2}.
\end{array}$$
For example, the following integrator has order two for the invariant measure if~$\MM$ is a sphere:
\begin{align*}
  Y_1&=X_n+h\left(\frac{3}{2}-\sqrt{2}\right) f(Y_2) +\sigma \sqrt{h} \left(1-\frac{\sqrt{2}}{2}\right) \xi_n + \lambda_1 (2 Y_1-Y_2), \quad \zeta(Y_1)=0,\\
  Y_2&=X_n+h f(Y_1) +\sigma \sqrt{h} \xi_n + \lambda_2 Y_1, \quad \zeta(Y_2)=0,\\
  X_{n+1}&=Y_2.
\end{align*}
\paragraph{Brownian motions on manifolds.\!\!}
Runge-Kutta methods of the form~\eqref{equation:defRK} can also be used for simulating a Brownian motion on a manifold (see~\cite[Chap.\ts III]{Hsu02sao}) by solving numerically
\begin{equation}
\label{equation:Brownian_walk}
dX=\Pi_\MM(X) \circ dW,\quad X(0)=X_0\in \MM.
\end{equation}
We recall that in the context of~$\R^d$, the Euler-Maruyama integrator is exact for approximating a Brownian motion in law. However, in the context of manifolds, there are no exact Runge-Kutta integrators for simulating a Brownian motion on~$\MM$ in general. In particular, the Euler scheme~\eqref{equation:IE} only has weak order one for solving~\eqref{equation:Brownian_walk} in general.
Fixing~$f=0$ in~\eqref{equation:defRK} yields a class of Runge-Kutta methods for solving~\eqref{equation:Brownian_walk}. The consistency conditions are~$d_s=1$ and~$\widehat{b}^T d=\widehat{b}^T (\delta \diam d)$.
The conditions for order two for the invariant measure (respectively for weak order two) of such a Runge-Kutta method are obtained by deleting the order conditions in Theorem~\ref{theorem:RK_conditions} (respectively in Theorem~\ref{theorem:weak_RK_conditions}) that involve~$A$,~$b$ or~$c$.
In the specific case where~$\MM$ is a sphere, the consistency conditions~\eqref{equation:consistency_conditions} become~$d_s=1$ and the weak order two conditions of Theorem~\ref{theorem:weak_RK_conditions} reduce to the two following conditions
$$\widehat{b}^T d=\frac{1}{2}, \quad \widehat{b}^T (d \diam \widehat{A}d)=\frac{1}{8}.$$
For example, a weak order two method for simulating a Brownian motion on a sphere is
$$
  X_{n+1}=X_n+\sqrt{h}\xi_n+ \lambda \frac{3 X_n +\sqrt{h}\xi_n + X_{n+1}}{4},\quad
  \zeta(X_{n+1})=0.
$$
In addition, there are no additional order two conditions for the invariant measure, that is, any consistent integrator, such as the Euler scheme~\eqref{equation:IE}, has at least order two for the invariant measure on the sphere.

\section{Exotic aromatic B-series for computing order conditions}
\label{section:B_series}

As described in the introduction, B-series were introduced to tackle the calculations of order conditions of ODEs by representing Taylor expansions with trees.
In~\cite{Chartier07pfi}, an extension of the original B-series, called aromatic B-series, was used to study volume-preserving integrators. It allowed in particular to represent the divergence of a B-series. B-series and aromatic B-series were also studied later in~\cite{MuntheKaas16abs,McLachlan16bsm,Floystad20tup} for their geometric properties, and in~\cite{Chartier10aso,Bogfjellmo19aso} for their algebraic structure of Hopf algebras.
In~\cite{Laurent20eab}, a new formalism of B-series, called exotic aromatic B-series, was introduced for computing order conditions for sampling the invariant measure of SDEs in~$\R^d$. It added a new kind of edge, called \emph{liana}, to the aromatic trees in order to represent new terms such as the Laplacian of an aromatic B-series.
In this section, we extend the formalism of exotic aromatic B-series by allowing the representation of scalar products, and show that the operators~$\LL^j$ and~$\AA_j$ can be represented conveniently in the form of B-series. We also rewrite the integration by parts formula~\eqref{equation:IPP_mu_infty} as a straightforward process on graphs, and apply it to compute~$\AA_1^0$.

We consider graphs~$\gamma=(V,E,L)$ where~$V$ is the set of nodes,~$E$ the set of edges and~$L$ the set of lianas.
We split the set of edges into~$E=E_0\cup E_S$ where~$E_0$ are the standard oriented edges as defined in~\cite{Laurent20eab}, and where~$E_S$ is a new set of non-oriented edges represented as double horizontal straight lines. If~$(v,w)=(w,v)\in E_S$, we consider this edge as an outgoing edge for both~$v$ and~$w$, but~$v$ and~$w$ are not predecessors of each other. If~$(v,w)\in E_S$, we denote~$S(v)=w$ and~$S(v)=v$ otherwise. 
We consider graphs where each node has exactly one outgoing edge, except exactly one node, called the root~$r$, that has none.
If we consider only the graph~$(V,E)$, where we erase the lianas, it can be decomposed in two kinds of connected components: one that contains the root, that we name the rooted tree, and the other components that we name aromas.
We decompose the set of nodes in~$V=V_f\cup V_g\cup\{r\}$ where~$V_f$ are the nodes representing the function~$f$ and are represented with black disks (respectively~$V_g$ represent the function~$g$ and are drawn with white disks).
We write~$N_f(\gamma)$ the number of elements of~$V_f$ (respectively~$N_g(\gamma)$ the number of elements of~$V_g$) and~$N_l(\gamma)$ the number of lianas.
The order of a directed graph~$\gamma=(V,E,L)$ is defined as
$$\abs{\gamma} =N_f(\gamma) +N_l(\gamma) +\frac{N_g(\gamma)}{2} -\abs{E_S}.$$
For instance, the graph~$\gamma=(V,E,L)$ with 
\begin{equation}
\label{equation:example_forest}
V_f=\{v_2,v_5,v_6\}, \quad V_g=\{v_1,v_3,v_4,v_7\}, \quad E_S=\{(v_6,v_7)\},
\end{equation}
$$E_0=\{(v_1,r),(v_2,v_1),(v_3,r),(v_4,v_4),(v_5,v_4)\},
\quad L=\{(v_2,v_2),(v_3,v_5),(v_5,v_6)\},$$
satisfies~$\abs{\gamma}=7$ and is represented as
$$\includegraphics[scale=1]{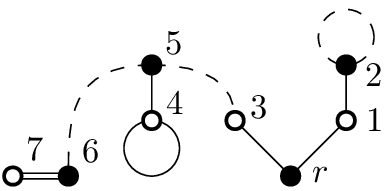}.$$
We say that two directed graphs~$(V^1,E^1,L^1)$ and~$(V^2,E^2,L^2)$ are equivalent if there exists a bijection~$\varphi:V^1\to V^2$ such that
$$\varphi(V_f^1)=V_f^2, \quad \varphi(V_g^1)=V_g^2, \quad (\varphi\times \varphi)(E^1)=E^2, \quad (\varphi\times \varphi)(E_S^1)=E_S^2, \quad (\varphi\times \varphi)(L^1)=L^2.$$
We call exotic aromatic forests the equivalence classes of these directed graphs~$\gamma=(V,E,L)$, and we denote~$\EE\AA\TT$ the set of exotic aromatic forests.
In addition, we need a different set of rooted forests where the root is in~$V_f$ or~$V_g$. We call them exotic aromatic vector fields and gather them together in the set~$\EE\AA\VV$.
The elementary differential associated to an exotic aromatic forest is given by the following definition.
\begin{definition}
Let~$\gamma=(V,E,L)\in \EE\AA\TT$, and let~$f$,~$g\colon \R^d \to \R^d$ and~$\phi:\R^d\to \R$ be smooth functions. We denote~$l_1,\dots ,l_s$ the elements of~$L$,~$v_1,\dots ,v_m$ the elements of~$V\smallsetminus\{r\}$ and~$\delta_{i,j}$ the Kronecker symbol ($\delta_{i,j}=1$ if~$i=j$,~$\delta_{i,j}=0$ else). We use the notation for~$v\in V$,~$I_{\pi(v)}=(i_{q_1},\dots ,i_{q_s})$ where~$\pi(v)=\{q_1,\dots ,q_s\}$ are the predecessors of~$v$, and~$J_{\Gamma(v)}=(j_{l_{x_1}},\dots ,j_{l_{x_t}})$ where~$\Gamma(v)=\{l_{x_1},\dots ,l_{x_t}\}$ are the lianas linked to~$v$.
Then~$F(\gamma)$ is defined as
\begin{align*}
F(\gamma)(f,g,\phi)&=\sigma^{2(\abs{\gamma} -N_f(\gamma))} G^{-N_g(\gamma)/2}\sum_{i_{v_1},\dots ,i_{v_m}=1}^d \sum_{j_{l_1},\dots ,j_{l_s}=1}^d 
\left(\prod_{v\in V_f} \delta_{i_v,i_{S(v)}} \partial_{I_{\pi(v)}} \partial_{J_{\Gamma(v)}} f_{i_v}\right)\\
&\cdot\left(\prod_{v\in V_g} \delta_{i_v,i_{S(v)}} \partial_{I_{\pi(v)}} \partial_{J_{\Gamma(v)}} g_{i_v}\right)
 \partial_{I_{\pi(r)}} \partial_{J_{\Gamma(r)}} \phi.
\end{align*}
\end{definition}

For example, the differential associated to the exotic aromatic forest~$\gamma$ given by~\eqref{equation:example_forest} is
\begin{align*}
F(\gamma)(f,g,\phi)&=\sigma^{8} G^{-2}\sum_{i_{v_1},\dots ,i_{v_7}=1}^d \sum_{j_{l_1},\dots ,j_{l_3}=1}^d
\partial_{j_{l_1}j_{l_1}}f_{i_{v_2}}
\partial_{j_{l_2}j_{l_3}}f_{i_{v_5}}
\delta_{i_{v_6},i_{v_7}} \partial_{j_{l_3}}f_{i_{v_6}}\\
&\cdot \partial_{i_{v_2}}g_{i_{v_1}}
\partial_{j_{l_2}}g_{i_{v_3}}
\partial_{i_{v_4}i_{v_5}}g_{i_{v_4}}
\delta_{i_{v_7},i_{v_6}} g_{i_{v_7}}
\partial_{i_{v_1}i_{v_3}}\phi
.
\end{align*}
We extend the definition of~$F$ on~$\Span(\EE\AA\TT)$ by linearity and write, for the sake of simplicity,~$F(\gamma)(\phi)$ instead of~$F(\gamma)(f,g,\phi)$.
An exotic aromatic B-series is a formal series indexed over~$\EE\AA\TT$ of the form
$$B(a)(\phi)=\sum_{\gamma\in\EE\AA\TT} h^{\abs{\gamma}}a(\gamma)F(\gamma)(\phi).$$

\begin{remark}
As we assumed that the functions~$f$ and~$g$ are gradients, multiple exotic aromatic forests can represent the same differential. We do not detail here the method to identify two such forests  as it is similar to~\cite[Prop.\ts 4.7]{Laurent20eab} in the context of~$\R^d$.
\end{remark}

The following result states that the operators~$\LL^j/j!$ and~$\AA_j$ can be written with exotic aromatic forests.
We omit the proof for the sake of brevity as it is similar to~\cite[Thm.\ts 4.1]{Laurent20eab}.

\begin{proposition}
Take a Runge-Kutta method of the form~\eqref{equation:defRK}, then the expansions~\eqref{equation:dvp_exa} and~\eqref{equation:dvp_num} can be formally written with exotic aromatic B-series, that is, there exists two maps~$e$ and~$a$ over~$\EE\AA\TT$ such that
$$\E[\phi(X(h))|X(0)=x]=B(e)(\phi)(x), \quad \E[\phi(X_1)|X_0=x]=B(a)(\phi)(x),$$
and where the operators are given by
$$\frac{\LL^j}{j!}=F\bigg(\sum_{\gamma\in\EE\AA\TT,\abs{\gamma}=j} e(\gamma)\gamma\bigg), \quad \AA_{j-1}=F\bigg(\sum_{\gamma\in\EE\AA\TT,\abs{\gamma}=j} a(\gamma)\gamma\bigg).$$
If~$e(\gamma)=a(\gamma)$ for all~$\gamma\in\EE\AA\TT$ with~$1\leq\abs{\gamma}\leq p$, then the integrator has at least weak order~$p$.
\end{proposition}

For example, the operator~$\LL$ in~\eqref{equation:generator_L_Langevin} can be rewritten with exotic aromatic forests as
\begin{align*}
\LL\phi
&=\phi'f
-G^{-1}(g,f)\phi'g
-\frac{\sigma^2}{2}G^{-1}\Div(g)\phi'g
+\frac{\sigma^2}{2}G^{-2}(g,g'g)\phi'g
+\frac{\sigma^2}{2}\Delta \phi
-\frac{\sigma^2}{2}G^{-1}\phi''(g,g)\\
&=F\Big(
\eatree2001
-\aroma2101\ \eatree2101
-\frac{1}{2}\ \aroma1101\ \eatree2101
+\frac{1}{2}\ \aroma3301\ \eatree2101
+\frac{1}{2}\ \eatree1011
-\frac{1}{2}\ \eatree3201
\Big)(\phi).
\end{align*}
We present in Table 2 (see Appendix~\ref{section:Tables_B_series}) the decomposition in exotic aromatic forests of the operators~$\LL^2\phi/2=\LL(\LL\phi)/2$ and~$\AA_1\phi$ under the consistency condition~\eqref{equation:consistency_conditions}.

\begin{remark}
If we replace the functions~$g$ and~$\phi$ by~$f$ and fix~$\sigma=G=1$, the newly obtained exotic aromatic B-series satisfy an isometric equivariance property, that is, they stay unchanged when applying an isometric coordinate transformation.
It was proved in~\cite{MuntheKaas16abs} that, under a condition of locality, aromatic B-series are exactly the affine equivariant methods, that is, the maps that stay unchanged when applying an affine coordinate transformation.
Analogously, it would be interesting to make a link between the isometric equivariant maps and the exotic aromatic B-series.
\end{remark}

In the spirit of the Butcher product on trees~\cite[Chap.\ts III]{Hairer06gni}, we introduce a few notations for writing with ease different operations on forests.
\begin{notation*}
Let~$\gamma$ be an exotic aromatic forest/vector field,~$\tau$ be an exotic aromatic vector field and~$v$ a node of~$\gamma$, then we define the following operators on forests.
\begin{enumerate}
\item~$\includegraphics[scale=0.5]{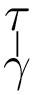}$: sum of all exotic aromatic forests/vector fields obtained by linking the root of~$\tau$ to a node of~$\gamma$ with a new edge in~$E_0$
\item~$\includegraphics[scale=0.5]{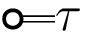}$ (resp.~$\includegraphics[scale=0.5]{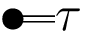}$): aroma obtained by linking the root of~$\tau$ to a white node (resp.\ts a black node) with a new edge in~$E_S$
\item~$\includegraphics[scale=0.5]{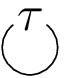}$: sum of all aromas obtained by linking the root of~$\tau$ to a node of~$\tau$ with a new edge in~$E_0$
\item~$\includegraphics[scale=0.5]{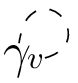}$: sum of all exotic aromatic forests/vector fields obtained by linking the node~$v$ to a node of~$\gamma$ with a new liana
\item~$\includegraphics[scale=0.5]{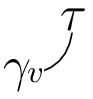}$: forest obtained by linking the root of~$\tau$ to the node~$v$ of~$\gamma$ with a new edge in~$E_0$
\end{enumerate}
For simplicity, we combine multiple operations on a same forest as in~$\includegraphics[scale=0.5]{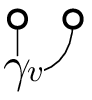}$ and~$\includegraphics[scale=0.5]{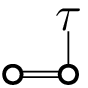}$, where operation 1 is always applied first.
\end{notation*}

For example, let~$\gamma=\eatree3201$,~$\tau=\includegraphics[scale=0.5]{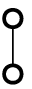}$ and~$v=r$ the root of~$\gamma$, then we get
$$
\includegraphics[scale=0.5]{Special_trees/Deriv}
=\eatree5402+2\ \eatree5403,
\quad
\includegraphics[scale=0.5]{Special_trees/Scal_g}
=\aroma3301,
\quad
\includegraphics[scale=0.5]{Special_trees/Div}
=\aroma2201 + \aroma2202,
\quad
\includegraphics[scale=0.5]{Special_trees/Deriv_liana}
=\eatree3211 + 2\ \eatree3212,
\quad
\includegraphics[scale=0.5]{Special_trees/Plug}
=\eatree5402.
$$

The integration by parts~\eqref{equation:IPP_mu_infty} can be rewritten conveniently with exotic aromatic forests.
\begin{lemma}
\label{lemma:IPP_B_series}
Let~$\gamma \in \EE\AA\TT$ and~$\tau \in \EE\AA\VV$, then the process of integration by parts rewrites into
\begin{align}
\label{equation:IPP_B_series_edge}
\int_\MM F\Big(
\includegraphics[scale=0.5]{Special_trees/Deriv}
- \includegraphics[scale=0.5]{Special_trees/Scal_g} \ \includegraphics[scale=0.5]{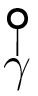} \Big)(\phi) d\mu_\infty
&=\int_\MM F\Big(
\includegraphics[scale=0.5]{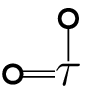} \ \includegraphics[scale=0.5]{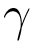}
-(N_g(\gamma)+N_g(\tau)+1) \ \includegraphics[scale=0.5]{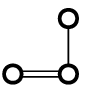} \ \includegraphics[scale=0.5]{Special_trees/Scal_g} \ \includegraphics[scale=0.5]{Special_trees/Id}
-\includegraphics[scale=0.5]{Special_trees/Div} \ \includegraphics[scale=0.5]{Special_trees/Id}\\&
+(N_g(\gamma)+N_g(\tau)) \ \includegraphics[scale=0.5]{Special_trees/Scal_g_deriv} \ \includegraphics[scale=0.5]{Special_trees/Id}
+\includegraphics[scale=0.5]{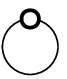} \ \includegraphics[scale=0.5]{Special_trees/Scal_g} \ \includegraphics[scale=0.5]{Special_trees/Id}
+2 \ \includegraphics[scale=0.5]{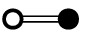} \ \includegraphics[scale=0.5]{Special_trees/Scal_g} \ \includegraphics[scale=0.5]{Special_trees/Id}
-2 \ \includegraphics[scale=0.5]{Special_trees/Scal_f} \ \includegraphics[scale=0.5]{Special_trees/Id} \Big)(\phi) d\mu_\infty, \nonumber \\
\label{equation:IPP_B_series_liana}
\int_\MM F\Big(
\includegraphics[scale=0.5]{Special_trees/Deriv_liana}
- \includegraphics[scale=0.5]{Special_trees/Deriv_g_Plug_g} \Big)(\phi) d\mu_\infty
&=\int_\MM F\Big(
-(N_g(\gamma)+1) \ \includegraphics[scale=0.5]{EATrees/aroma3301} \ \includegraphics[scale=0.5]{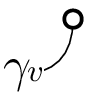}\\
&+N_g(\gamma) \ \includegraphics[scale=0.5]{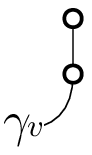}
+\includegraphics[scale=0.5]{EATrees/aroma1101} \ \includegraphics[scale=0.5]{Special_trees/Plug_g}
+2 \ \includegraphics[scale=0.5]{EATrees/aroma2101} \ \includegraphics[scale=0.5]{Special_trees/Plug_g}
-2 \ \includegraphics[scale=0.5]{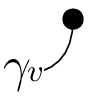} \Big)(\phi) d\mu_\infty. \nonumber
\end{align}
\end{lemma}

We write~$\gamma\sim \widetilde{\gamma}$ if it is possible to go from~$\gamma \in \EE\AA\TT$ to~$\widetilde{\gamma}\in \Span(\EE\AA\TT)$ with the processes of integration by parts~\eqref{equation:IPP_B_series_edge} or~\eqref{equation:IPP_B_series_liana}. We extend this relation by linearity on~$\Span(\EE\AA\TT)$ and make it symmetric so that~$\sim$ becomes an equivalence relation on~$\Span(\EE\AA\TT)$.
For example, the integrations by parts~\eqref{equation:ex_IPP_1} and~\eqref{equation:ex_IPP_2} can be rewritten with exotic aromatic B-series by using~\eqref{equation:IPP_B_series_liana} with~$\gamma=\eatree1011$ and~$\gamma=\eatree3201$. It yields
\begin{align*}
\frac{1}{8}\ \eatree1021
-\frac{1}{8}\ \eatree3211
&\sim
-\frac{1}{8}\ \aroma3301\ \eatree2111
+\frac{1}{8}\ \aroma1101\ \eatree2111
+\frac{1}{4}\ \aroma2101\ \eatree2111
-\frac{1}{4}\ \eatree2013,\\
\frac{1}{8}\ \eatree3211
-\frac{1}{8}\ \eatree5401
&\sim
-\frac{1}{4}\ \eatree3212
+\frac{1}{2}\ \eatree5402
-\frac{3}{8}\ \aroma3301\ \eatree4301
+\frac{1}{8}\ \aroma1101\ \eatree4301
+\frac{1}{4}\ \aroma2101\ \eatree4301
-\frac{1}{4}\ \eatree4201.
\end{align*}
For the sake of completeness, we present in Appendix~\ref{section:IPP_details} the integrations by parts for the order~$3$ terms of~$\AA_1\phi$. The computations are similar for the terms of order two in~$\phi$.

\begin{remark}
In the Euclidean case~$\R^d$, that is, for a forest~$\gamma \in\EE\AA\TT$ and a vector field~$\tau \in\EE\AA\VV$ with~$N_g(\gamma)=N_g(\tau)=0$ and~$g=0$, Lemma~\ref{lemma:IPP_B_series} reduces to the two following equations:
$$
\includegraphics[scale=0.5]{Special_trees/Deriv}
\sim
-\includegraphics[scale=0.5]{Special_trees/Div} \ \includegraphics[scale=0.5]{Special_trees/Id}
-2 \ \includegraphics[scale=0.5]{Special_trees/Scal_f} \ \includegraphics[scale=0.5]{Special_trees/Id},
\quad
\includegraphics[scale=0.5]{Special_trees/Deriv_liana}
\sim
-2 \ \includegraphics[scale=0.5]{Special_trees/Plug_f}.
$$
We recover the process of integration by parts described in~\cite[Thm.\ts 4.4]{Laurent20eab} in the context of exotic aromatic B-series in~$\R^d$.
\end{remark}

We can now revisit the statement of Theorem~\ref{theorem:operator_conditions_invariant_measure} in terms of B-series.
\begin{theorem}
\label{theorem:B_series_order_conditions}
Take a consistent ergodic Runge-Kutta method of the form~\eqref{equation:defRK}. We denote~$\AA_{i}=F(\gamma_i)$ with~$\gamma_i\in\EE\AA\TT$. If~$\gamma_i\sim \gamma_i^0$ and~$F(\gamma_i^0)=0$ for~$1\leq i< r$, then the method has at least order~$r$ for the invariant measure.
\end{theorem}

By applying repeatedly the process of integrations by parts described in Lemma~\ref{lemma:IPP_B_series}, one can simplify the operator~$\AA_{i}=F(\gamma_i)$ into an operator of the form~$\AA_{i}^0=F(\gamma_i^0)$ such that~$\gamma_i\sim \gamma_i^0$.
The complete decomposition of~$\AA_1^0$ into exotic aromatic forests is detailed in Table 3 (see Appendix~\ref{section:Tables_B_series}).
According to Theorem~\ref{theorem:B_series_order_conditions}, choosing the coefficients of the Runge-Kutta method such that~$\gamma_1^0=0$ yields the order two conditions for the invariant measure, as stated in Theorem~\ref{theorem:RK_conditions}.

\begin{remark}
We call~$\EE\AA\TT^0$ the subset of exotic aromatic forests whose root has only one predecessor (that is, the forests associated to an order one operator) or that have a rooted tree of the form~$\eatree3201$,~$\eatree4301$,~$\eatree5401$, \dots
Then, if~$\gamma \in\EE\AA\TT$, there exists~$\gamma^0 \in \EE\AA\TT^0$ such that~$\gamma\sim \gamma^0$.
For instance, for a consistent method of the form~\eqref{equation:defRK}, the operator~$\AA_1^0=F(\gamma_1^0)$ has the form
$$\gamma_1^0=(b^T d-\widehat{b}^T d)\ \aroma3202\ \eatree3201 + \sum_{\underset{\abs{\pi(r)}=1}{\abs{\gamma}=2}} a^0(\gamma)\gamma,$$
so that~$\gamma_1^0 \in \EE\AA\TT^0$, and~$\AA_1^0$ is a differential operator of order one if the condition~$b^T d=\widehat{b}^T d$ holds.
\end{remark}

\section{Numerical experiments}
\label{section:numerical_experiments}

In this section, we perform numerical experiments to confirm the theoretical findings, first on a sphere and a torus in~$\R^3$, and then on the special linear group.

\subsection{Invariant measure approximation on a sphere and a torus}

To check the numerical order two of the Runge-Kutta integrator~\eqref{equation:RK_order_2} presented in Section~\ref{section:example_methods}, we first compare it with the Euler scheme~\eqref{equation:IE} on the unit sphere in~$\R^3$, where the constraint is given by~$\zeta(x)=(x_1^2+x_2^2+x_3^2-1)/2$. We choose the potential~$V(x)=25(1-x_1^2-x_2^2)$, with~$\sigma=\sqrt{2}$,~$\phi(x)=x_3^2$,~$f=-\nabla V$,~$g=\nabla \zeta$,~$M=10^7$ independent trajectories to have a small Monte-Carlo error and a final time~$T=20$.
Observe that for the smaller final time~$T=10$ (not included in the figures for conciseness), the convergence curves reveal nearly identical to the case~$T=20$ considered in Figure~\ref{figure:Plot_sphere}, which suggests that the numerical solutions are already very close to equilibrium at these final times.
Following Remark~\ref{remark:discrete_rv} and Lemma~\ref{lemma:existence_fixed_point}, we use discrete bounded random variables satisfying~\eqref{equation:discrete_rv} in the implementation of the integrators.
For both integrators, we compute the Monte-Carlo estimator~$\widebar{J}=\frac{1}{M}\sum_{m=1}^M \phi(X_N^{(m)}) \simeq \E[\phi(X_N)]$, where~$X_n^{(m)}$ is the~$m$-th realisation of the integrator at time~$t_n=nh$, and~$N$ is an integer satisfying~$Nh=T$.
We compare this approximation with a reference value of~$\int_\MM \phi d\mu_\infty$ computed via a standard quadrature formula, and we plot the error for the invariant measure~\eqref{equation:def_error_inv_measure} versus different timestep~$h$.
We also plot an estimate of the Monte-Carlo error by using the standard error of the mean estimator
$\big(\sum_{m=1}^M (\phi(X_N^{(m)})-\widebar{J})^2\big)^{1/2}/\sqrt{M(M-1)}$.
We observe in all convergence plots that the Monte-Carlo error prevails for small values of the timestep~$h$.
On Figure~\ref{figure:Plot_sphere}, we observe as expected order one for the Euler scheme~\eqref{equation:IE} and order two for the Runge-Kutta scheme~\eqref{equation:RK_order_2}.
\begin{figure}[ht]
	\begin{minipage}[c]{.49\linewidth}
		\includegraphics[scale=0.5]{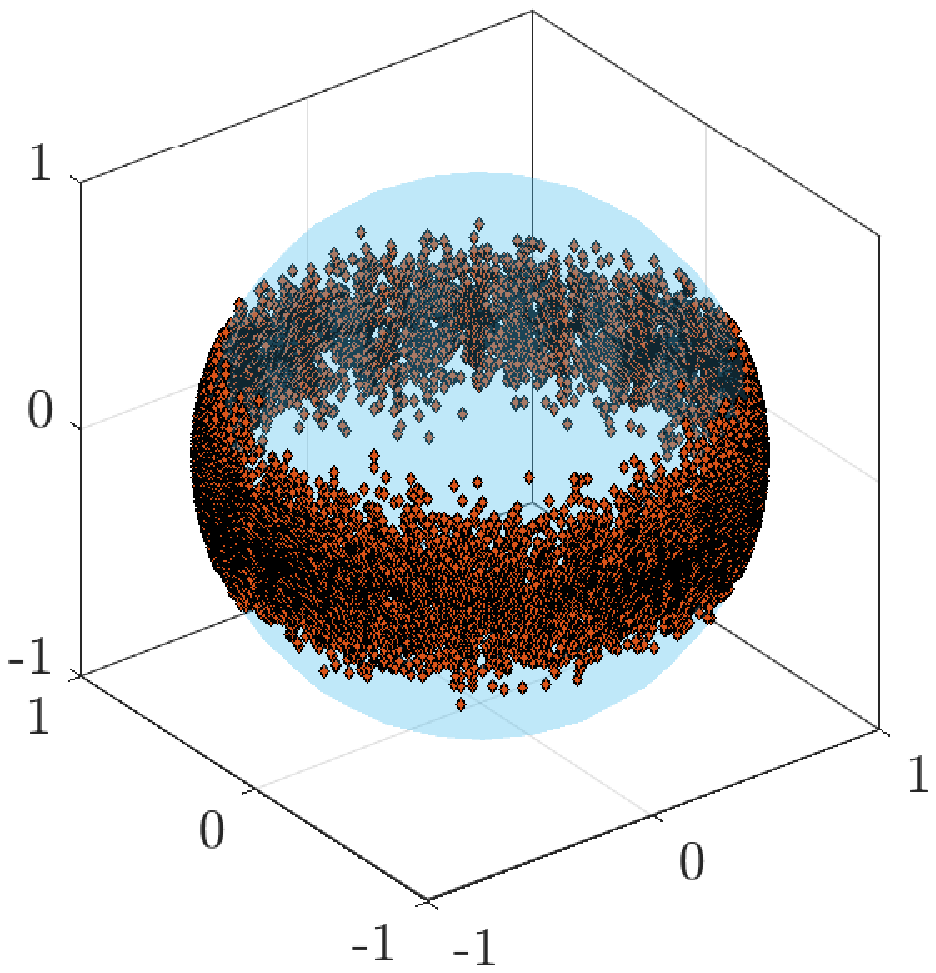}
	\end{minipage} \hfill
	\begin{minipage}[c]{.49\linewidth}
		\includegraphics[scale=0.5]{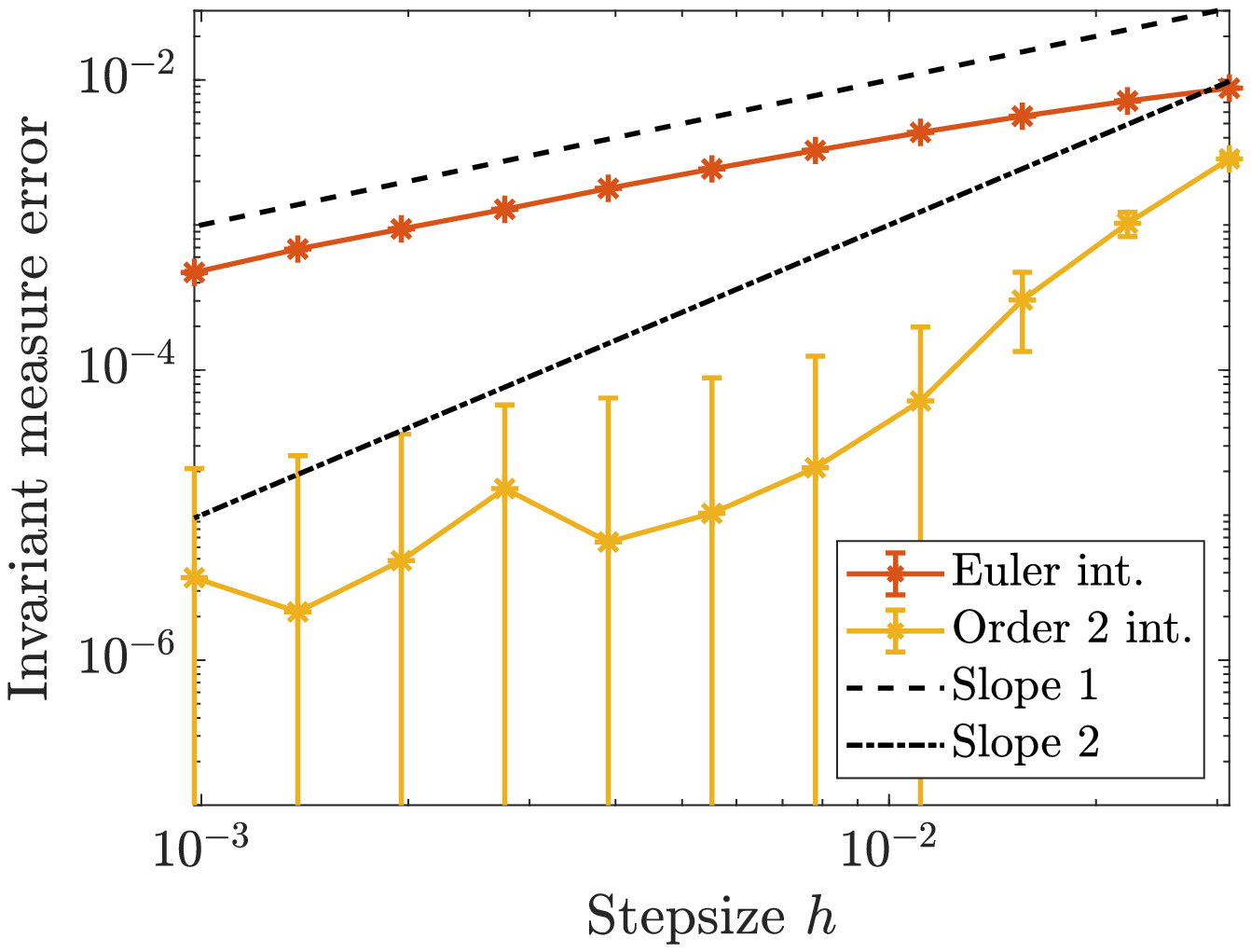}
	\end{minipage}
	\caption{A trajectory of the order two method (left) and the convergence curve for the sphere for the invariant measure (right) with the potential~$V(x)=25(1-x_1^2-x_2^2)$,~$\phi(x)=x_3^2$, a final time~$T=20$ and~$M=10^7$ trajectories.}
	\label{figure:Plot_sphere}
\end{figure}

We then apply the Euler scheme~\eqref{equation:IE} and the Runge-Kutta integrator~\eqref{equation:RK_order_2} on a torus defined by the constraint~$\zeta(x)=(x_1^2+x_2^2+x_3^2+R^2-r^2)^2-4R^2(x_1^2+x_2^2)$ with~$R=3$ and~$r=1$. The potential is~$V(x)=25(x_3-r)^2$ and we choose~$\sigma=\sqrt{2}$,~$\phi(x)=x_3^2$,~$f=-\nabla V$,~$g=\nabla \zeta$, a final time~$T=20$ and~$M=10^7$ independent trajectories.
On Figure~\ref{figure:Plot_torus}, we plot the error for the invariant measure versus the timestep~$h$, by using a reference value for~$\int_\MM \phi d\mu_\infty$ obtained with a standard quadrature formula. As expected, we observe order two for the proposed integrator.
These curves confirm the theoretical findings presented in Section~\ref{section:Runge_Kutta_section}.
In particular, the scheme~\eqref{equation:RK_order_2} has order two of accuracy for the invariant measure on manifolds, according to Theorem~\ref{theorem:RK_conditions}.
Note that if we had chosen a very short final time~$T$, we would have observed the weak order one instead of the order two for the invariant measure as we would not have reached equilibrium.
\begin{figure}[ht]
	\begin{minipage}[c]{.49\linewidth}
		\includegraphics[scale=0.5]{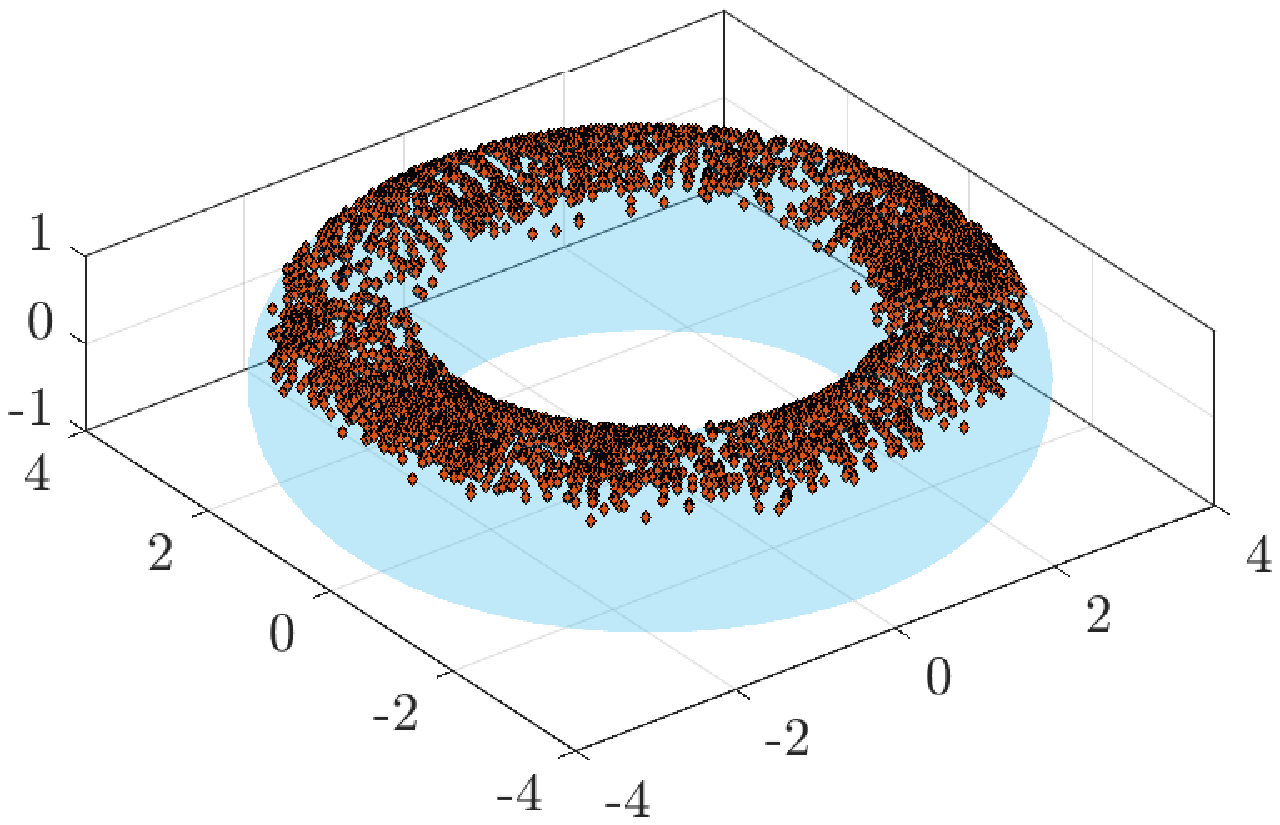}
	\end{minipage} \hfill
	\begin{minipage}[c]{.49\linewidth}
		\includegraphics[scale=0.5]{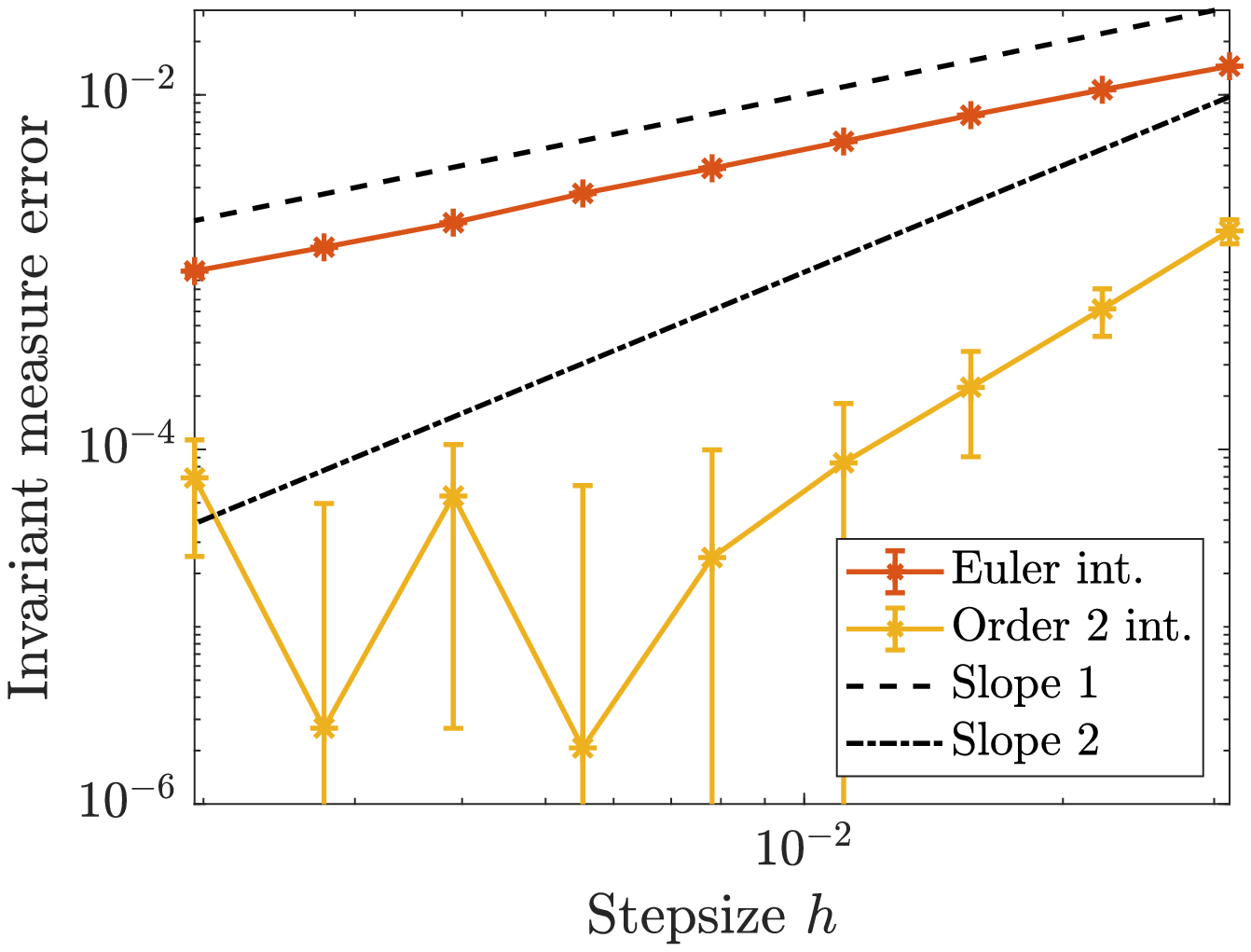}
	\end{minipage}
	\caption{A trajectory of the order two method (left) and the convergence curve for the torus for the invariant measure (right) with the potential~$V(x)=25(x_3-r)^2$,~$\phi(x)=x_3^2$, a final time~$T=20$ and~$M=10^7$ trajectories.}
	\label{figure:Plot_torus}
\end{figure}

\subsection{Invariant measure approximation on the special linear group}

Sampling on a manifold~$\MM$ is especially useful to compute integrals of the form~$\int_\MM \phi(x) d\mu_\infty$ when~$\MM$ is a manifold of high dimension. The class of methods~\eqref{equation:defRK} is convenient as the number of order conditions does not increase with the dimension of the space increasing. We apply Method~\eqref{equation:RK_order_2} on a Lie group (in the spirit of~\cite{Zappa18mco,Zhang19eso}) to see how it performs in high dimension. We choose the special linear group~$\SL(m)=\{M\in\R^{m\times m},\det(M)=1\}$, seen as a submanifold of~$\R^{m^2}$ of codimension~$1$.
As explained in Remark~\ref{remark:weakened_assumptions}, our analysis still applies to~$\SL(m)$ if we choose a potential~$V$ with appropriate growth assumptions, even if it is not a compact manifold.
We compare the Euler scheme~\eqref{equation:IE} and the Runge-Kutta integrator~\eqref{equation:RK_order_2} on~$\MM=\SL(m)$ for different~$m$ (that is, with the constraint~$\zeta(x)=\det(x)-1$), where we use in the implementation discrete random variables satisfying~\eqref{equation:discrete_rv}.
We choose the potential 
\begin{equation}
\label{equation:potential_SL}
V(x)=25\Trace((x-I_{m^2})^T(x-I_{m^2}))
\end{equation}
and the parameters~$\sigma=\sqrt{2}$,~$\phi(x)=\Trace(x)$ and~$M=10^6$ trajectories. Each trajectory is an approximation of the solution of equation~\eqref{equation:projected_Langevin} at time~$T=10$ with a timestep~$h=T/N$ and~$N=2^{12}$ steps.
With this timestep~$h$, the Newton method used in the Euler scheme~\eqref{equation:IE} does not converge for approximately~$0.005\%$ of the trajectories for~$m=4$. We choose to discard these trajectories, which induces a negligible bias in the expectation.
This does not occur for the Runge-Kutta integrator~\eqref{equation:RK_order_2}.
We recall that for a small enough timestep~$h$, the Newton method would always converge (see also Remark~\ref{remark:Newton_method}).
The reference solution for~$J(m)=\int_{\SL(m)} \phi(x) d\mu_\infty(x)$ is computed with the Runge-Kutta method~\eqref{equation:RK_order_2} with~$h_{\text{ref}}=2^{-14}T$.
With the factor~$25$ in the potential~\eqref{equation:potential_SL}, the solution of~\eqref{equation:projected_Langevin} stays close to~$I_{m^2}$, and~$J(m)$ is close to~$\phi(I_{m^2})=m$.
This choice of factor permits to explore a reasonably small area of~$\SL(m)$ with moderate manifold curvature.
We observe numerically that replacing the factor 25 by 1 in~\eqref{equation:potential_SL} induces a severe timestep restriction (results not included for conciseness).
The computation of~$J(m)$ could also be done via the parametrization given by the Iwasawa decomposition for~$\SL(m)$ (see, for instance,~\cite[Chap.\ts 1]{Goldfeld06afa}) and the use of standard quadrature methods, but these methods have prohibitive costs in high dimension.
We put together the numerical results in Table~\ref{Table:numerical_results_SL} and observe that the Runge-Kutta method~\eqref{equation:RK_order_2} performs significantly better than the Euler scheme~\eqref{equation:IE}.

\begin{table}[H]
	\setcellgapes{3pt}
	\centering
	\begin{tabular}{|c||c||c||c|c||c|c|}
	\hline
	$m$ &~$\dim(\SL(m))$ & \multicolumn{1}{c||}{$J(m)$} & \multicolumn{1}{c|}{$\widebar{J}_\text{Euler}$} & \multicolumn{1}{c||}{error for~$\widebar{J}_\text{Euler}$} & \multicolumn{1}{c|}{$\widebar{J}_2$} & \multicolumn{1}{c|}{error for~$\widebar{J}_2$} \\
	\hhline{|=||=||=||=|=||=|=|}
	$2$ & 3 &~$2.00967$ &~$2.01031$ &~$6.4 \cdot 10^{-4}$ &~$2.00962$ &~$4.4 \cdot 10^{-5}$ \\
	\hline
	$3$ & 8 &~$3.01954$ &~$3.02068$ &~$1.1 \cdot 10^{-3}$ &~$3.01934$ &~$2.0\cdot 10^{-4}$ \\
	\hline
	$4$ & 15 &~$4.02930$ &~$4.03095$ &~$1.6 \cdot 10^{-3}$ &~$4.02907$ &~$2.3 \cdot 10^{-4}$  \\
	\hline
	\end{tabular}
	\caption{Numerical approximation of the integral~$J(m)=\int_{\SL(m)} \phi(x) d\mu_\infty$ for~$2\leq m \leq 4$ with the estimator~$\widebar{J}=M^{-1}\sum_{k=1}^M \phi(X_N^{(k)})$ where~$(X_n)$ is given by the Euler scheme~\eqref{equation:IE} for~$\widebar{J}_\text{Euler}$ and by the Runge-Kutta integrator~\eqref{equation:RK_order_2} for~$\widebar{J}_2$, with their respective errors. The average is taken over~$M=10^6$ trajectories, with the potential~\eqref{equation:potential_SL},~$\phi(x)=\Trace(x)$, a final time~$T=10$ and a timestep~$h=2^{-12}T$.}
	\label{Table:numerical_results_SL}
	\setcellgapes{1pt}
\end{table}

\bigskip

\noindent \textbf{Acknowledgements.}\
This work was partially supported by the Swiss National Science Foundation, grants No. 200020\_184614, No. 200021\_162404 and No. 200020\_178752.
The computations were performed at the University of Geneva on the Baobab cluster using the Julia programming language.

\bibliographystyle{abbrv}
\small
\bibliography{Ma_Bibliographie}
\normalsize

\vskip-1ex
\begin{appendices}

\section{Proof of Theorem~\ref{theorem:operator_conditions_invariant_measure}}
\label{section:proof_order_invariant_measure}

In the spirit of backward error analysis for differential equations (see~\cite{Hairer06gni,Zygalakis11ote,Abdulle12hwo,Debussche12wbe}), we build a modified generator~$\LL^h$ such that~$U(x,h)=\E[\phi(X_1)|X_0=x]$ formally satisfies
\begin{equation}
\label{equation:backward_analysis}
U(x,h)=\sum_{j\geq 0} \frac{h^j}{j!} (\LL^h)^j\phi(x).
\end{equation}
Truncating this formal series yields an estimate of the form
$$
U(x,h)=\phi(x)+\sum_{j=1}^N \frac{h^j}{j!} (\LL^h)^j\phi(x)+ h^{N+1}R_N^h(\phi,x),\quad x\in N_\MM,
$$
where~$N_\MM$ is an open neighbourhood of~$\MM$ in~$\R^d$ and the remainder satisfies~$\abs{R_N^h(\phi,x)}\leq C_N(\phi)$.
For this, we write formally~$\LL^h=\LL+\sum_{n\geq 1} h^n L_n$ and compare the series expression in~\eqref{equation:dvp_num} and~\eqref{equation:backward_analysis}. By formally identifying the powers of~$h$, we deduce the following rigorous definition of the~$L_n$ on an open neighbourhood of~$\MM$ in~$\R^d$, 
\begin{equation}
\label{equation:def_L_n}
L_0=\LL,\quad L_n=\AA_n+\sum_{l=1}^n \frac{B_l}{l!}\sum_{n_1+\dots+n_{l+1}=n-l}L_{n_1}\cdots L_{n_l}\AA_{n_{l+1}},\quad n\geq 1,
\end{equation}
where the~$B_l$ are the Bernoulli numbers (see~\cite{Debussche12wbe,Zygalakis11ote,Abdulle14hon} for similar expansions in~$\T^d$ or~$\R^d$).
Using Assumption~\ref{assumption:inv_measure_Poisson}, we build recursively a sequence of functions~$(\rho_n)$ such that
\begin{equation}
\label{equation:def_mu_n}
\LL^*\rho_n=-\sum_{l=1}^n L_l^*\rho_{n-l} \quad \text{and} \quad \rho_0=\rho_{\infty},
\end{equation}
where~$\int_\MM \rho_n d\sigma_\MM=0$ for~$n\geq 1$.
We denote~$\rho_{r}^h=\sum_{n=0}^r h^n \rho_n$ and~$d\mu_r^h=\rho_{r}^h d\sigma_\MM$ and adapt on the manifold~$\MM$ the following result from~\cite[Thm.\ts 2.1]{Debussche12wbe} in the context of~$\R^d$.
\begin{lemma}
\label{lemma:Debussche_Faou_theorem}
Under Assumptions~\ref{assumption:ellipticity},~\ref{assumption:inv_measure_Poisson} and~\ref{assumption:dvp_num}, for all~$\phi\in \CC^\infty(\R^d,\R)$, for every positive integer~$r$, there exists a constant~$C_r(\phi)$ independent of~$h$ such that, for all~$h$ small enough,
\begin{equation}
\label{equation:distance_mu_h_rho_m_M}
\abs{\int_\MM \phi d\mu^h - \int_\MM \phi d\mu_r^h}\leq C_r(\phi) h^{r+1}.
\end{equation}
\end{lemma}

We omit the proof of Lemma~\ref{lemma:Debussche_Faou_theorem} as it is exactly the same as in~\cite[Thm.\ts 2.1]{Debussche12wbe} by replacing~$dx$ by~$d\sigma_\MM$ and~$\T^d$ by~$\MM$. We are now able to prove Theorem~\ref{theorem:operator_conditions_invariant_measure}.

\begin{proof}[Proof of Theorem~\ref{theorem:operator_conditions_invariant_measure}]
As~$\AA_j^*\rho_{\infty}=0$ for~$j=1,\dots,r-1$, we deduce recursively from~\eqref{equation:def_L_n} and~\eqref{equation:def_mu_n} that~$\rho_j=0$ for~$j=1,\dots,r-1$, which yields~$\rho_r^h=\rho_\infty+h^r\rho_r$. Using the definition of the error for the invariant measure~\eqref{equation:def_error_inv_measure} and the ergodicity of the integrator~\eqref{equation:ergodicity_num}, equation~\eqref{equation:distance_mu_h_rho_m_M} becomes
$$
\Bigg|e(\phi,h)-h^r\int_\MM  \phi(x) \rho_r(x) d\sigma_\MM(x)\Bigg|\leq  C h^{r+1}.
$$
We are left to prove that~$$\int_\MM  \phi(x) \rho_r(x) d\sigma_\MM(x)=\int_0^\infty \int_\MM u(x,t)\AA_r^* \rho_\infty(x) d\sigma_\MM(x) dt.$$
By the backward Kolmogorov equation and ergodicity,~$u$ satisfies
$$\lim_{T\to\infty} u(x,T)=\phi(x)+\int_0^\infty \LL u(x,t) dt=\int_\MM \phi(y) d\mu_\infty(y).$$
Using~$\LL^* \rho_r= -L_r^* \rho_\infty=-\AA_r^*\rho_\infty$, we deduce
\begin{align*}
\int_\MM  \phi(x) \rho_r(x) d\sigma_\MM(x)&=-\int_0^\infty \int_\MM  \LL u(x,t) \rho_r(x) d\sigma_\MM(x) dt
+\int_\MM \phi(y) d\mu_\infty(y) \int_\MM  \rho_r(x) d\sigma_\MM(x)\\
&=-\int_0^\infty \int_\MM  u(x,t) \LL^* \rho_r(x) d\sigma_\MM(x) dt\\
&=\int_0^\infty \int_\MM  u(x,t) \AA_r^* \rho_\infty(x) d\sigma_\MM(x) dt,
\end{align*}
where we used that~$\int_\MM  \rho_r(x) d\sigma_\MM(x)=0$. This concludes the proof of Theorem~\ref{theorem:operator_conditions_invariant_measure}.
\end{proof}

\section{Integration by parts using the tree formalism}
\label{section:IPP_details}

We provide here the detailed calculations of the integrations by parts of the order three terms, that are needed for the proof of Theorem~\ref{theorem:RK_conditions}.
After applying the operations~\eqref{equation:ex_IPP_1} and~\eqref{equation:ex_IPP_2},~$\int_\MM \AA_1\phi d\mu_\infty$ is transformed into~$\int_\MM \BB\phi d\mu_\infty$ where~$\BB$ is a differential operator of order three given by
\begin{align*}
\BB\phi&=F\Big(
	\frac{1}{4}\ \eatree2013
	-\frac{1}{4}\ \eatree4201
	+\frac{1}{2}\ \eatree5402
	-\frac{1}{4}\ \eatree3212
	+\frac{1}{4}\ \aroma2101\ \eatree4301\\
	&+\frac{1}{8}\ \aroma1101\ \eatree4301
	-\frac{3}{8}\ \aroma3301\ \eatree4301
	-\frac{1}{4}\ \aroma2101\ \eatree2111
	-\frac{1}{8}\ \aroma1101\ \eatree2111
	+\frac{1}{8}\ \aroma3301\ \eatree2111
	\Big)(\phi)+\RR\phi,
\end{align*}
and~$\RR$ is a differential operator of order two.
Using Lemma~\ref{lemma:IPP_B_series} multiple times, we get the following integrations by parts of the order three terms of~$\BB\phi$.
\begin{align*}
\frac{1}{4}\ \eatree2013
&-\frac{1}{4}\ \eatree4201
\sim
-\frac{1}{4}\ \eatree2012
+\frac{1}{4}\ \eatree4203
-\frac{1}{4}\ \aroma3301\ \eatree3101
+\frac{1}{4}\ \aroma1101\ \eatree3101
+\frac{1}{2}\ \aroma2101\ \eatree3101
-\frac{1}{2}\ \eatree3001
\\
-\frac{1}{4}\ \eatree3212
&+\frac{1}{4}\ \eatree5402
\sim
\frac{1}{4}\ \eatree3217
+\frac{1}{4}\ \eatree3218
-\frac{1}{4}\ \eatree5405
-\frac{1}{4}\ \eatree5404\\&
+\frac{3}{4}\ \aroma3301\ \eatree4302
-\frac{1}{2}\ \eatree5403
-\frac{1}{4}\ \aroma1101\ \eatree4302
-\frac{1}{2}\ \aroma2101\ \eatree4302
+\frac{1}{2}\ \eatree4202
\\
-\frac{1}{4}\ \aroma2101\ \eatree2111
&+\frac{1}{4}\ \aroma2101\ \eatree4301
\sim
\frac{1}{4}\ \aroma2101\ \eatree2112
+\frac{1}{4}\ \eatree4203
+\frac{1}{4}\ \eatree4202
-\frac{3}{4}\ \aroma2101\ \eatree4302
-\frac{1}{4}\ \aroma3202\ \eatree3201\\&
-\frac{1}{4}\ \aroma3201\ \eatree3201
+\frac{3}{4}\ \aroma2101\ \aroma3301\ \eatree3201
-\frac{1}{4}\ \aroma2101\ \aroma1101\ \eatree3201
-\frac{1}{2}\ \aroma2101\ \aroma2101\ \eatree3201
+\frac{1}{2}\ \aroma2101\ \eatree3101
\\
-\frac{1}{8}\ \aroma1101\ \eatree2111
&+\frac{1}{8}\ \aroma1101\ \eatree4301
\sim
\frac{1}{8}\ \aroma2101\ \eatree2112
+\frac{1}{8}\ \eatree3217
-\frac{1}{8}\ \aroma2201\ \eatree3201
-\frac{3}{8}\ \aroma1101\ \eatree4302\\&
+\frac{3}{8}\ \aroma1101\ \aroma3301\ \eatree3201
-\frac{1}{8}\ \aroma1101\ \aroma1101\ \eatree3201
-\frac{1}{4}\ \aroma2101\ \aroma1101\ \eatree3201
+\frac{1}{4}\ \aroma1101\ \eatree3101
\\
\frac{1}{8}\ \aroma3301\ \eatree2111
&-\frac{1}{8}\ \aroma3301\ \eatree4301
\sim
-\frac{1}{8}\ \aroma3301\ \eatree2112
-\frac{1}{4}\ \eatree5403
-\frac{1}{8}\ \eatree5404
+\frac{1}{4}\ \aroma4401\ \eatree3201
+\frac{1}{8}\ \aroma4402\ \eatree3201\\&
+\frac{5}{8}\ \aroma3301\ \eatree4302
-\frac{5}{8}\ \aroma3301\ \aroma3301\ \eatree3201
+\frac{1}{8}\ \aroma1101\ \aroma3301\ \eatree3201
+\frac{1}{4}\ \aroma2101\ \aroma3301\ \eatree3201
-\frac{1}{4}\ \aroma3301\ \eatree3101
\\
\frac{1}{4}\ \eatree5402
&-\frac{1}{4}\ \aroma3301\ \eatree4301
\sim
-\frac{1}{2}\ \eatree5403
+\frac{1}{2}\ \aroma3301\ \eatree4302
+\frac{5}{4}\ \aroma4401\ \eatree3201
+\frac{1}{4}\ \aroma4402\ \eatree3201\\&
-\frac{5}{4}\ \aroma3301\ \aroma3301\ \eatree3201
-\frac{1}{4}\ \aroma2201\ \eatree3201
-\frac{1}{4}\ \aroma2202\ \eatree3201
+\frac{1}{4}\ \aroma1101\ \aroma3301\ \eatree3201
+\frac{1}{2}\ \aroma2101\ \aroma3301\ \eatree3201
-\frac{1}{2}\ \aroma3201\ \eatree3201
\end{align*}

\section{Coefficients of the order two Runge-Kutta method}
\label{section:RK_coefficients}

The coefficients of the Runge-Kutta method~\eqref{equation:RK_order_2} used in Section~\ref{section:numerical_experiments} are
\begin{equation*}
\begin{aligned}[c]
c_2&=0.621729189582953540,\\
d_1&=-0.898931652839146019,\\
d_3&=0.318924515019668897,\\
\widehat{a}_{31}&=0.887706593835748395,\\
\widehat{a}_{41}&=0.0547449506054026516,\\
\widehat{a}_{22}&=1-\widehat{a}_{21},\\
\widehat{a}_{43}&=1-\widehat{a}_{41}-\widehat{a}_{42}.
\end{aligned}
\qquad\qquad
\begin{aligned}[c]
c_3&=0.102032386582165330,\\
d_2&=-1.66233102561284629,\\
\widehat{a}_{21}&=0.584372887990673524,\\
\widehat{a}_{32}&=-0.345018694936693742,\\
\widehat{a}_{42}&=-0.0205123070437693053,\\
\widehat{a}_{33}&=1-\widehat{a}_{31}-\widehat{a}_{32},\\
&
\end{aligned}
\end{equation*}

\newpage
\section{Decomposition of the operators in exotic aromatic forests}
\label{section:Tables_B_series}

\begin{table}[!htb]
	\setcellgapes{3pt}
    \centering
    \begin{tabularx}{\textwidth}{|c|c|c|X|}
	\hline
	$\text{Forest }  \gamma$ &~$\text{Differential }  F(\gamma)(\phi)$ & \multicolumn{1}{c|}{Exact~$e(\gamma)$} & \multicolumn{1}{c|}{Numerical approximation~$a(\gamma)$} \\
	\hhline{|=|=|=|=|}
	\multicolumn{4}{|c|}{Terms of order 4 w.r.t.~$\phi$}\\
	\hhline{|=|=|=|=|}
	$ \eatree1021$ &~$\sigma^4\Delta^2\phi$ &~$\frac{1}{8}$ &~$\frac{1}{8}$ \\
	\hline
	$ \eatree3211$ &~$\sigma^4 G^{-1} \Delta\phi''(g,g)$ &~$-\frac{1}{4}$ &~$-\frac{1}{4}$ \\
	\hline
	$ \eatree5401$ &~$\sigma^4 G^{-2} \phi^{(4)}(g,g,g,g)$ &~$\frac{1}{8}$ &~$\frac{1}{8}$ \\
	\hhline{|=|=|=|=|}
	\multicolumn{4}{|c|}{Terms of order 3 w.r.t.~$\phi$}\\
	\hhline{|=|=|=|=|}
	$ \eatree2013$ &~$\sigma^2\Delta\phi'f$ &~$\frac{1}{2}$ &~$\frac{1}{2}$ \\
	\hline
	$ \eatree4201$ &~$\sigma^2 G^{-1} \phi^{(3)}(g,g,f)$ &~$-\frac{1}{2}$ &~$-\frac{1}{2}$ \\
	\hline
	$ \eatree5402$ &~$\sigma^4 G^{-2} \phi^{(3)}(g,g,g'g)$ &~$1$ &~$1$ \\
	\hline
	$ \eatree3212$ &~$\sigma^4 G^{-1} \sum  \phi^{(3)}(g,g'e_i,e_i)$ &~$-\frac{1}{2}$ &~$-\frac{1}{2}$ \\
	\hline
	$ \aroma2101\ \eatree4301$ &~$\sigma^2 G^{-2} (g,f)\phi^{(3)}(g,g,g)$ &~$\frac{1}{2}$ &~$\frac{1}{2}$ \\
	\hline
	$ \aroma1101\ \eatree4301$ &~$\sigma^4 G^{-2} \Div(g)\phi^{(3)}(g,g,g)$ &~$\frac{1}{4}$ &~$\frac{1}{4}$ \\
	\hline
	$ \aroma3301\ \eatree4301$ &~$\sigma^4 G^{-2} (g,g'g)\phi^{(3)}(g,g,g)$ &~$-\frac{3}{4}$ &~$-\frac{3}{4}$ \\
	\hline
	$ \aroma2101\ \eatree2111$ &~$\sigma^2 G^{-1} (g,f)\Delta\phi'(g)$ &~$-\frac{1}{2}$ &~$-\frac{1}{2}$ \\
	\hline
	$ \aroma1101\ \eatree2111$ &~$\sigma^4 G^{-1} \Div(g)\Delta\phi'(g)$ &~$-\frac{1}{4}$ &~$-\frac{1}{4}$ \\
	\hline
	$ \aroma3301\ \eatree2111$ &~$\sigma^4 G^{-2} (g,g'g)\Delta\phi'(g)$ &~$\frac{1}{4}$ &~$\frac{1}{4}$ \\
	\hhline{|=|=|=|=|}
	\multicolumn{4}{|c|}{Terms of order 2 w.r.t.~$\phi$}\\
	\hhline{|=|=|=|=|}
	$\eatree3001$ &~$\phi''(f,f)$ &~$\frac{1}{2}$ &~$\frac{1}{2}$ \\
	\hline
	$ \eatree2012$ &~$\sigma^2\sum  \phi''(f'e_i,e_i)$ &~$\frac{1}{2}$ &~$b^T d$ \\
	\hline
	$ \eatree4202$ &~$\sigma^2 G^{-1} \phi''(g,g'f)$ &~$-1$ &~$-1$ \\
	\hline
	$ \eatree4203$ &~$\sigma^2 G^{-1} \phi''(g,f'g)$ &~$-1$ &~$-b^T d-\widehat{b}^T d$ \\
	\hline
	$ \eatree5403$ &~$\sigma^4 G^{-2} \phi''(g,g'g'g)$ &~$\frac{3}{2}$ &~$-2\widehat{b}^T (d\diam \widehat{A}d)-(\widehat{b}^T d)^2+2\widehat{b}^T d+1$ \\
	\hline
	\end{tabularx}
	\renewcommand\thetable{2 (Part 1/7)}
	\caption{Coefficients in exotic aromatic B-series of the operators~$\LL^2\phi/2=\sum e(\gamma)   F(\gamma)(\phi)$ and~$\AA_1\phi=\sum a(\gamma)   F(\gamma)(\phi)$ for consistent Runge-Kutta methods of the form~\eqref{equation:defRK}.}
	\setcellgapes{1pt}
\end{table}

\begin{table}[!htb]
	\setcellgapes{3pt}
    \centering
    \begin{tabularx}{\textwidth}{|c|c|c|X|}
	\hline
	$  \gamma$ &~$  F(\gamma)(\phi)$ & \multicolumn{1}{c|}{$e(\gamma)$} & \multicolumn{1}{c|}{$a(\gamma)$} \\
	\hhline{|=|=|=|=|}
	$ \eatree5405$ &~$\sigma^2 G^{-2} \phi''(g'g,g'g)$ &~$\frac{1}{4}$ &~$-\widehat{b}^T (d\diam \widehat{A}d)-\frac{1}{2}(\widehat{b}^T d)^2+\widehat{b}^T d$ \\
	\hline
	$ \eatree5404$ &~$\sigma^4 G^{-2} \phi''(g,g''(g,g))$ &~$\frac{1}{2}$ &~$2\widehat{b}^T d^{\diam 2}-2\widehat{b}^T (\delta \diam d^{\diam 2})+\frac{1}{2}$ \\
	\hline
	$ \eatree3217$ &~$\sigma^4 G^{-1} \phi''(g,\Delta g)$ &~$-\frac{1}{2}$ &~$-\frac{1}{2}$ \\
	\hline
	$ \eatree3218$ &~$\sigma^4 G^{-1} \sum  \phi''(g'e_i,g'e_i)$ &~$-\frac{1}{4}$ &~$\widehat{b}^T (d\diam \widehat{A}d)+\frac{1}{2}(\widehat{b}^T d)^2-\widehat{b}^T d$ \\
	\hline
	$ \eatree3213$ &~$\sigma^4 G^{-1} \sum \phi''(g''(g,e_i),e_i)$ &~$0$ &~$\widehat{b}^T (\delta \diam d^{\diam 2})-\widehat{b}^T d^{\diam 2}$ \\
	\hline
	$\aroma2101\ \eatree3101$ &~$ G^{-1} (g,f)\phi''(g,f)$ &~$-1$ &~$-1$ \\
	\hline
	$ \aroma1101\ \eatree3101$ &~$\sigma^2 G^{-1} \Div(g) \phi''(g,f)$ &~$-\frac{1}{2}$ &~$-\frac{1}{2}$ \\
	\hline
	$ \aroma3301\ \eatree3101$ &~$\sigma^2 G^{-2} (g,g'g)\phi''(g,f)$ &~$\frac{1}{2}$ &~$\frac{1}{2}$ \\
	\hline
	$ \aroma2101\ \eatree2112$ &~$\sigma^2 G^{-1} (g,f)\sum \phi''(e_i,g'e_i)$ &~$-\frac{1}{2}$ &~$-\widehat{b}^T d$ \\
	\hline
	$ \aroma1101\ \eatree2112$ &~$\sigma^4 G^{-1} \Div(g)\sum \phi''(e_i,g'e_i)$ &~$-\frac{1}{4}$ &~$-\frac{1}{2}\widehat{b}^T d$ \\
	\hline
	$ \aroma3301\ \eatree2112$ &~$\sigma^4 G^{-2} (g,g'g)\sum \phi''(e_i,g'e_i)$ &~$\frac{1}{4}$ &~$\frac{1}{2}\widehat{b}^T d$ \\
	\hline
	$ \aroma2101\ \eatree4302$ &~$\sigma^2 G^{-2} (g,f)\phi''(g,g'g)$ &~$2$ &~$2\widehat{b}^T d+1$ \\
	\hline
	$ \aroma1101\ \eatree4302$ &~$\sigma^4 G^{-2} \Div(g)\phi''(g,g'g)$ &~$1$ &~$\widehat{b}^T d+\frac{1}{2}$ \\
	\hline
	$ \aroma3301\ \eatree4302$ &~$\sigma^4 G^{-3} (g,g'g)\phi''(g,g'g)$ &~$-\frac{5}{2}$ &~$2\widehat{b}^T (d\diam \widehat{A}d)+(\widehat{b}^T d)^2-3\widehat{b}^T d-\frac{3}{2}$ \\
	\hline
	$ \aroma3201\ \eatree3201$ &~$\sigma^2 G^{-2} (g,g'f) \phi''(g,g)$ &~$1$ &~$1$ \\
	\hline
	$ \aroma3202\ \eatree3201$ &~$\sigma^2 G^{-2} (g,f'g) \phi''(g,g)$ &~$\frac{1}{2}$ &~$b^T d$ \\
	\hline
	$ \aroma2201\ \eatree3201$ &~$\sigma^4 G^{-2} \sum \partial_{ij} g_i g_j \phi''(g,g)$ &~$\frac{1}{2}$ &~$\frac{1}{2}$ \\
	\hline
	$ \aroma2202\ \eatree3201$ &~$\sigma^4 G^{-2} \sum \partial_j g_i \partial_i g_j \phi''(g,g)$ &~$\frac{1}{4}$ &~$\frac{1}{4}$ \\
	\hline
	$ \aroma4401\ \eatree3201$ &~$\sigma^4 G^{-3} (g,g'g'g) \phi''(g,g)$ &~$-\frac{7}{4}$ &~$\widehat{b}^T (d\diam \widehat{A}d)+\frac{1}{2}(\widehat{b}^T d)^2-\widehat{b}^T d-\frac{3}{2}$ \\
	\hline
	$ \aroma4402\ \eatree3201$ &~$\sigma^4 G^{-3} (g,g''(g,g)) \phi''(g,g)$ &~$-\frac{1}{2}$ &~$\widehat{b}^T (\delta \diam d^{\diam 2})-\widehat{b}^T d^{\diam 2}-\frac{1}{2}$ \\
	\hline
	$\aroma2101\ \aroma2101\ \eatree3201$ &~$ G^{-2} (g,f)^2 \phi''(g,g)$ &~$\frac{1}{2}$ &~$\frac{1}{2}$ \\
	\hline
	$ \aroma2101\ \aroma1101\ \eatree3201$ &~$\sigma^2 G^{-2} (g,f) \Div(g)  \phi''(g,g)$ &~$\frac{1}{2}$ &~$\frac{1}{2}$ \\
	\hline
	$ \aroma2101\ \aroma3301\ \eatree3201$ &~$\sigma^2 G^{-3} (g,f) (g,g'g) \phi''(g,g)$ &~$-2$ &~$-\widehat{b}^T d-\frac{3}{2}$ \\
	\hline
	\end{tabularx}
	\renewcommand\thetable{2 (Part 2/7)}
	\caption{Coefficients in exotic aromatic B-series of the operators~$\LL^2\phi/2=\sum e(\gamma)   F(\gamma)(\phi)$ and~$\AA_1\phi=\sum a(\gamma)   F(\gamma)(\phi)$ for consistent Runge-Kutta methods of the form~\eqref{equation:defRK}.}
	\setcellgapes{1pt}
\end{table}

\begin{table}[!htb]
	\setcellgapes{3pt}
	\centering
	\begin{tabularx}{\textwidth}{|c|c|X|}
	\hline
	$  \gamma$ & \multicolumn{1}{c|}{$e(\gamma)$} & \multicolumn{1}{c|}{$a(\gamma)$} \\
	\hhline{|=|=|=|}
	$ \aroma1101\ \aroma1101\ \eatree3201$ &~$\frac{1}{8}$ &~$\frac{1}{8}$ \\
	\hline
	$ \aroma1101\ \aroma3301\ \eatree3201$ &~$-1$ &~$-\frac{1}{2}\widehat{b}^T d-\frac{3}{4}$ \\
	\hline
	$ \aroma3301\ \aroma3301\ \eatree3201$ &~$\frac{19}{8}$ &~$-\widehat{b}^T (d\diam \widehat{A}d)-\frac{1}{2}(\widehat{b}^T d)^2+\frac{3}{2}\widehat{b}^T d+\frac{15}{8}$ \\
	\hhline{|=|=|=|}
	\multicolumn{3}{|c|}{Terms of order 1 w.r.t.~$\phi$}\\
	\hhline{|=|=|=|}
	$\eatree3002$ &~$\frac{1}{2}$ &~$b^T c$ \\
	\hline
	$ \eatree2011$ &~$\frac{1}{4}$ &~$\frac{1}{2} b^T d^{\diam 2}$ \\
	\hline
	$ \eatree4204$ &~$-\frac{1}{2}$ &~$\widehat{b}^T (d \diam \widehat{A}c) -\widehat{b}^T d$ \\
	\hline
	$ \eatree4205$ &~$-\frac{1}{2}$ &~$\widehat{b}^T A((\delta-\ind)\diam d) -\widehat{b}^T db^T d$ \\
	\hline
	$ \eatree4206$ &~$0$ &~$b^T (d \diam \widehat{A}((\delta-\ind)\diam d))$ \\
	\hline
	$ \eatree4207$ &~$-\frac{1}{4}$ &~$-\frac{1}{2}b^T (\delta \diam d^{\diam 2})$ \\
	\hline
	$ \eatree4208$ &~$0$ &~$\widehat{b}^T (\delta \diam c \diam d) -\widehat{b}^T (c \diam d)$ \\
	\hline
	$ \eatree3214$ &~$-\frac{1}{4}$ &~$\frac{1}{2}\widehat{b}^T (d \diam \widehat{A}d^{\diam 2})-\frac{1}{2}\widehat{b}^T d$ \\
	\hline
	 &  &~$~$ \\[-3ex]
	 &  &~$-\widehat{b}^T (d \diam (\widehat{A} d)^{\diam 2})-3 \widehat{b}^T (d \diam \widehat{A} (d \diam \widehat{A} ((\ind-\delta)\diam d)))~$ \\[-5ex]
	$ \eatree5406$ &~$\frac{1}{2}$ &~$-2\widehat{b}^T d\widehat{b}^T (d \diam \widehat{A}d)+(\widehat{b}^T d)^2+\widehat{b}^T d$ \\
	\hline
	$ \eatree5407$ &~$\frac{1}{4}$ &~$\widehat{b}^T (d \diam \widehat{A}d^{\diam 2})-\frac{3}{2}\widehat{b}^T (d \diam \widehat{A}(\delta \diam d^{\diam 2})) +\widehat{b}^T d(\widehat{b}^T d^{\diam 2}-\widehat{b}^T (\delta \diam d^{\diam 2}))+\frac{1}{2}\widehat{b}^T d$ \\
	\hline
	$ \eatree5408$ &~$0$ &~$-2\widehat{b}^T (d^{\diam 2}\diam \widehat{A}d)
	+(\widehat{b}^T d+1)\widehat{b}^T d^{\diam 2}-\widehat{b}^T d \widehat{b}^T (\delta \diam d^{\diam 2})$ \\
	\hline
	$ \eatree3215$ &~$0$ &~$\widehat{b}^T (d^{\diam 2}\diam \widehat{A}d)-\frac{1}{2}\widehat{b}^T d^{\diam 2}$ \\
	\hline
	$ \eatree3216$ &~$0$ &~$\frac{1}{2}\widehat{b}^T (\delta \diam d^{\diam 3})-\frac{1}{2}\widehat{b}^T d^{\diam 3}$ \\
	\hline
	$\aroma2101\ \eatree3102$ &~$-\frac{1}{2}$ &~$-\widehat{b}^T c$ \\
	\hline
	\end{tabularx}
	\renewcommand\thetable{2 (Part 3/7)}
	\caption{Coefficients in exotic aromatic B-series of the operators~$\LL^2\phi/2=\sum e(\gamma)   F(\gamma)(\phi)$ and~$\AA_1\phi=\sum a(\gamma)   F(\gamma)(\phi)$ for consistent Runge-Kutta methods of the form~\eqref{equation:defRK}.}
	\setcellgapes{1pt}
\end{table}

\begin{table}[!htb]
	\setcellgapes{3pt}
	\centering
	\begin{tabularx}{\textwidth}{|c|c|X|}
	\hline
	$  \gamma$ & \multicolumn{1}{c|}{$e(\gamma)$} & \multicolumn{1}{c|}{$a(\gamma)$} \\
	\hhline{|=|=|=|}
	$ \aroma1101\ \eatree3102$ &~$-\frac{1}{4}$ &~$-\frac{1}{2} \widehat{b}^T c$ \\
	\hline
	$ \aroma3301\ \eatree3102$ &~$\frac{1}{4}$ &~$\frac{1}{2} \widehat{b}^T c$ \\
	\hline
	$\aroma2101\ \eatree3103$ &~$-\frac{1}{2}$ &~$-b^T (\delta \diam c)$ \\
	\hline
	$ \aroma1101\ \eatree3103$ &~$-\frac{1}{4}$ &~$-\frac{1}{2} b^T (\delta \diam d^{\diam 2})$ \\
	\hline
	$ \aroma3301\ \eatree3103$ &~$\frac{1}{4}$ &~$b^T (d \diam \widehat{A}((\ind-\delta)\diam d)))+\frac{1}{2}b^T (\delta \diam d^{\diam 2})$ \\
	\hline
	 &  &~$\widehat{b}^T (c \diam \widehat{A}((\ind-\delta)\diam d))+(\widehat{b}^T d)^2+\widehat{b}^T d$ \\[-4ex]
	$\aroma2101\ \eatree4303$ &~$1$ &~$+\widehat{b}^T (d \diam \widehat{A}((\ind-\delta)\diam d))-\widehat{b}^T (d \diam \widehat{A}(\delta \diam c))$ \\
	\hline
	 &  &~$\frac{1}{2}\widehat{b}^T (d \diam \widehat{A}((\ind-\delta) \diam d))+\frac{1}{2}(\widehat{b}^T d)^2+\frac{1}{2}\widehat{b}^T d$ \\[-4ex]
	$ \aroma1101\ \eatree4303$ &~$\frac{1}{2}$ &~$+\frac{1}{2}\widehat{b}^T (d^{\diam 2} \diam \widehat{A}((\ind-\delta) \diam d))-\frac{1}{2}\widehat{b}^T (d \diam \widehat{A}(\delta \diam d^{\diam 2}))$ \\
	\hline
	 &  &~$3 \widehat{b}^T (d \diam \widehat{A} (d \diam \widehat{A} ((\ind-\delta)\diam d)))
	 +(2\widehat{b}^T d-\frac{1}{2})\widehat{b}^T (d \diam \widehat{A}d)$ \\
	 &  &~$-3\widehat{b}^T (d \diam (\widehat{A}((\ind-\delta) \diam d))^{\diam 2})
	+\frac{1}{2}\widehat{b}^T (d \diam \widehat{A}(\delta \diam d))
	+\widehat{b}^T (d \diam (\widehat{A} d)^{\diam 2})$ \\[-4ex]
	$ \aroma3301\ \eatree4303$ &~$-1$ &~$
	+\frac{1}{2}\widehat{b}^T (d^{\diam 2} \diam \widehat{A}((\delta-\ind) \diam d))
	+\frac{1}{2}\widehat{b}^T (d \diam \widehat{A}(\delta \diam d^{\diam 2}))
	-\frac{3}{2}(\widehat{b}^T d)^2
	-\frac{3}{2}\widehat{b}^T d
	$ \\
	\hline
	$ \aroma2101\ \eatree2113$ &~$-\frac{1}{4}$ &~$-\frac{1}{2}\widehat{b}^T d^{\diam 2}$ \\
	\hline
	$ \aroma1101\ \eatree2113$ &~$-\frac{1}{8}$ &~$-\frac{1}{4}\widehat{b}^T d^{\diam 2}$ \\
	\hline
	$ \aroma3301\ \eatree2113$ &~$\frac{1}{8}$ &~$\frac{1}{4}\widehat{b}^T d^{\diam 2}$ \\
	\hline
	$ \aroma2101\ \eatree4304$ &~$\frac{1}{4}$ &~$\frac{1}{2}\widehat{b}^T (\delta \diam d^{\diam 2})$ \\
	\hline
	$ \aroma1101\ \eatree4304$ &~$\frac{1}{8}$ &~$\frac{1}{4}\widehat{b}^T (\delta \diam d^{\diam 2})$ \\
	\hline
	$ \aroma3301\ \eatree4304$ &~$-\frac{1}{8}$ &~$\widehat{b}^T (d^{\diam 2}\diam \widehat{A}d)-(\widehat{b}^T d+\frac{1}{2})\widehat{b}^T d^{\diam 2}+(\widehat{b}^T d-\frac{1}{4})\widehat{b}^T (\delta \diam d^{\diam 2})$ \\
	\hline
	$ \aroma3201\ \eatree3202$ &~$\frac{1}{2}$ &~$-\widehat{b}^T (d \diam \widehat{A}c)+\widehat{b}^T d$ \\
	\hline
	$ \aroma3202\ \eatree3202$ &~$\frac{1}{2}$ &~$\widehat{b}^T (\delta \diam A((\ind-\delta)\diam d))+\widehat{b}^T d b^T d$ \\
	\hline
	$ \aroma2201\ \eatree3202$ &~$\frac{1}{4}$ &~$-\frac{1}{2}\widehat{b}^T (d \diam \widehat{A}d^{\diam 2})+\frac{1}{2}\widehat{b}^T d$ \\
	\hline
	$ \aroma2202\ \eatree3202$ &~$0$ &~$-\widehat{b}^T (d^{\diam 2}\diam \widehat{A}d)+\frac{1}{2}\widehat{b}^T (\delta \diam d^{\diam 2})$ \\
	\hline
	\end{tabularx}
	\renewcommand\thetable{2 (Part 4/7)}
	\caption{Coefficients in exotic aromatic B-series of the operators~$\LL^2\phi/2=\sum e(\gamma)   F(\gamma)(\phi)$ and~$\AA_1\phi=\sum a(\gamma)   F(\gamma)(\phi)$ for consistent Runge-Kutta methods of the form~\eqref{equation:defRK}.}
	\setcellgapes{1pt}
\end{table}

\begin{table}[!htb]
	\setcellgapes{3pt}
	\centering
	\begin{tabularx}{\textwidth}{|c|c|X|}
	\hline
	$  \gamma$ & \multicolumn{1}{c|}{$e(\gamma)$} & \multicolumn{1}{c|}{$a(\gamma)$} \\
	\hhline{|=|=|=|}
	 &  &~$\widehat{b}^T (d \diam (\widehat{A} d)^{\diam 2})+3 \widehat{b}^T (d \diam \widehat{A} (d \diam \widehat{A} ((\ind-\delta)\diam d)))$ \\[-2ex]
	$ \aroma4401\ \eatree3202$ &~$-\frac{1}{2}$ &~$
	+2\widehat{b}^T (d^{\diam 2}\diam \widehat{A}d)
	+2\widehat{b}^T d \widehat{b}^T (d \diam \widehat{A}d)
	-\widehat{b}^T (\delta \diam d^{\diam 2})
	-(\widehat{b}^T d)^2
	-\widehat{b}^T d
	$ \\
	\hline
	$ \aroma4402\ \eatree3202$ &~$-\frac{1}{4}$ &~$\frac{1}{2}\widehat{b}^T (d \diam \widehat{A}((3\delta-2\cdot\ind) \diam d^{\diam 2}))+\widehat{b}^T d(\widehat{b}^T (\delta \diam d^{\diam 2})-\widehat{b}^T d^{\diam 2})-\frac{1}{2}\widehat{b}^T d$ \\
	\hline
	$\aroma2101\ \aroma2101\ \eatree3202$ &~$\frac{1}{2}$ &~$\widehat{b}^T (\delta \diam c)$ \\
	\hline
	$ \aroma2101\ \aroma1101\ \eatree3202$ &~$\frac{1}{2}$ &~$\frac{1}{2}\widehat{b}^T (\delta \diam d^{\diam 2})+\frac{1}{2}\widehat{b}^T (\delta \diam c)$ \\
	\hline
	 &  &~$\widehat{b}^T (c \diam \widehat{A}((\delta-\ind) \diam d))+\widehat{b}^T (d \diam \widehat{A}((\delta-\ind) \diam d))$ \\[-2ex]
	$ \aroma2101\ \aroma3301\ \eatree3202$ &~$-\frac{3}{2}$ &~$+\widehat{b}^T (d \diam \widehat{A}(\delta \diam c))-\frac{1}{2}\widehat{b}^T (\delta \diam d^{\diam 2})-\frac{1}{2}\widehat{b}^T (\delta \diam c)-(\widehat{b}^T d)^2-\widehat{b}^T d$ \\
	\hline
	$ \aroma1101\ \aroma1101\ \eatree3202$ &~$\frac{1}{8}$ &~$\frac{1}{4}\widehat{b}^T (\delta \diam d^{\diam 2})$ \\
	\hline
	 &  &~$\frac{1}{2}\widehat{b}^T (d \diam \widehat{A}((\delta-\ind) \diam d))+\frac{1}{2}\widehat{b}^T (d \diam \widehat{A}(\delta \diam d^{\diam 2}))$ \\[-2ex]
	$ \aroma1101\ \aroma3301\ \eatree3202$ &~$-\frac{3}{4}$ &~$+\frac{1}{2}\widehat{b}^T (d^{\diam 2} \diam \widehat{A}((\delta-\ind) \diam d))-\frac{1}{2}\widehat{b}^T (\delta \diam d^{\diam 2})-\frac{1}{2}(\widehat{b}^T d)^2-\frac{1}{2}\widehat{b}^T d$ \\
	\hline
	 &  &~$3\widehat{b}^T (d \diam (\widehat{A}((\ind-\delta) \diam d))^{\diam 2})
	 -\frac{1}{2}\widehat{b}^T (d^{\diam 2} \diam \widehat{A}((\delta+\ind) \diam d))
	$ \\
	 &  &~$
	 -\widehat{b}^T (d \diam (\widehat{A} d)^{\diam 2})-3 \widehat{b}^T (d \diam \widehat{A} (d \diam \widehat{A} ((\ind-\delta)\diam d)))
	$ \\
	 &  &~$
	 +(\frac{1}{2}-2\widehat{b}^T d)\widehat{b}^T (d \diam \widehat{A}d)
	-\frac{1}{2}\widehat{b}^T (d \diam \widehat{A}(\delta \diam d))
	-\frac{1}{2}\widehat{b}^T (d \diam \widehat{A}(\delta \diam d^{\diam 2}))
	$ \\[-2ex]
	$ \aroma3301\ \aroma3301\ \eatree3202$ &~$\frac{9}{8}$ &~$
	+\frac{9}{2}\widehat{b}^T (\delta \diam d^{\diam 3})
	 -\frac{15}{4}\widehat{b}^T (\delta \diam d^{\diam 2})
	 +\frac{3}{2}(\widehat{b}^T d)^2
	 +\frac{3}{2}\widehat{b}^T d
	$ \\
	\hline
	$\aroma3101\ \eatree2101$ &~$-\frac{1}{2}$ &~$-b^T c~$ \\
	\hline
	$\aroma3102\ \eatree2101$ &~$-\frac{1}{2}$ &~$-\frac{1}{2}~$ \\
	\hline
	$ \aroma2102\ \eatree2101$ &~$-\frac{1}{2}$ &~$-\frac{1}{2}~$ \\
	\hline
	$ \aroma2103\ \eatree2101$ &~$-\frac{1}{4}$ &~$-\frac{1}{2}b^T d^{\diam 2}~$ \\
	\hline
	$ \aroma2104\ \eatree2101$ &~$-\frac{1}{2}$ &~$-b^T d~$ \\
	\hline
	$ \aroma4301\ \eatree2101$ &~$\frac{3}{2}$ &~$-\widehat{b}^T (d \diam \widehat{A}c) + \widehat{b}^T d+1$ \\
	\hline
	$ \aroma4302\ \eatree2101$ &~$\frac{3}{2}$ &~$b^T (d \diam \widehat{A}((\ind-\delta)\diam d))+\widehat{b}^T A((\ind-\delta)\diam d)+(\widehat{b}^T d+2)b^T d$ \\
	\hline
	$ \aroma4303\ \eatree2101$ &~$\frac{1}{2}$ &~$\widehat{b}^T (c \diam d)-\widehat{b}^T (\delta\diam c \diam d)+\frac{1}{2}$ \\
	\hline
	$ \aroma4304\ \eatree2101$ &~$\frac{1}{4}$ &~$\frac{1}{2}b^T (\delta \diam d^{\diam 2})$ \\
	\hline
	$ \aroma3302\ \eatree2101$ &~$\frac{3}{4}$ &~$-\frac{1}{2}\widehat{b}^T (d \diam \widehat{A}d^{\diam 2})+\frac{1}{2}\widehat{b}^T d+\frac{1}{2}$ \\
	\hline
	\end{tabularx}
	\renewcommand\thetable{2 (Part 5/7)}
	\caption{Coefficients in exotic aromatic B-series of the operators~$\LL^2\phi/2=\sum e(\gamma)   F(\gamma)(\phi)$ and~$\AA_1\phi=\sum a(\gamma)   F(\gamma)(\phi)$ for consistent Runge-Kutta methods of the form~\eqref{equation:defRK}.}
	\setcellgapes{1pt}
\end{table}

\begin{table}[!htb]
	\setcellgapes{3pt}
	\centering
	\begin{tabularx}{\textwidth}{|c|c|X|}
	\hline
	$  \gamma$ & \multicolumn{1}{c|}{$e(\gamma)$} & \multicolumn{1}{c|}{$a(\gamma)$} \\
	\hhline{|=|=|=|}
	$ \aroma3303\ \eatree2101$ &~$\frac{1}{4}$ &~$\frac{1}{2}\widehat{b}^T d^{\diam 3}-\frac{1}{2}\widehat{b}^T (\delta \diam d^{\diam 3})+\frac{1}{4}$ \\
	\hline
	$ \aroma3304\ \eatree2101$ &~$\frac{1}{2}$ &~$-\widehat{b}^T (d^{\diam 2} \diam \widehat{A}d)+\frac{3}{2}\widehat{b}^T d^{\diam 2}-\widehat{b}^T (\delta \diam d^{\diam 2})+\frac{1}{2}$ \\
	\hline
	$ \aroma3305\ \eatree2101$ &~$\frac{1}{4}$ &~$-\widehat{b}^T (d \diam \widehat{A}d)-\frac{1}{2}(\widehat{b}^T d)^2+\widehat{b}^T d$ \\
	\hline
	 &  &~$\widehat{b}^T (d \diam (\widehat{A} d)^{\diam 2})+3 \widehat{b}^T (d \diam \widehat{A} (d \diam \widehat{A} ((\ind-\delta)\diam d)))$ \\[-4ex]
	$ \aroma5501\ \eatree2101$ &~$-\frac{9}{4}$ &~$
	+(2\widehat{b}^T d+3)\widehat{b}^T (d \diam \widehat{A}d)
	+\frac{1}{2}(\widehat{b}^T d)^2
	-4\widehat{b}^T d
	-1
	$ \\
	\hline
	 &  &~$\frac{1}{2}\widehat{b}^T (d \diam \widehat{A}((3\delta-2\cdot\ind)\diam d^{\diam 2}))+2\widehat{b}^T (d^{\diam 2}\diam \widehat{A}d)$ \\[-2ex]
	$ \aroma5502\ \eatree2101$ &~$-\frac{7}{4}$ &~$
	-(2\widehat{b}^T d+3)\widehat{b}^T d^{\diam 2}
	+2(\widehat{b}^T d+1)\widehat{b}^T (\delta \diam d^{\diam 2})
	-\frac{1}{2}\widehat{b}^T d
	-\frac{3}{2}
	$ \\
	\hline
	$ \aroma5503\ \eatree2101$ &~$-\frac{1}{8}$ &~$-\frac{1}{8}$ \\
	\hline
	$ \aroma1111\ \eatree2101$ &~$-\frac{1}{8}$ &~$-\frac{1}{8}$ \\
	\hline
	$\aroma2101\ \aroma3201\ \eatree2101$ &~$\frac{3}{2}$ &~$\widehat{b}^T c+1$ \\
	\hline
	$ \aroma1101\ \aroma3201\ \eatree2101$ &~$\frac{3}{4}$ &~$\frac{1}{2}\widehat{b}^T c+\frac{1}{2}$ \\
	\hline
	$ \aroma3301\ \aroma3201\ \eatree2101$ &~$-\frac{9}{4}$ &~$
	\widehat{b}^T (d \diam \widehat{A}c)
	-\frac{1}{2}\widehat{b}^T c
	-\widehat{b}^T d
	-\frac{3}{2}
	$ \\
	\hline
	$\aroma2101\ \aroma3202\ \eatree2101$ &~$\frac{1}{2}$ &~$b^T (\delta \diam c)$ \\
	\hline
	$ \aroma1101\ \aroma3202\ \eatree2101$ &~$\frac{1}{4}$ &~$\frac{1}{2}b^T (\delta \diam d^{\diam 2})$ \\
	\hline
	 &  &~$
	 \widehat{b}^T (\delta \diam A((\delta-\ind) \diam d))
	~$ \\[-0.5ex]
	$ \aroma3301\ \aroma3202\ \eatree2101$ &~$-\frac{5}{4}$ &~$
	 +b^T (d \diam \widehat{A}((\delta-\ind) \diam d))
	-\frac{1}{2}b^T (\delta \diam d^{\diam 2})
	-(\widehat{b}^T d+1)b^T d
	$ \\
	\hline
	$ \aroma2101\ \aroma2201\ \eatree2101$ &~$\frac{3}{4}$ &~$\frac{1}{2}\widehat{b}^T d^{\diam 2}+\frac{1}{2}$ \\
	\hline
	$ \aroma1101\ \aroma2201\ \eatree2101$ &~$\frac{3}{8}$ &~$\frac{1}{4}\widehat{b}^T d^{\diam 2}+\frac{1}{4}$ \\
	\hline
	$ \aroma3301\ \aroma2201\ \eatree2101$ &~$-\frac{9}{8}$ &~$\frac{1}{2}\widehat{b}^T (d \diam \widehat{A}d^{\diam 2})-\frac{1}{4}\widehat{b}^T d^{\diam 2}-\frac{1}{2}\widehat{b}^T d-\frac{3}{4}$ \\
	\hline
	$ \aroma2101\ \aroma2202\ \eatree2101$ &~$\frac{1}{2}$ &~$\widehat{b}^T d$ \\
	\hline
	$ \aroma1101\ \aroma2202\ \eatree2101$ &~$\frac{1}{4}$ &~$\frac{1}{2}\widehat{b}^T d$ \\
	\hline
	$ \aroma3301\ \aroma2202\ \eatree2101$ &~$-\frac{1}{2}$ &~$\widehat{b}^T (d^{\diam 2}\diam \widehat{A}d)-\frac{1}{2}\widehat{b}^T (\delta \diam d^{\diam 2})-\frac{1}{2}\widehat{b}^T d-\frac{1}{4}$ \\
	\hline
	 &  &~$\widehat{b}^T (c \diam \widehat{A}((\delta-\ind) \diam d))+\widehat{b}^T (d \diam \widehat{A}(\delta \diam c))$ \\[-2ex]
	$ \aroma2101\ \aroma4401\ \eatree2101$ &~$-3$ &~$
	+\widehat{b}^T (d \diam \widehat{A}((\delta-\ind) \diam d))
	-(\widehat{b}^T d)^2
	-3\widehat{b}^T d
	-1
	$ \\
	\hline
	\end{tabularx}
	\renewcommand\thetable{2 (Part 6/7)}
	\caption{Coefficients in exotic aromatic B-series of the operators~$\LL^2\phi/2=\sum e(\gamma)   F(\gamma)(\phi)$ and~$\AA_1\phi=\sum a(\gamma)   F(\gamma)(\phi)$ for consistent Runge-Kutta methods of the form~\eqref{equation:defRK}.}
	\setcellgapes{1pt}
\end{table}

\begin{table}[!htb]
	\setcellgapes{3pt}
	\centering
	\begin{tabularx}{\textwidth}{|c|c|X|}
	\hline
	$  \gamma$ & \multicolumn{1}{c|}{$e(\gamma)$} & \multicolumn{1}{c|}{$a(\gamma)$} \\
	\hhline{|=|=|=|}
	 &  &~$\frac{1}{2}\widehat{b}^T (d^{\diam 2} \diam \widehat{A}((\delta-\ind) \diam d))+\frac{1}{2}\widehat{b}^T (d \diam \widehat{A}(\delta\diam d^{\diam 2}))$ \\[-2ex]
	$ \aroma1101\ \aroma4401\ \eatree2101$ &~$-\frac{3}{2}$ &~$
	+\frac{1}{2}\widehat{b}^T (d \diam \widehat{A}((\delta-\ind) \diam d))
	-\frac{1}{2}(\widehat{b}^T d)^2
	-\frac{3}{2}\widehat{b}^T d
	-\frac{1}{2}
	$ \\
	\hline
	 &  &~$3\widehat{b}^T (d \diam (\widehat{A}((\ind-\delta) \diam d))^{\diam 2})
	-\frac{1}{2}\widehat{b}^T (d^{\diam 2} \diam \widehat{A}((3\cdot\ind+\delta) \diam d))
	$ \\
	 &  &~$-2\widehat{b}^T (d \diam (\widehat{A} d)^{\diam 2})-6 \widehat{b}^T (d \diam \widehat{A} (d \diam \widehat{A} ((\ind-\delta)\diam d)))
	$ \\
	 &  &~$
	-\frac{1}{2}\widehat{b}^T (d \diam \widehat{A}(\delta \diam d^{\diam 2}))
	-\frac{1}{2}\widehat{b}^T (d \diam \widehat{A}(\delta \diam d))
	$ \\[-2ex]
	$ \aroma3301\ \aroma4401\ \eatree2101$ &~$\frac{23}{4}$ &~$
	-(4\widehat{b}^T d+\frac{5}{2})\widehat{b}^T (d \diam \widehat{A}d)
	+\widehat{b}^T (\delta \diam d^{\diam 2})
	+(\widehat{b}^T d)^2
	+\frac{13}{2}\widehat{b}^T d
	+3
	$ \\
	\hline
	$ \aroma2101\ \aroma4402\ \eatree2101$ &~$-\frac{3}{4}$ &~$-\frac{1}{2}\widehat{b}^T (\delta \diam d^{\diam 2})-\frac{1}{2}$ \\
	\hline
	$ \aroma1101\ \aroma4402\ \eatree2101$ &~$-\frac{3}{8}$ &~$-\frac{1}{4}\widehat{b}^T (\delta \diam d^{\diam 2})-\frac{1}{4}$ \\
	\hline
	 &  &~$\frac{1}{2}\widehat{b}^T (d \diam \widehat{A}((2\cdot\ind -3 \delta) \diam d^{\diam 2}))
	-\widehat{b}^T (d^{\diam 2}\diam \widehat{A}d)$ \\[-0.5ex]
	$ \aroma3301\ \aroma4402\ \eatree2101$ &~$\frac{13}{8}$ &~$
	+(2\widehat{b}^T d+\frac{3}{2})\widehat{b}^T d^{\diam 2}
	-(2\widehat{b}^T d+\frac{3}{4})\widehat{b}^T (\delta \diam d^{\diam 2})
	+\frac{1}{2}\widehat{b}^T d
	+\frac{5}{4}$ \\
	\hline
	$\aroma2101\ \aroma2101\ \aroma3301\ \eatree2101$ &~$-1$ &~$-\widehat{b}^T (\delta \diam c)-\frac{1}{2}$ \\
	\hline
	$ \aroma2101\ \aroma1101\ \aroma3301\ \eatree2101$ &~$-1$ &~$-\frac{1}{2}\widehat{b}^T (\delta \diam c)-\frac{1}{2}\widehat{b}^T (\delta \diam d^{\diam 2})-\frac{1}{2}$ \\
	\hline
	 &  &~$
	\widehat{b}^T (c \diam \widehat{A}((\ind-\delta)\diam d))
	+\widehat{b}^T (d \diam \widehat{A}((\ind-\delta)\diam d))
	$ \\[-0.5ex]
	$ \aroma2101\ \aroma3301\ \aroma3301\ \eatree2101$ &~$\frac{7}{2}$ &~$
	-\widehat{b}^T (d \diam \widehat{A}(\delta \diam c))
	+\frac{1}{2}\widehat{b}^T (\delta \diam c)
	+\frac{1}{2}\widehat{b}^T (\delta \diam d^{\diam 2})
	+(\widehat{b}^T d)^2
	+2\widehat{b}^T d
	+\frac{3}{2}
	$ \\
	\hline
	$ \aroma1101\ \aroma1101\ \aroma3301\ \eatree2101$ &~$-\frac{1}{4}$ &~$-\frac{1}{4}\widehat{b}^T (\delta \diam d^{\diam 2})-\frac{1}{8}$ \\
	\hline
	 &  &~$\frac{1}{2}\widehat{b}^T (d^{\diam 2} \diam \widehat{A}((\ind - \delta) \diam d))
	-\frac{1}{2}\widehat{b}^T (d \diam \widehat{A}(\delta\diam d^{\diam 2}))$ \\[-0.5ex]
	$ \aroma1101\ \aroma3301\ \aroma3301\ \eatree2101$ &~$\frac{7}{4}$ &~$
	+\frac{1}{2}\widehat{b}^T (d \diam \widehat{A}((\ind-\delta)\diam d))
	+\frac{1}{2}\widehat{b}^T (\delta \diam d^{\diam 2})
	+\frac{1}{2}(\widehat{b}^T d)^2
	+\widehat{b}^T d
	+\frac{3}{4}
	$ \\
	\hline
	 &  &~$\widehat{b}^T (d \diam (\widehat{A} d)^{\diam 2})+3 \widehat{b}^T (d \diam \widehat{A} (d \diam \widehat{A} ((\ind-\delta)\diam d)))
	$ \\
	 &  &~$-3\widehat{b}^T (d \diam (\widehat{A}((\ind-\delta) \diam d))^{\diam 2})
	 +(2\widehat{b}^T d+\frac{1}{2})\widehat{b}^T (d \diam \widehat{A}d)
	$ \\
	 &  &~$
	+\frac{1}{2}\widehat{b}^T (d^{\diam 2} \diam \widehat{A}((\ind + \delta) \diam d))
	+\frac{1}{2}\widehat{b}^T (d \diam \widehat{A}(\delta \diam d^{\diam 2}))
	-\frac{9}{2}\widehat{b}^T (\delta \diam d^{\diam 3})
	$ \\[-0.5ex]
	$ \aroma3301\ \aroma3301\ \aroma3301\ \eatree2101$ &~$-\frac{7}{2}$ &~$
	+\frac{1}{2}\widehat{b}^T (d \diam \widehat{A}(\delta \diam d))
	+\frac{15}{4}\widehat{b}^T (\delta \diam d^{\diam 2})
	-(\widehat{b}^T d)^2
	-3\widehat{b}^T d
	-\frac{15}{8}
	$ \\
	\hline
	\end{tabularx}
	\renewcommand\thetable{2 (Part 7/7)}
	\caption{Coefficients in exotic aromatic B-series of the operators~$\LL^2\phi/2=\sum e(\gamma)   F(\gamma)(\phi)$ and~$\AA_1\phi=\sum a(\gamma)   F(\gamma)(\phi)$ for consistent Runge-Kutta methods of the form~\eqref{equation:defRK}.}
	\setcellgapes{1pt}
\end{table}

\renewcommand\thetable{2}

\begin{table}[!htb]
	\setcellgapes{3pt}
	\centering
	\begin{tabularx}{\textwidth}{|c|X|}
	\hline
	$  \gamma$ & \multicolumn{1}{c|}{$a^0(\gamma)$} \\
	\hhline{|=|=|}
	$ \aroma3202\ \eatree3201$ &~$b^T d-\widehat{b}^T d$ \\
	\hline
	$\eatree3002$ &~$b^T c-2b^T d+\frac{1}{2}$ \\
	\hline
	$ \eatree2011$&~$\frac{1}{2} b^T d^{\diam 2}-b^T d+\frac{1}{4}$ \\
	\hline
	$ \eatree4204$&~$\widehat{b}^T (d \diam \widehat{A}c) -2\widehat{b}^T (d \diam \widehat{A}d) -(\widehat{b}^T d)^2+2\widehat{b}^T d-\frac{1}{2}$ \\
	\hline
	$ \eatree4205$&~$\widehat{b}^T A((\delta-\ind)\diam d)) -\widehat{b}^T d b^T d+2\widehat{b}^T d-\frac{1}{2}$ \\
	\hline
	$ \eatree4206$&~$b^T (d \diam \widehat{A}((\delta-\ind)\diam d))$ \\
	\hline
	$ \eatree4207$&~$-\frac{1}{2}b^T (\delta \diam d^{\diam 2})+b^T d-\frac{1}{4}$ \\
	\hline
	$ \eatree4208$&~$\widehat{b}^T (\delta \diam c \diam d) -\widehat{b}^T (c \diam d)+2\widehat{b}^T d^{\diam 2}-2\widehat{b}^T (\delta \diam d^{\diam 2})$ \\
	\hline
	$ \eatree3214$&~$\frac{1}{2}\widehat{b}^T (d \diam \widehat{A}d^{\diam 2})-\widehat{b}^T (d \diam \widehat{A}d)-\frac{1}{2}(\widehat{b}^T d)^2+\widehat{b}^T d-\frac{1}{4}$ \\
	\hline
	 &~$~$ \\[-3ex]
	 &~$-\widehat{b}^T (d \diam (\widehat{A} d)^{\diam 2})-3 \widehat{b}^T (d \diam \widehat{A} (d \diam \widehat{A} ((\ind-\delta)\diam d)))$ \\[-5.5ex]
	$ \eatree5406$ &~$+(4-2\widehat{b}^T d)\widehat{b}^T (d \diam \widehat{A}d)
	+3(\widehat{b}^T d)^2
	-4\widehat{b}^T d
	+1$ \\
	\hline
	 &~$\frac{1}{2}\widehat{b}^T (d \diam \widehat{A}((2\cdot\ind-3\delta)\diam d^{\diam 2}))
	+\widehat{b}^T (d \diam \widehat{A}d)$ \\[-4ex]
	$ \eatree5407$ &~$
	+(\widehat{b}^T d-1)(\widehat{b}^T d^{\diam 2}	-\widehat{b}^T (\delta \diam d^{\diam 2}))
	+\frac{1}{2}(\widehat{b}^T d)^2
	-\widehat{b}^T d
	+\frac{1}{4}$ \\
	\hline
	 &~$-2\widehat{b}^T (d^{\diam 2}\diam \widehat{A}d)
	+2\widehat{b}^T (d \diam \widehat{A}d)$ \\[-4ex]
	$ \eatree5408$ &~$
	+(\widehat{b}^T d-2)\widehat{b}^T d^{\diam 2}
	+(3-\widehat{b}^T d)\widehat{b}^T (\delta \diam d^{\diam 2})
	+(\widehat{b}^T d)^2
	-2\widehat{b}^T d
	+\frac{1}{2}$ \\
	\hline
	$ \eatree3215$ &~$\widehat{b}^T (d^{\diam 2}\diam \widehat{A}d)
	-\widehat{b}^T (d \diam \widehat{A}d)
	+\frac{1}{2}\widehat{b}^T d^{\diam 2}
	-\widehat{b}^T (\delta \diam d^{\diam 2})
	-\frac{1}{2}(\widehat{b}^T d)^2
	+\widehat{b}^T d
	-\frac{1}{4}
	$ \\
	\hline
	$ \eatree3216$ &~$\frac{1}{2}\widehat{b}^T (\delta \diam d^{\diam 3})
	-\frac{1}{2}\widehat{b}^T d^{\diam 3}
	+\widehat{b}^T d^{\diam 2}
	-\widehat{b}^T (\delta \diam d^{\diam 2})$ \\
	\hline
	$\eatree5409$ &~$-\widehat{b}^T d^{\diam 2}
	+\widehat{b}^T (\delta \diam d^{\diam 2})$ \\
	\hline
	$\aroma2101\ \eatree3102$ &~$-\widehat{b}^T c
	+2\widehat{b}^T d
	-\frac{1}{2}$ \\
	\hline
	$\aroma1101\ \eatree3102$ &~$-\frac{1}{2}\widehat{b}^T c
	+\widehat{b}^T d
	-\frac{1}{4}$ \\
	\hline
	\end{tabularx}
	\renewcommand\thetable{3 (Part 1/5)}
	\caption{Coefficients in exotic aromatic B-series of the operator~$\AA_1^0\phi=\sum a^0(\gamma)   F(\gamma)(\phi)$ for consistent Runge-Kutta methods of the form~\eqref{equation:defRK}.}
	\setcellgapes{1pt}
\end{table}

\begin{table}[!htb]
	\setcellgapes{3pt}
	\centering
	\begin{tabularx}{\textwidth}{|c|X|}
	\hline
	$  \gamma$ & \multicolumn{1}{c|}{$a^0(\gamma)$} \\
	\hhline{|=|=|}
	$\aroma3301\ \eatree3102$ &~$\frac{1}{2}\widehat{b}^T c
	-\widehat{b}^T d
	+\frac{1}{4}$ \\
	\hline
	$\aroma2101\ \eatree3103$ &~$-b^T (\delta \diam c)
	+2b^T d
	-\frac{1}{2}$ \\
	\hline
	$ \aroma1101\ \eatree3103$ &~$-\frac{1}{2}b^T (\delta \diam d^{\diam 2})
	+b^T d
	-\frac{1}{4}$ \\
	\hline
	$ \aroma3301\ \eatree3103$ &~$b^T (d \diam \widehat{A}((\ind-\delta)\diam d))
	+\frac{1}{2}b^T (\delta \diam d^{\diam 2})
	-b^T d
	+\frac{1}{4}$ \\
	\hline
	 &~$
	\widehat{b}^T (d \diam \widehat{A}((3\cdot\ind-\delta) \diam d))
	-\widehat{b}^T (d \diam \widehat{A}(\delta \diam c))
	~$ \\[-4ex]
	$\aroma2101\ \eatree4303$ &~$
	+\widehat{b}^T (c \diam \widehat{A}((\ind-\delta)\diam d))
	+2(\widehat{b}^T d)^2
	-4\widehat{b}^T d
	+1$ \\
	\hline
	 &~$\frac{1}{2}\widehat{b}^T (d \diam \widehat{A}((3\cdot\ind-\delta) \diam d))
	-\frac{1}{2}\widehat{b}^T (d \diam \widehat{A}(\delta \diam d^{\diam 2}))$ \\[-4ex]
	$\aroma1101\ \eatree4303$ &~$
	+\frac{1}{2}\widehat{b}^T (d^{\diam 2} \diam \widehat{A}((\ind-\delta)\diam d))
	+(\widehat{b}^T d)^2
	-2\widehat{b}^T d
	+\frac{1}{2}
	$ \\
	\hline
	 &~$\widehat{b}^T (d \diam (\widehat{A} d)^{\diam 2})+3 \widehat{b}^T (d \diam \widehat{A} (d \diam \widehat{A} ((\ind-\delta)\diam d)))
	 +(2\widehat{b}^T d-\frac{11}{2})\widehat{b}^T (d \diam \widehat{A}d)$ \\
	 &~$-3\widehat{b}^T (d \diam (\widehat{A}((\ind-\delta) \diam d))^{\diam 2})
	-\frac{1}{2}\widehat{b}^T (d^{\diam 2} \diam \widehat{A}((\ind-\delta)\diam d))$ \\[-4ex]
	$\aroma3301\ \eatree4303$ &~$
	+\frac{1}{2}\widehat{b}^T (d \diam \widehat{A}(\delta \diam d))
	+\frac{1}{2}\widehat{b}^T (d \diam \widehat{A}(\delta \diam d^{\diam 2}))
	-4(\widehat{b}^T d)^2
	+6\widehat{b}^T d
	-\frac{3}{2}
	$ \\
	\hline
	$ \aroma2101\ \eatree2113$ &~$-\frac{1}{2}\widehat{b}^T d^{\diam 2}
	+\widehat{b}^T d
	-\frac{1}{4}$ \\
	\hline
	$ \aroma1101\ \eatree2113$ &~$-\frac{1}{4}\widehat{b}^T d^{\diam 2}
	+\frac{1}{2}\widehat{b}^T d
	-\frac{1}{8}$ \\
	\hline
	$ \aroma3301\ \eatree2113$ &~$\frac{1}{4}\widehat{b}^T d^{\diam 2}
	-\frac{1}{2}\widehat{b}^T d
	+\frac{1}{8}$ \\
	\hline
	$ \aroma2101\ \eatree4304$ &~$-2\widehat{b}^T d^{\diam 2}
	+\frac{5}{2}\widehat{b}^T (\delta \diam d^{\diam 2})
	-\widehat{b}^T d
	+\frac{1}{4}$ \\
	\hline
	$\aroma1101\ \eatree4304$ &~$
	\frac{5}{4}\widehat{b}^T (\delta \diam d^{\diam 2})
	-\widehat{b}^T d^{\diam 2}
	-\frac{1}{2}\widehat{b}^T d
	+\frac{1}{8}
	$ \\
	\hline
	 &~$
	(\frac{5}{2}-\widehat{b}^T d)\widehat{b}^T d^{\diam 2}
	+(\widehat{b}^T d-\frac{13}{4})\widehat{b}^T (\delta \diam d^{\diam 2})
	~$ \\[-2ex]
	$\aroma3301\ \eatree4304$ &~$
	+\widehat{b}^T (d^{\diam 2}\diam \widehat{A}d)
	-\widehat{b}^T (d \diam \widehat{A}d)
	-\frac{1}{2}(\widehat{b}^T d)^2
	+\frac{3}{2}\widehat{b}^T d
	-\frac{3}{8}
	$ \\
	\hline
	$ \aroma3201\ \eatree3202$ &~$-\widehat{b}^T (d \diam \widehat{A}c)
	+2\widehat{b}^T (d \diam \widehat{A}d)
	+(\widehat{b}^T d)^2
	-2\widehat{b}^T d
	+\frac{1}{2}$ \\
	\hline
	$ \aroma3202\ \eatree3202$ &~$\widehat{b}^T (\delta \diam A((\ind-\delta)\diam d))
	+\widehat{b}^T d b^T d
	-2\widehat{b}^T d
	+\frac{1}{2}$ \\
	\hline
	$ \aroma2201\ \eatree3202$ &~$-\frac{1}{2}\widehat{b}^T (d \diam \widehat{A}d^{\diam 2})
	+\widehat{b}^T (d \diam \widehat{A}d)
	+\frac{1}{2}(\widehat{b}^T d)^2
	-\widehat{b}^T d
	+\frac{1}{4}$ \\
	\hline
	$ \aroma2202\ \eatree3202$ &~$-\widehat{b}^T (d^{\diam 2}\diam \widehat{A}d)
	+\widehat{b}^T (d \diam \widehat{A}d)
	+\frac{1}{2}\widehat{b}^T (\delta \diam d^{\diam 2})
	+\frac{1}{2}(\widehat{b}^T d)^2
	-\widehat{b}^T d
	+\frac{1}{4}$ \\
	\hline
	\end{tabularx}
	\renewcommand\thetable{3 (Part 2/5)}
	\caption{Coefficients in exotic aromatic B-series of the operator~$\AA_1^0\phi=\sum a^0(\gamma)   F(\gamma)(\phi)$ for consistent Runge-Kutta methods of the form~\eqref{equation:defRK}.}
	\setcellgapes{1pt}
\end{table}

\begin{table}[!htb]
	\setcellgapes{3pt}
	\centering
	\begin{tabularx}{\textwidth}{|c|X|}
	\hline
	$  \gamma$ & \multicolumn{1}{c|}{$a^0(\gamma)$} \\
	\hhline{|=|=|}
	 &~$\widehat{b}^T (d \diam (\widehat{A} d)^{\diam 2})+3 \widehat{b}^T (d \diam \widehat{A} (d \diam \widehat{A} ((\ind-\delta)\diam d)))
	+2\widehat{b}^T (d^{\diam 2}\diam \widehat{A}d)$ \\[-2ex]
	$ \aroma4401\ \eatree3202$ &~$
	+(2\widehat{b}^T d-6) \widehat{b}^T (d \diam \widehat{A}d)
	-\widehat{b}^T (\delta \diam d^{\diam 2})
	-4(\widehat{b}^T d)^2
	+6\widehat{b}^T d
	-\frac{3}{2}$ \\
	\hline
	 &~$\frac{1}{2}\widehat{b}^T (d \diam \widehat{A}((3\delta-2\cdot\ind) \diam d^{\diam 2}))
	-\widehat{b}^T (d \diam \widehat{A}d)$ \\[-2ex]
	$ \aroma4402\ \eatree3202$ &~$
	+(1-\widehat{b}^T d)(\widehat{b}^T d^{\diam 2}-\widehat{b}^T (\delta \diam d^{\diam 2}))
	-\frac{1}{2}(\widehat{b}^T d)^2
	+\widehat{b}^T d
	-\frac{1}{4}$ \\
	\hline
	$\aroma2101\ \aroma2101\ \eatree3202$ &~$\widehat{b}^T (\delta \diam c)
	-2\widehat{b}^T d
	+\frac{1}{2}$ \\
	\hline
	$ \aroma2101\ \aroma1101\ \eatree3202$ &~$
	\frac{1}{2}\widehat{b}^T (\delta \diam d^{\diam 2})
	+\frac{1}{2}\widehat{b}^T (\delta \diam c)
	-2\widehat{b}^T d
	+\frac{1}{2}$ \\
	\hline
	 &~$\widehat{b}^T (c \diam \widehat{A}((\delta-\ind) \diam d))
	 +\widehat{b}^T (d \diam \widehat{A}((\delta -3\cdot\ind) \diam d))$ \\[-2ex]
	$ \aroma2101\ \aroma3301\ \eatree3202$ &~$
	 +\widehat{b}^T (d \diam \widehat{A}(\delta \diam c))
	 -\frac{1}{2}\widehat{b}^T (\delta \diam d^{\diam 2})
	-\frac{1}{2}\widehat{b}^T (\delta \diam c)
	-2(\widehat{b}^T d)^2
	+6\widehat{b}^T d
	-\frac{3}{2}$ \\
	\hline
	$ \aroma1101\ \aroma1101\ \eatree3202$ &~$\frac{1}{4}\widehat{b}^T (\delta \diam d^{\diam 2})
	-\frac{1}{2}\widehat{b}^T d
	+\frac{1}{8}$ \\
	\hline
	 &~$\frac{1}{2}\widehat{b}^T (d \diam \widehat{A}((\delta-3\cdot\ind) \diam d))+\frac{1}{2}\widehat{b}^T (d \diam \widehat{A}(\delta \diam d^{\diam 2}))$ \\[-2ex]
	$ \aroma1101\ \aroma3301\ \eatree3202$ &~$+\frac{1}{2}\widehat{b}^T (d^{\diam 2} \diam \widehat{A}((\delta-\ind) \diam d))-\frac{1}{2}\widehat{b}^T (\delta \diam d^{\diam 2})-(\widehat{b}^T d)^2+3\widehat{b}^T d-\frac{3}{4}$ \\
	\hline
	 &~$3\widehat{b}^T (d \diam (\widehat{A}((\ind-\delta) \diam d))^{\diam 2})
	+(\frac{13}{2}-2\widehat{b}^T d)\widehat{b}^T (d \diam \widehat{A}d)~$ \\
	 &~$
	 -\widehat{b}^T (d \diam (\widehat{A} d)^{\diam 2})-3 \widehat{b}^T (d \diam \widehat{A} (d \diam \widehat{A} ((\ind-\delta)\diam d)))
	~$ \\
	 &~$
	-\frac{1}{2}\widehat{b}^T (d^{\diam 2} \diam \widehat{A}((\delta+\ind) \diam d))
	-\frac{1}{2}\widehat{b}^T (d \diam \widehat{A}(\delta \diam d^{\diam 2}))
	+\frac{9}{2}\widehat{b}^T (\delta \diam d^{\diam 3})$ \\[-2ex]
	$ \aroma3301\ \aroma3301\ \eatree3202$ &~$
	-\frac{1}{2}\widehat{b}^T (d \diam \widehat{A}(\delta \diam d))
	 -\frac{15}{4}\widehat{b}^T (\delta \diam d^{\diam 2})
	 +\frac{9}{2}(\widehat{b}^T d)^2
	 -\frac{15}{2}\widehat{b}^T d
	 +\frac{15}{8}
	$ \\
	\hline
	$\aroma3101\ \eatree2101$ &~$-b^T c
	+2\widehat{b}^T d
	-\frac{1}{2}$ \\
	\hline
	$ \aroma2103\ \eatree2101$ &~$-\frac{1}{2}b^T d^{\diam 2}
	+\widehat{b}^T d
	-\frac{1}{4}$ \\
	\hline
	$ \aroma2104\ \eatree2101$ &~$-b^T d
	+\widehat{b}^T d$ \\
	\hline
	$ \aroma4301\ \eatree2101$ &~$-\widehat{b}^T (d \diam \widehat{A}c)
	+2\widehat{b}^T (d \diam \widehat{A}d)
	+(\widehat{b}^T d)^2
	-2\widehat{b}^T d
	+\frac{1}{2}$ \\
	\hline
	$ \aroma4302\ \eatree2101$ &~$b^T (d \diam \widehat{A}((\ind-\delta)\diam d))
	+\widehat{b}^T A((\ind-\delta)\diam d)
	+(\widehat{b}^T d+2)b^T d
	-4\widehat{b}^T d
	+\frac{1}{2}$ \\
	\hline
	$ \aroma4303\ \eatree2101$ &~$\widehat{b}^T (c \diam d)
	-\widehat{b}^T (\delta\diam c \diam d)
	-2\widehat{b}^T d^{\diam 2}
	+2\widehat{b}^T (\delta \diam d^{\diam 2})$ \\
	\hline
	$ \aroma4304\ \eatree2101$ &~$\frac{1}{2}b^T (\delta \diam d^{\diam 2})
	-\widehat{b}^T d
	+\frac{1}{4}$ \\
	\hline
	$ \aroma3302\ \eatree2101$ &~$-\frac{1}{2}\widehat{b}^T (d \diam \widehat{A}d^{\diam 2})
	+\widehat{b}^T (d \diam \widehat{A}d)
	+\frac{1}{2}(\widehat{b}^T d)^2
	-\widehat{b}^T d
	+\frac{1}{4}$ \\
	\hline
	$ \aroma3303\ \eatree2101$ &~$\frac{1}{2}\widehat{b}^T d^{\diam 3}
	-\frac{1}{2}\widehat{b}^T (\delta \diam d^{\diam 3})
	-\widehat{b}^T d^{\diam 2}
	+\widehat{b}^T (\delta \diam d^{\diam 2})$ \\
	\hline
	\end{tabularx}
	\renewcommand\thetable{3 (Part 3/5)}
	\caption{Coefficients in exotic aromatic B-series of the operator~$\AA_1^0\phi=\sum a^0(\gamma)   F(\gamma)(\phi)$ for consistent Runge-Kutta methods of the form~\eqref{equation:defRK}.}
	\setcellgapes{1pt}
\end{table}

\begin{table}[!htb]
	\setcellgapes{3pt}
	\centering
	\begin{tabularx}{\textwidth}{|c|X|}
	\hline
	$  \gamma$ & \multicolumn{1}{c|}{$a^0(\gamma)$} \\
	\hhline{|=|=|}
	$ \aroma3304\ \eatree2101$ &~$-\widehat{b}^T (d^{\diam 2} \diam \widehat{A}d)
	+\widehat{b}^T (d \diam \widehat{A}d)
	-\frac{1}{2}\widehat{b}^T d^{\diam 2}
	+\widehat{b}^T (\delta \diam d^{\diam 2})
	+\frac{1}{2}(\widehat{b}^T d)^2
	-\widehat{b}^T d
	+\frac{1}{4}$ \\
	\hline
	 &~$\widehat{b}^T (d \diam (\widehat{A} d)^{\diam 2})+3 \widehat{b}^T (d \diam \widehat{A} (d \diam \widehat{A} ((\ind-\delta)\diam d)))
	~$ \\[-4ex]
	$ \aroma5501\ \eatree2101$ &~$
	+(2\widehat{b}^T d-4)\widehat{b}^T (d \diam \widehat{A}d)
	-3(\widehat{b}^T d)^2
	+4\widehat{b}^T d
	-1
	$ \\
	\hline
	 &~$\frac{1}{2}\widehat{b}^T (d \diam \widehat{A}((3\delta-2\cdot\ind)\diam d^{\diam 2}))
	 +2\widehat{b}^T (d^{\diam 2}\diam \widehat{A}d)
	 -3\widehat{b}^T (d \diam \widehat{A}d)$ \\[-2ex]
	$ \aroma5502\ \eatree2101$ &~$
	+(3-2\widehat{b}^T d)\widehat{b}^T d^{\diam 2}
	+(2\widehat{b}^T d-4)\widehat{b}^T (\delta \diam d^{\diam 2})
	-\frac{3}{2}(\widehat{b}^T d)^2
	+3\widehat{b}^T d
	-\frac{3}{4}
	$ \\
	\hline
	$ \aroma5503\ \eatree2101$ &~$\widehat{b}^T d^{\diam 2}
	-\widehat{b}^T (\delta \diam d^{\diam 2})$ \\
	\hline
	$\aroma2101\ \aroma3201\ \eatree2101$ &~$\widehat{b}^T c
	-2\widehat{b}^T d
	+\frac{1}{2}$ \\
	\hline
	$ \aroma1101\ \aroma3201\ \eatree2101$ &~$
	\frac{1}{2}\widehat{b}^T c
	-\widehat{b}^T d
	+\frac{1}{4}
	$ \\
	\hline
	$ \aroma3301\ \aroma3201\ \eatree2101$ &~$
	\widehat{b}^T (d \diam \widehat{A}c)
	-2\widehat{b}^T (d \diam \widehat{A}d)
	-\frac{1}{2}\widehat{b}^T c
	-(\widehat{b}^T d)^2
	+3\widehat{b}^T d
	-\frac{3}{4}
	$ \\
	\hline
	$\aroma2101\ \aroma3202\ \eatree2101$ &~$b^T (\delta \diam c)
	-2\widehat{b}^T d
	+\frac{1}{2}$ \\
	\hline
	$\aroma1101\ \aroma3202\ \eatree2101$ &~$\frac{1}{2} b^T (\delta \diam d^{\diam 2})
	-\widehat{b}^T d
	+\frac{1}{4}$ \\
	\hline
	 &~$
	b^T (d \diam \widehat{A}((\delta-\ind) \diam d))
	+\widehat{b}^T (\delta \diam A((\delta-\ind) \diam d))
	~$ \\[-0.5ex]
	$ \aroma3301\ \aroma3202\ \eatree2101$ &~$
	-\frac{1}{2} b^T (\delta \diam d^{\diam 2})
	-(\widehat{b}^T d+1)b^T d
	+4\widehat{b}^T d
	-\frac{3}{4}
	$ \\
	\hline
	$ \aroma2101\ \aroma2201\ \eatree2101$ &~$\frac{1}{2}\widehat{b}^T d^{\diam 2}
	-\widehat{b}^T d
	+\frac{1}{4}$ \\
	\hline
	$ \aroma1101\ \aroma2201\ \eatree2101$ &~$\frac{1}{4}\widehat{b}^T d^{\diam 2}
	-\frac{1}{2}\widehat{b}^T d
	+\frac{1}{8}$ \\
	\hline
	$ \aroma3301\ \aroma2201\ \eatree2101$ &~$\frac{1}{2}\widehat{b}^T (d \diam \widehat{A}d^{\diam 2})
	-\widehat{b}^T (d \diam \widehat{A}d)	
	-\frac{1}{4}\widehat{b}^T d^{\diam 2}
	-\frac{1}{2}(\widehat{b}^T d)^2
	+\frac{3}{2}\widehat{b}^T d
	-\frac{3}{8}$ \\
	\hline
	$ \aroma3301\ \aroma2202\ \eatree2101$ &~$\widehat{b}^T (d^{\diam 2}\diam \widehat{A}d)
	-\widehat{b}^T (d \diam \widehat{A}d)
	-\frac{1}{2}\widehat{b}^T (\delta \diam d^{\diam 2})
	-\frac{1}{2}(\widehat{b}^T d)^2
	+\widehat{b}^T d
	-\frac{1}{4}$ \\
	\hline
	 &~$\widehat{b}^T (c \diam \widehat{A}((\delta-\ind) \diam d))+\widehat{b}^T (d \diam \widehat{A}(\delta \diam c))$ \\[-2ex]
	$ \aroma2101\ \aroma4401\ \eatree2101$ &~$
	+\widehat{b}^T (d \diam \widehat{A}((\delta -3\cdot\ind) \diam d))
	-2(\widehat{b}^T d)^2
	+4\widehat{b}^T d
	-1
	$ \\
	\hline
	 &~$\frac{1}{2}\widehat{b}^T (d^{\diam 2} \diam \widehat{A}((\delta-\ind) \diam d))
	+\frac{1}{2}\widehat{b}^T (d \diam \widehat{A}(\delta\diam d^{\diam 2}))
	$ \\[-2ex]
	$ \aroma1101\ \aroma4401\ \eatree2101$ &~$+\frac{1}{2}\widehat{b}^T (d \diam \widehat{A}((\delta -3\cdot\ind) \diam d))
	-(\widehat{b}^T d)^2
	+2\widehat{b}^T d
	-\frac{1}{2}
	$ \\
	\hline
	 &~$3\widehat{b}^T (d \diam (\widehat{A}((\ind-\delta) \diam d))^{\diam 2})
	-\frac{1}{2}\widehat{b}^T (d^{\diam 2} \diam \widehat{A}((3\cdot\ind +\delta) \diam d))
	$ \\
	 &~$-2\widehat{b}^T (d \diam (\widehat{A} d)^{\diam 2})-6 \widehat{b}^T (d \diam \widehat{A} (d \diam \widehat{A} ((\ind-\delta)\diam d)))
	+(\frac{23}{2}-4\widehat{b}^T d)\widehat{b}^T (d \diam \widehat{A}d)
	$ \\[-2ex]
	$ \aroma3301\ \aroma4401\ \eatree2101$ &~$
	-\frac{1}{2}\widehat{b}^T (d \diam \widehat{A}(\delta \diam d))
	-\frac{1}{2}\widehat{b}^T (d \diam \widehat{A}(\delta\diam d^{\diam 2}))
	+\widehat{b}^T (\delta \diam d^{\diam 2})
	+8(\widehat{b}^T d)^2
	-12\widehat{b}^T d
	+3
	$ \\
	\hline
	$ \aroma2101\ \aroma4402\ \eatree2101$ &~$
	2\widehat{b}^T d^{\diam 2}
	-\frac{5}{2}\widehat{b}^T (\delta \diam d^{\diam 2})
	+\widehat{b}^T d
	-\frac{1}{4}$ \\
	\hline
	$ \aroma1101\ \aroma4402\ \eatree2101$ &~$
	\widehat{b}^T d^{\diam 2}
	-\frac{5}{4}\widehat{b}^T (\delta \diam d^{\diam 2})
	+\frac{1}{2}\widehat{b}^T d
	-\frac{1}{8}$ \\
	\hline
	\end{tabularx}
	\renewcommand\thetable{3 (Part 4/5)}
	\caption{Coefficients in exotic aromatic B-series of the operator~$\AA_1^0\phi=\sum a^0(\gamma)   F(\gamma)(\phi)$ for consistent Runge-Kutta methods of the form~\eqref{equation:defRK}.}
	\setcellgapes{1pt}
\end{table}

\begin{table}[!htb]
	\setcellgapes{3pt}
	\centering
	\begin{tabularx}{\textwidth}{|c|X|}
	\hline
	$  \gamma$ & \multicolumn{1}{c|}{$a^0(\gamma)$} \\
	\hhline{|=|=|}
	 &~$\frac{1}{2}\widehat{b}^T (d \diam \widehat{A}((2\cdot\ind -3 \delta) \diam d^{\diam 2}))
	-\widehat{b}^T (d^{\diam 2}\diam \widehat{A}d)
	+2\widehat{b}^T (d \diam \widehat{A}d)$ \\[-0.5ex]
	$ \aroma3301\ \aroma4402\ \eatree2101$ &~$
	+(2\widehat{b}^T d-\frac{7}{2})\widehat{b}^T d^{\diam 2}
	+(\frac{17}{4}-2\widehat{b}^T d)\widehat{b}^T (\delta \diam d^{\diam 2})
	+(\widehat{b}^T d)^2
	-\frac{5}{2}\widehat{b}^T d
	+\frac{5}{8}$ \\
	\hline
	$\aroma2101\ \aroma2101\ \aroma3301\ \eatree2101$ &~$-\widehat{b}^T (\delta \diam c)
	+2\widehat{b}^T d
	-\frac{1}{2}$ \\
	\hline
	$ \aroma2101\ \aroma1101\ \aroma3301\ \eatree2101$ &~$
	-\frac{1}{2}\widehat{b}^T (\delta \diam c)
	-\frac{1}{2}\widehat{b}^T (\delta \diam d^{\diam 2})
	+2\widehat{b}^T d
	-\frac{1}{2}
	$ \\
	\hline
	 &~$
	\widehat{b}^T (c \diam \widehat{A}((\ind-\delta)\diam d))
	+\widehat{b}^T (d \diam \widehat{A}((3\cdot\ind-\delta) \diam d))
	$ \\[-0.5ex]
	$ \aroma2101\ \aroma3301\ \aroma3301\ \eatree2101$ &~$
	-\widehat{b}^T (d \diam \widehat{A}(\delta \diam c))
	+\frac{1}{2}\widehat{b}^T (\delta \diam c)
	+\frac{1}{2}\widehat{b}^T (\delta \diam d^{\diam 2})
	+2(\widehat{b}^T d)^2
	-6\widehat{b}^T d
	+\frac{3}{2}
	$ \\
	\hline
	$ \aroma1101\ \aroma1101\ \aroma3301\ \eatree2101$ &~$
	-\frac{1}{4}\widehat{b}^T (\delta \diam d^{\diam 2})
	+\frac{1}{2}\widehat{b}^T d
	-\frac{1}{8}
	$ \\
	\hline
	 &~$\frac{1}{2}\widehat{b}^T (d^{\diam 2} \diam \widehat{A}((\ind - \delta) \diam d))
	-\frac{1}{2}\widehat{b}^T (d \diam \widehat{A}(\delta\diam d^{\diam 2}))
	$ \\[-0.5ex]
	$ \aroma1101\ \aroma3301\ \aroma3301\ \eatree2101$ &~$
	+\frac{1}{2}\widehat{b}^T (d \diam \widehat{A}((3\cdot\ind-\delta) \diam d))
	+\frac{1}{2}\widehat{b}^T (\delta \diam d^{\diam 2})
	+(\widehat{b}^T d)^2
	-3\widehat{b}^T d
	+\frac{3}{4}
	$ \\
	\hline
	 &~$\widehat{b}^T (d \diam (\widehat{A} d)^{\diam 2})+3 \widehat{b}^T (d \diam \widehat{A} (d \diam \widehat{A} ((\ind-\delta)\diam d)))
	$ \\
	 &~$-3\widehat{b}^T (d \diam (\widehat{A}((\ind-\delta) \diam d))^{\diam 2})
	+\frac{1}{2}\widehat{b}^T (d \diam \widehat{A}(\delta \diam d))
	$ \\
	 &~$
	+(2\widehat{b}^T d-\frac{13}{2})\widehat{b}^T (d \diam \widehat{A}d)
	+\frac{1}{2}\widehat{b}^T (d^{\diam 2} \diam \widehat{A}((\ind + \delta) \diam d))
	$ \\[-0.5ex]
	$ \aroma3301\ \aroma3301\ \aroma3301\ \eatree2101$ &~$
	+\frac{1}{2}\widehat{b}^T (d \diam \widehat{A}(\delta \diam d^{\diam 2}))	
	-\frac{9}{2}\widehat{b}^T (\delta \diam d^{\diam 3})
	+\frac{15}{4}\widehat{b}^T (\delta \diam d^{\diam 2})
	-\frac{9}{2}(\widehat{b}^T d)^2
	+\frac{15}{2}\widehat{b}^T d
	-\frac{15}{8}
	$ \\
	\hline
	\end{tabularx}
	\renewcommand\thetable{3 (Part 5/5)}
	\caption{Coefficients in exotic aromatic B-series of the operator~$\AA_1^0\phi=\sum a^0(\gamma)   F(\gamma)(\phi)$ for consistent Runge-Kutta methods of the form~\eqref{equation:defRK}.}
	\setcellgapes{1pt}
\end{table}

\end{appendices}

\end{document}